\documentclass[11pt]{amsart}
\usepackage{amssymb,amscd,amsxtra,calc}
\usepackage{cmmib57}
\usepackage{graphicx}
\usepackage{amssymb,amsmath,amsfonts,mathrsfs,enumerate,xcolor}
\usepackage{amsthm}

\usepackage{soul}
\usepackage[citecolor=Navy]{hyperref}
\usepackage[ruled,vlined]{algorithm2e}

\usepackage{lipsum}
\usepackage{amsfonts}
\usepackage{graphicx}
\usepackage{epstopdf}
\usepackage{algorithmic}
\usepackage{float}
\restylefloat{table}
\usepackage{placeins}
\usepackage{array}
\ifpdf
  \DeclareGraphicsExtensions{.eps,.pdf,.png,.jpg}
\else
  \DeclareGraphicsExtensions{.eps}
\fi

\usepackage{subcaption}

\theoremstyle{plain}
    \newtheorem{thm}{Theorem}[section]

    \newtheorem{corollary}[thm]{Corollary}
    \newtheorem{lemma}[thm]{Lemma}

    \newtheorem{theorem}[thm]{Theorem}

\theoremstyle{definition}

    \newtheorem{remark}[thm]{Remark}
\theoremstyle{remark}

\newcommand{\authorfootnotes}{\renewcommand\thefootnote{\@fnsymbol\c@footnote}}

\usepackage{caption} 
 \title[Fast $\&$ simple modification of Newton's method avoiding saddle points]{A fast and simple modification of Newton's method helping to avoid saddle points}
  \author[T. T. Truong {\it et al.}]{Tuyen Trung Truong, Tat Dat To, Hang-Tuan Nguyen, Thu Hang Nguyen, Hoang  Phuong Nguyen, Maged Helmy}

    \date{\today}

\begin{document}
\maketitle

{\bf Affiliations and email addresses:}  T.T.T.: Department of Mathematics, University of Oslo, Blindern 0851 Oslo, Norway, tuyentt@math.uio.no; {\bf ORCID:} 0000-0001-9103-0923.

T.D.T.:  Ecole Nationale de l'Aviation Civile, 31400 Toulouse, France, tatdat.to@gmail.com; Current workplace: Institut de math\'ematique de Jussieu-Paris Rive Gauche, Sorbone University, France, tat-dat.to@imj-prg.fr 

H.-T. N.: Axon AI Research, Seattle, Washington, USA, hnguyen@axon.com; 

T. H. N.: Torus Actions SAS, 31400 Toulouse, France, hangntt@torus-actions.fr; 

H. P. N.: Torus Actions SAS, 31400 Toulouse, France, nhphuong@torus-actions.fr; 

and M. H.: Department of Informatics, University of Oslo, Blindern 0851 Oslo, Norway $\&$ ODI Medical AS, Vinderen 0370 Oslo, Norway, magedaa@ifi.uio.no, office@odimedical.com.

\begin{abstract}
We propose in this paper New Q-Newton's method. The update rule is very simple conceptually, for example $x_{n+1}=x_n-w_n$  where $w_n=pr_{A_n,+}(v_n)-pr_{A_n,-}(v_n)$, with  $A_n=\nabla ^2f(x_n)+\delta _n||\nabla f(x_n)||^2.Id$ and $v_n=A_n^{-1}.\nabla f(x_n)$. Here $\delta _n$ is an appropriate real number so that $A_n$ is invertible, and $pr_{A_n,\pm}$ are projections to the vector subspaces generated by eigenvectors of positive (correspondingly negative) eigenvalues of $A_n$. 

The main result of this paper roughly says that if $f$ is $C^3$ (can be unbounded from below) and a sequence $\{x_n\}$, constructed by the New Q-Newton's method from a random initial point $x_0$, {\bf converges}, then the limit point is a critical point and is not a saddle point, and the convergence rate is the same as that of Newton's method. The first author has recently been successful incorporating Backtracking line search to New Q-Newton's method, thus resolving the convergence guarantee issue observed for some (non-smooth) cost functions. An application to quickly finding zeros of a univariate meromorphic function will be discussed. Various experiments are performed, against well known algorithms such as BFGS and Adaptive Cubic Regularization are presented. 
\end{abstract}



{\bf Declarations:} Funding: T.T.T. is supported by Young Research Talents grant number 300814 from Research Council of Norway. Conflicts of interests/Competing interests: not applicable. Available of data and material: not applicable.  
Code availability: available on GitHub \cite{phuongGitHub}.


\section{Introduction}

\label{Subsection1}

An important question one faces in research and real life applications is that of finding minima of some objective cost functions. In realistic applications the optimisation problem is so large scale that no one can hope to find closed form solutions. Indeed, optimisation problems associated to Deep Neural Networks (DNN) easily have millions of variables. We note that finding global optima is NP-hard. Moreover, saddle points are dominant in higher dimensions, see Subsection \ref{Subsection4}.  Therefore, one is more than happy with iterative methods which can guarantee convergence to local minima. 

To date, only modifications of a classical iterative method by Armijo (also called Backtracking GD)  are theoretically proven to, when the cost function is Morse or satisfies the Losjasiewicz gradient inequality, assure convergence to local minima. More details are presented in Section 2. Experiments on DNN with CIFAR10 and CIFAR100 datasets \cite{truong-nguyen1, truong-nguyen2} (see \cite{vaswani-etal} for a more recent similar implementation) show that Backtracking Gradient Descent is also implementable in huge scale optimisation problems in Deep Neural Networks, with better accuracy than the popular used algorithms (such as Stochastic Gradient Descent, Adam, Adadelta, RMSProp, NAG,  Momentum and so on) and without worry about manual fine tuning of learning rates, while needing only a comparable computing time. See Section \ref{SubsectionLargeScale} for some experimental results, reported in \cite{truong-nguyen1, truong-nguyen2}. Hence, it can be said that Backtracking GD is theoretically the best iterative method, and for GD methods in DNN it is also practically the best. 

On the other hand, Newton's method is known to usually converge faster than GD (more precisely, in terms of the number of iterations needed), if it {\bf actually converges}. However, it is known that Newton's method can diverge even if the cost function has compact sublevels and can converge to saddle points or local maxima. Newton's method and modifications are a very popular topic: it seems that at least one paper about this topic appears every month. Therefore, it is desirable if one can modify Newton's method in such a way so that if it {\bf converges}, then its rate of convergence is the same as that of Newton's method and it avoids saddle points. Also, to be able to apply this method practically, it is desirable that the modification is simple. 

We recall that if a sequence $x_n$ converges to a point $x_{\infty}$, and $||x_{n+1}-x_{\infty}||=O(||x_n-x_{\infty}||^{\epsilon})$ for some positive constant $\epsilon >0$, here we use the big-O notation, then $\epsilon$ is called the rate of convergence for the sequence $x_n$. If $\epsilon =1$, then we also say that the rate of convergence is linear; while if $\epsilon =2$, then we say that the rate of convergence is quadratic. 

The main result of this paper is to propose such a modification, called New Q-Newton's method, see Subsection \ref{Subsection5}.  The main result we obtain is the following. 
\begin{theorem} Let $f:\mathbb{R}^m\rightarrow \mathbb{R}$ be a $C^3$ function. Let $\{x_n\}$ be a sequence constructed by the New Q-Newton's method. Assume that $\{x_n\}$ converges to $x_{\infty}$. Then

1) $\nabla f(x_{\infty})=0$, that is $x_{\infty}$ is a critical point of $f$.

2) If the hyperparameters $\delta _0,\ldots ,\delta _m$ are chosen {\bf randomly}, there is a set $\mathcal{A}\subset \mathbb{R}^m$ of Lebesgue measure $0$, so that if $x_0\notin \mathcal{A}$, then $x_{\infty}$ cannot be  a saddle point of $f$. 

3) If $x_0\notin \mathcal{A}$ (as defined in part 2) and $\nabla ^2f(x_{\infty})$ is invertible, then $x_{\infty}$ is a local minimum and the rate of convergence is quadratic. 

4) More generally, if $\nabla ^2f(x_{\infty})$ is invertible (but no assumption on the randomness of $x_0$), then the rate of convergence is at least linear. 

5) If $x_{\infty}'$ is a non-degenerate local minimum of $f$, then for initial points $x_0'$ close enough to $x_{\infty}'$, the sequence $\{x_n'\}$  constructed by New Q-Newton's method will converge to $x_{\infty}'$. 
\label{TheoremMain}\end{theorem} 
Part 4) of Theorem \ref{TheoremMain} shows that generally the rate of convergence of the New Q-Newton's method is not worse than that of Gradient Descent method. Part 5) shows that we can find all non-degenerate local minima of $f$ by New Q-Newton's method. However, as will be seen later, it is not known whether New Q-Newton's method (and all other modifications of Newton's method) can guarantee convergence.  

As a consequence, we obtain the following interesting result. 
\begin{corollary}
Let $f$ be a $C^3$ function, which is Morse, that is all its critical points are non-degenerate (i.e. $\nabla ^2f$ is invertible at all critical points of $f$). Let $x_0$ be a random initial point, and let $\{x_n\}$ be a sequence constructed by the New Q-Newton's method, where the hyperparameters $\delta _0,\ldots ,\delta _m$ are randomly chosen. If $x_n$ converges to $x_{\infty}$, then $x_{\infty}$ is a local minimum and the rate of convergence is quadratic. 
\label{CorollaryMain}\end{corollary}
We note, see Subsection \ref{Subsection6}, that even for Morse functions in dimension 1 the usual Newton's method can converge to a local maximum. 

Essential definitions and related works are detailed in Subsections \ref{Subsection2}, \ref{Subsection3} and \ref{Subsection4}. 

{\bf New contribution in this paper:} We propose a new way to modify the Hessian $\nabla ^2f(x_n)$ in Newton's method, by adding a matrix $\delta _n Id$, so that  the new matrix $\nabla ^2f(x_n)+\delta _nId$ is {\bf invertible}. (Note that this is different from previous work on quasi-Newton's methods where it is required that the new matrix is {\bf positive definite}, see Subsection \ref{Subsection4} for more details.) We then, in contrast to Newton's method, do not use the update $x_{n+1}=x_n-(\nabla ^2f(x_n)+\delta _nId)^{-1}\nabla f(x_n)$, but instead $x_{n+1}=x_n-w_n$ where $w_n$ is the reflection of $(\nabla ^2f(x_n)+\delta _nId)^{-1}\nabla f(x_n)$ via the vector subspace generated by eigenvectors of negative eigenvalues of $\nabla ^2f(x_n)+\delta _nId$. We also arrange so that $|\delta _n|$ is bounded by $||\nabla f(x_n)||$ when $||\nabla f(x_n)||$ is small. Note that our new algorithm uses crucially the fact that the Hessian $\nabla ^2f$ is symmetric. While our new method is {\bf very simple} to implement, its theoretical guarantee (avoidance of saddle points) is proven under quite general assumptions. This is different from the many other modifications of Newton's method in the literature, see Subsection \ref{Subsection3} for detailed comparisons. In particular, there is an active research direction \cite{goldfarb, gould-etal} of using (an approximation of) the eigenvector corresponding to the most negative eigenvalue of the Hessian coupled with Backtracking line search to avoid saddle points, but that method is more complicated to describe, and practically can be slower than our method (in particularly in small and medium dimensions) and they do not rigorously treat the case where the Hessian is not invertible. Besides, there is no result concerning cost functions satisfying Losjasiewicz gradient inequality, the latter being a large class of relevance to Deep Learning. See Subsection \ref{Subsection3} for details. 

{\bf The role of the randomness of the hyperparameters $\delta _0,\ldots ,\delta _m$:} Here we explain where the randomness of the hyperparameters $\delta _0,\ldots ,\delta _m$ is needed. As stated in Theorem \ref{TheoremMain}, this is used only in part 2  to assure that when we start with a random initial point, then if the sequence constructed by New Q-Newton's method converges, the limit cannot be a saddle point. More precisely, to assure this, there are two steps. Step 1: show the existence of local Stable - Center manifolds around saddle points. For this step, the randomness of $\delta _0,\ldots ,\delta _m$ is {\bf not} needed. Step 2: show that the preimage of the local Stable - Center manifolds by the associated dynamical system has zero Lebesgue measure. This is where, in the proof, the randomness of  $\delta _0,\ldots ,\delta _m$ is exactly needed.   

Here is a table summarising New Q-Newton's method, where $pr_{A_k,\pm}$ are linear projections to the direct sum of eigenspaces of positive eigenvalues and of negative eigenvalues of $A_k$, see Subsection \ref{Subsection5}. A variant, where we choose the $\delta _i$s' randomly at each step, which seems to behave better in particular in the stochastic setting, is detailed in Section \ref{SectionImplementation}.  

\medskip

{\color{blue}
 \begin{algorithm}[H]
\SetAlgoLined
\KwResult{Find a critical point of $f:\mathbb{R}^m\rightarrow \mathbb{R}$}
Given: $\Delta=\{\delta_0,\delta_1,\ldots, \delta_{m}\}$\ (chosen {\bf randomly}) and $\alpha >0$;
Initialization: $x_0\in \mathbb{R}^m$\;
 \For{$k=0,1,2\ldots$}{ 
    $j=0$\\
    \If{$\|\nabla f(x_k)\|\neq 0$}{
   \While{$\det(\nabla^2f(x_k)+\delta_j \|\nabla f(x_k)\|^{1+\alpha}Id)=0$}{$j=j+1$}}

$A_k:=\nabla^2f(x_k)+\delta_j \|\nabla f(x_k)\|^{1+\alpha}Id$\\
$v_k:=A_k^{-1}\nabla f(x_k)=pr_{A_k,+}(v_k)+pr_{A_k,-}(v_k)$\\
$w_k:=pr_{A_k,+}(v_k)-pr_{A_k,-}(v_k)$\\
$x_{k+1}:=x_k-w_k$
   }
  \caption{New Q-Newton's method} \label{table:alg}
\end{algorithm}
}

\medskip

\begin{remark}
 As the proof of the main results and the experiments show, in Algorithm 1 one does not need to have exact values of the Hessian, its eigenvalues and eigenvectors. Approximate values are good enough, for both theoretical and experimental purposes. 
 
 The augmentation term $||\nabla f(x_n)||^{1+\alpha}$ in Algorithm 1 simultaneously serves several purposes: 
 
 - It is a convenient term to add into $\nabla ^2f(x_n)$ to make sure that the resulting matrix is invertible, whenever $x_n$ is not a critical point of $f$.
 
 - Near a non-degenerate critical point, it becomes small compared to the main term $\nabla ^2f(x_n)$ coming from the original Newton's method, and hence Algorithm 1 basically reduces to Newton's method. 
 
 - Also, it is discovered in the recent work \cite{truong2021} by the first author that it also helps to keep the eigenvalues of the resulting matrix sufficiently large, and hence helps to resolve the convergence guarantee issue in New Q-Newton's method. 
 
 \end{remark}

The plan of this paper is as follows. In Section 2, we briefly review about gradient descent methods for continuous optimisation problems. In the same section, we also briefly review about variants of Newton's method, what is currently known in the literature about convergence and avoidance of saddle points of iterative methods, and large scale performance of them. In Section 3, we present the definition of New Q-Newton's method and the proof of Theorem \ref{TheoremMain}. There we will also briefly review a very recent result by the first author incorporating Backtracking line search to New Q-Newton's method, and explain how it can be used to quickly find roots of meromorphic functions. The last section presents about implementation details, some experimental results (including a toy model of protein folding  \cite{shh} and stochastic optimization) and finding roots of meromorphic functions in 1 variable, in comparison with various well known second order optimization methods such as Newton's method, BFGS, Adaptive Cubic Regularization, Random damping Newton's method and Inertial Newton's method, and a first order algorithm Unbounded Backtracking GD.  There also some conclusions and ideas for future work are presented.  To keep the paper succinct, we present many more experimental results in the appendix.

Given that it is very expensive to implement this method, and that the authors at the moment have no access to huge computing resources, we defer the implementation of our method to large scale optimisation as in Deep Neural Networks to future work, when such computing resources and further theoretical work on numerical implementation are available. We mention that a large scale implementation of the related algorithm in \cite{dauphin-pascanu-gulcehre-cho-ganguli-bengjo} is now available at the GitHub link \cite{SFO}, however results reported there are only competitive on small datasets and DNN such as for MNIST. See Section \ref{SubsectionLargeScale} for some more discussions.

\section{Overview of the literature}

\subsection{Brief review of gradient descent methods}
\label{Subsection2}

The general version of Gradient Descent (GD), invented by Cauchy in 1847 \cite{cauchy}, is as follows. Let $\nabla f(x)$ be the gradient of $f$ at a point $x$, and $||\nabla f(x)||$ its Euclidean norm in $\mathbb{R}^k$.  We choose randomly a point $x_0\in \mathbb{R}^k$ and define a sequence
\begin{eqnarray*}
x_{n+1}=x_n-\delta (x_n) \nabla f(x_n),
\end{eqnarray*}
where $\delta (x_n)>0$ (learning rate), is appropriately chosen. We hope that the sequence $\{x_n\}$ will converge to a (global) minimum point of $f$. 

The simplest and most known version of GD is Standard GD, where we choose $\delta (x_n)=\delta _0$ for all $n$, here $\delta _0$ is a given positive number. Because of its simplicity, it has been used frequently in Deep Neural Networks and other applications.  Another basic version of GD is (discrete) Backtracking GD, which works as follows. We fix real numbers $\delta _0>0$ and $0<\alpha ,\beta <1$. We choose $\delta (x_n)$ to be the largest number  $\delta $ among the sequence $\{\beta ^m\delta _0:~m=0,1,2,\ldots\}$ satisfying Amijo's condition \cite{armijo}:
\begin{eqnarray*}
f(x_n-\delta \nabla f(x_n))-f(x_n)\leq -\alpha \delta ||\nabla f(x_n)||^2.  
\end{eqnarray*}

There are also the inexact version of GD (see e.g. \cite{bertsekas, truong-nguyen1, truong-nguyen2}). More complicated variants of the above two basic GD methods include: Momentum, NAG, Adam,  for Standard GD (see an overview in \cite{ruder}); and Two-way Backtracking GD, Backtracking Momentum, Backtracking NAG  for Backtracking GD (first defined in \cite{truong-nguyen1, truong-nguyen2}). There is also a stochastic version, denoted by SGD, which is usually used to justify the use of Standard GD in Deep Neural Networks.  

For convenience, we recall that a function $f$ is in class $C^{1,1}_L$, if $\nabla f$ is globally Lipschitz continuous with the Lipschitz constant $L$. The latter means that for all $x,y\in \mathbb{R}^m$ we have $||\nabla f(x)-\nabla f(y)||\leq L||x-y||$. We note that it could be a difficult task to determine whether a function is in $C^{1,1}_L$ or real analytic (or more generally satisfying the so-called Losjasiewicz gradient inequality), while usually with only one quick glance one could get a very good guess whether a function is in $C^1$ or $C^2$ (conditions needed to guarantee good performance of modifications of Backtracking GD). 

Closely related to Backtracking GD is the so-called Wolfe's method \cite{wolfe, wolfe2}, where the learning rates are chosen not by Backtracking but by combining Armijo's condition with an additional condition regarding curvature. The idea was to overcome the fact that the original version of Backtracking GD requires the learning rates to be uniformly bounded from above. To this end, we note that in the recent work \cite{truong-nguyen1, truong-nguyen2}, learning rates in Backtracking GD are now allowed to be unbounded from above. Moreover, Wolfe's method does not work as well as Backtracking GD: its theoretical results can only proven for functions in class $C^{1,1}_L$, and there are no proven results on convergence to critical points or avoidance of saddle points as good as those for Backtracking GD (see Subsection \ref{Subsection4}).

\subsection{Brief review of the literature on (quasi-)Newton's method}
\label{Subsection3}

Another famous iterative optimisation method is Newton's method (\cite{Newton1} and Section 1.4 in \cite{bertsekas}). It applies to functions $f\in C^2$ with the update rule: if $\nabla ^2f(x_n)$ is invertible, then we define $x_{n+1}=x_n-(\nabla ^2f(x_n))^{-1}\nabla f(x_n)$. 

If $f(x)=\frac{1}{2}<Ax,x>$ where $A$ is an invertible symmetric matrix, then $\nabla f(x)=Ax$ and $\nabla ^2f(x)=A$. Therefore, for every initial point $x_0$, the next point in the update of Newton's method is $x_1=x_0-(\nabla ^2f(x_0))^{-1}\nabla f(x_0)=0$. Hence, in case $A$ has negative eigenvalues, Newton's method will converge to a saddle point $x_0=0$.  

Its main purpose is to find critical points of $f$. Its is most preferred because if it converges, then usually it converges very fast, with rate of convergence being quadratic. However, this comes with a cost: we need to compute second derivatives, and hence Newton's method is very costly when applied to huge scale optimisation problems. Also, it has several other undesirable features. First, as seen above, it can converge to saddle points or even local maxima. Second, there are examples (see Subsection \ref{Subsection6}) where Newton's method diverges to infinity even when the cost function has compact sublevels. Third, we can not proceed in Newton's method when $\nabla ^2f$ is not invertible. 

There are many modifications of Newton's methods, aiming at resolving the three issues mentioned in the end of the previous paragraph. Among them, most famous ones are so-called quasi-Newton's methods (\cite{Newton2} and Section 2.2 in \cite{bertsekas}). Some famous quasi-Newton's methods are:  BFGS, Broyden's family, DFP and SR1. There are two main ideas common in these methods. The first main idea is to add replace $\nabla ^2f(x_n)$ by some {\bf positive definite} matrices in a less expensive manner. The heuristic behinds this is try to have the inverse of $\nabla f(x)$ to be a descent direction, and hence trying to have the sequence constructed by these methods to converge to local minima only. (As mentioned in the introduction, our idea of New Q-Newton's method is different in that we require only that $\nabla ^2f(x_n)+B_n$ is invertible.) However, we are not aware of rigorous results where avoidance of saddle points are established for these modifications, under such general conditions as in the main results in this paper. This procedure also aims to resolve the case where $\nabla ^2f(x_n)$  is not invertible, whence Newton's method is not applicable. The second main idea is to replace the expensive computation of $\nabla ^2f(x_n)$ by using first order approximations. This second main idea can be used also to our New Q-Newton's method to reduce the cost so that it can be implementable in huge scale optimisation problems.   

Another class of modifications of Newton's methods is that of damping Newton's method. The simplest form of this method is the update rule: $x_{n+1}=x_n-\delta _n(\nabla f(x_n))^{-1}\nabla f(x_n)$, where $\delta _n>0$ is a real number. One can choose $\delta _n$ randomly at each step. To this end, we note the paper \cite{sumi}, where by methods in complex dynamics, it is shown that Random damping Newton's method can find {\bf all roots} of a complex polynomial in 1 variable, if we choose $\delta _n$ to be a {\bf complex} random number so that $|\delta _n-1|<1$, and the rate of convergence is the same as the usual Newton's method. It is hopeful that this result can be extended to systems of polynomials in higher dimensions. On the other hand, this result shows that again damping Newton's method is not effective in finding local minima, since it can converge to all critical points. On how Random damping Newton's method works with non-polynomial functions, the readers can see some experimental results. 

Yet another common class of modifications of Newton's methods is discretisation of some differential equations, taking inspiration from physics. From basic differential equations corresponding to the usual Newton's method, one can add more terms (representing some physical rules) and discretising to obtain modifications. One recent such modification is the so-called  Inertial Newton's method \cite{bolte-etal}. The experimental results available for this method is not yet competitive enough. We note that again there is no theoretical guarantee that this method is effective in finding local minima, and also that its rate of convergence is not as fast as the usual Newton's method but rather comparable to that of Gradient Descent methods. 

Next, we will compare New Q-Newton's method to a couple of specific relevant methods. The first is a method in \cite{goldfeld-etal} (which has some relations to the Levenberg-Marquardt method). It proposes to add a term $-\lambda _1(x)+R||\nabla f(x)||$ to $\nabla ^2f(x)$, where $\lambda _1(x) $ is the smallest eigenvalue of $\nabla ^2f(x)$, and $R>0$ is chosen by some rules. While the term $||\nabla f(x)||$ is similar to the term $\delta _i||\nabla f(x)||^{1+\alpha} $ in New Q-Newton's method, the method in \cite{goldfeld-etal} is of a heuristic nature, and rigorous theoretical guarantee is given only for the case where the cost function is a quadratic function. The other relevant method is that of cubic regularization of Newton's method (which is relevant to the popular truth region method), see \cite{nesterov-polyak}. In \cite{nesterov-polyak}, the method is defined for the (restrictive) class of cost functions $f$ whose Hessian is {\bf globally Lipschitz continuous}, where at each iteration an optimal subproblem on the whole space $\mathbb{R}^m$ is requested. (The precise solution of this optima subproblem requires knowing eigenvalues and eigenvectors of the Hessian.) Under this restriction, it is shown that the sequence constructed has the descent property $f(x_{n+1})\leq f(x_n)$, and generalised saddle points can be avoided. Under some further restrictions, then convergence is also guaranteed. This method has been extended to more general functions in \cite{cartis-etal}, where the update rule is more complicated to describe and still need a similar optimal subproblem at each step. However, theoretical guarantees for this extension are weaker and are only provided under restrictive or practically difficult to check conditions, and experiments (see below) performed on an implementation of its \cite{ARCGitHub} do not show better performance than other algorithms. Moreover, we do not know if the cubic regularization method has yet a workable extension to Riemannian manifolds (only under some very restrictive assumptions, such as requiring that the cost function is in $C^{1,1}_L$ - this latter condition being quite cumbersome to define on Riemannian manifolds, that such versions exist or have good theoretical properties). In comparison, the first author of the current paper has been successful to extend New Q-Newton's method to Riemannian manifolds setting, see \cite{truongnew}.  

There is an active research direction mixing between Backtracking GD and Newton's method worth mentioning \cite{goldfarb, gould-etal} (our analysis below applies generally to the more recent works in this direction as well).  Since this seems to be the most relevant to New Q-Newton's method, we will provide a rather detailed analysis. The idea is, if at the point $x_n$ the Hessian $\nabla ^2f(x_n)$ has a negative eigenvalue, then one can also explore the (approximation of the) eigenvector  $d_n$ corresponding to the smallest eigenvalue, in addition to a gradient-like direction $s_n$. Here, it is assumed that there are 2 constants $c_1,c_2>0$ so that $<s_n,\nabla f(x_n)> \geq c_1||\nabla f(x_n))||^2$ and $||s_n||\leq c_2||\nabla f(x_n)||$ for all $n$. Note that checking these two conditions can be non-trivial, and can make computations expensive.  They choose $p_n=d_n$ or $s_n$ depending on whether a test is satisfied, and add a quadratic term into Armijo's condition, checking one condition of the form (some variants choose instead $p_n=$ a non-trivival linear combination of $s_n$ and $d_n$, and employ the usual Armijo's condition, see e.g. \cite{gould-etal} for details):  
\begin{equation}
f(x_n-\delta p_n)- f(x_n)\leq -\alpha [\delta <p_n,\nabla f(x_n)>+\frac{1}{2}\delta ^2 \min \{0,<\nabla ^2f(x_n)p_n,p_n>\}]. 
\label{EquationGoldfarb}\end{equation}
(Note that the $\delta$ satisfying this stronger inequality could be smaller than the one satisfying Armijo's condition, and hence practically this method can be {\bf slower} than if one uses Armijo's condition only.) The pro of this algorithm is that, because of the additional direction $d_n$, it can be shown that any {\bf cluster point} of the sequence $\{x_n\}$ cannot be a generalised saddle point. Moreover, one does not need to require that the initial point is {\bf randomly chosen} (but this condition is not essential, since if one wants to attain a good point at the end, then one better chooses a random initial point).  However, the addition of this $d_n$ is also a weak point of this method, as we will describe next.  

Note that in \cite{gould-etal}, while there is a statement that any cluster point of the sequence constructed by their algorithm is a critical point of the cost function, there is no explicit statement about condition for which the whole sequence converges to a unique limit point.   Here $s_n$' are chosen in two common classes, and we will separately analyse them. In Case 1, $s_n$ is an approximation of  the gradient $\nabla f(x_n)$. Then, because $d_n$ has almost no relation to $\nabla f(x_n)$, except the condition that $<d_n,\nabla f(x_n)>\geq 0$ (note, here we use learning rate $>0$, hence changing the size of $d_n$ from that in \cite{gould-etal}), there is no evidence that this algorithm has strong theoretical guarantee for cost functions satisfying the Losjasiewicz gradient inequality as Backtracking GD methods, see next section. Moreover, a modification of Backtracking GD, using only the (approximation of the) gradient direction and with a more carefully chosen learning rate, also can avoid generalised saddle points if one starts from a random initial point, see Theorem .\ref{Theorem1}. The mentioned theorem also gives support that Backtracking GD itself can avoid generalised saddle points, if one starts from a random initial point. Hence, both from theoretical and practical viewpoints, there is no real advantage of using the methods in \cite{goldfarb, gould-etal} over the usual Backtracking GD.  In Case 2, $s_n$' are chosen as Newton's like direction. In this case, also truncated calculations are used to apply to large scale. However, there are several disadvantages. First, it is not stated clearly how the algorithm deals with the case the Hessian is not invertible. New Q-Newton's method deals with this in a simple manner. Second,  if one needs a quadratic rate of convergence result for this method, then one needs to choose $s_n$ like $\nabla ^2f(x_n)^{-1}.\nabla f(x_n)$ near a non-degenerate local minimum, and the two conditions  $<s_n,\nabla f(x_n)> \geq c_1||\nabla f(x_n))||^2$ and $||s_n||\leq c_2||\nabla f(x_n)||$ are not enough. To this end, the truncated calculations are generally not enough to guarantee this, and hence a full calculations of eigenvectors and eigenvalues, as in New Q-Newton's method, will be needed. Then, near a degenerate critical point, the postulation about the existence of two constants $c_1,c_2$ satsifying $<s_n,\nabla f(x_n)> \geq c_1||\nabla f(x_n))||^2$ and $||s_n||\leq c_2||\nabla f(x_n)||$ for all $n$ cannot be fulfilled. In New Q-Newton's method we do not postulate this. Also, again because of the appearance of $d_n$, there is no guarantee about convergence of this method for cost functions satisfying the Losjasiewicz gradient inequality. See Subsection \ref{SectionRootMeromorphicFunction}  for some results which can be proven by New Q-Newton's method Backtracking \cite{truong2021} concerning these conditions. Hence, in this case, from both theoretical and practical viewpoints again, at least in medium-sized problems where calculating eigenvectors and eigenvalues of a square symmetric matrix is possible in a reasonable time, there is no real advantage of using  the concerned method over New Q-Newton's method Backtracking.

\subsection{Brief review of literature on convergence to critical points and avoidance of saddle points}
\label{Subsection4}

Here we provide a very brief review of the currently known most general results on performance of iterative methods, regarding convergence to critical points and avoidance of saddle points. More details to special cases can be found in the references mentioned here and references therein. 

{\bf Convergence to critical points:} We recall that a function $f$ is Morse if it is $C^2$, and if all of its critical points are non-degenerate (that is, if $\nabla f(x_0)=0$ then $\nabla ^2f(x_0)$ is invertible). By transversality results, Morse functions are dense in the set of all continuous functions. In other words, if we choose a random $C^2$ function, then it is Morse. We note also that the set of critical points of a Morse function is discrete. The following result (\cite{truong-nguyen1, truong-nguyen2, truong, truong2, truong3}) illustrates the good features of modifications of Backtracking GD: If $f$ is a Morse function, and $\{x_n\}$ is a sequence constructed by the Backtracking GD (or one of its various modifications), then either $\lim _{n\rightarrow\infty}||x_n||=\infty$ or there is $x_{\infty}$ so that $\lim _{n\rightarrow\infty}x_n=x_{\infty}$ and $\nabla f(x_{\infty})=0$. In the general case, where $f$ is only assumed to be $C^1$, it is shown in the mentioned papers that if the set of cluster points $\mathcal{D}$ of $\{x_n\}$ intersects one compact component of $\mathcal{C}=$ critical points of $f$, then $\mathcal{D}$ is connected and is contained in that compact component. This result also extends to functions defined on Banach spaces \cite{truong4}. To date, we do not know any other iterative methods whose convergence is as strongly guaranteed as Backtracking GD. 

For some special but interesting classes of functions, the corresponding results have been known much earlier. For example, in the case $f$ is in $C^{1,1}_L$ and has compact sublevels, the corresponding results are classical, and can be found as early as in Chapter 12 in \cite{lange}.  When the function $f$ is real analytic (or more generally satisfying the so-called Losjasiewicz gradient inequality), then we obtain the strongest form of convergence guarantee where no assumptions on the set of critical points are needed \cite{absil-mahony-andrews}.    

{\bf Avoidance of saddle points:} Besides minima, other common critical points for a function are maxima and saddle points. In fact, for a $C^2$ cost function, a non-degenerate critical point can only be one of these three types. While maxima are rarely a problem for descent methods, saddle points can theoretically be problematic, as we will present later in this subsection. Before then, we recall definitions of saddle points and generalised saddle points for the sake of unambiguous presentation. Let $f:\mathbb{R}^k\rightarrow \mathbb{R}$ be a $C^1$ function. Let $x_0$ be a critical point of $f$ near it  $f$ is $C^2$. 

{\bf Saddle point.} We say that $x_0$ is a saddle point if the Hessian $\nabla ^2f(x_0)$ is non-singular and has both positive and negative eigenvalues. 

{\bf Generalised saddle point.} We say that $x_0$ is a {\bf generalised} saddle point if the Hessian $\nabla ^2f(x_0)$ has at least one negative eigenvalue. Hence, this is the case for a non-degenerate maximum point.  

In practical applications, we would like the sequence $\{x_n\}$ to converge to a minimum point.  It has been shown in \cite{dauphin-pascanu-gulcehre-cho-ganguli-bengjo} via experiments that for cost functions appearing in DNN the ratio between minima and other types of critical points becomes exponentially small when the dimension $k$ increases, which illustrates a theoretical result for generic functions  \cite{bray-dean}. Which leads to the question: Would in most cases an iterative algorithm  converge to a minimum? 

To this question, again so far Backtracking GD and its modifications provide the best answer. For the special case of functions in class $C^{1,1}_{L}$, it is shown in \cite{lee-simchowitz-jordan-recht, panageas-piliouras} that if the initial point $x_0$ is outside a set of Lebesgue's measure $0$ then for the sequence $\{x_n\}$ constructed by Standard GD, with fixed learning rate $\delta <1/L$, if $x_n$ {\bf does converge} to a point $x_{\infty}$ then $x_{\infty}$ cannot be  a generalised saddle point. This result has been more recently extended in \cite{truong} to functions $f$ satisfying the more general assumption that $\nabla f$ is {\bf locally} Lipschitz continuous (for example, this is satisfied when $f$ is in $C^2$), by replacing Standard GD by Backtracking GD. The result is also valid more generally for functions defined on Banach spaces, see \cite{truong4}. By using the convergence results in \cite{absil-mahony-andrews, truong-nguyen1, truong-nguyen2}, one immediately obtain the following result,  (which as far as we know, is the strongest theoretical guarantee for iterative methods in the contemporary literature) - mentioned also in \cite{truong-nguyen2}:

\begin{theorem}
\label{Theorem1}

If one applies the variant of Backtracking GD in \cite{truong} to a cost function $f:\mathbb{R}^m\rightarrow \mathbb{R}$, which either has at most countably many critical points or satisfies the Losjasiewicz gradient inequality, then for a random initial point $x_0$, the sequence $x_n$ constructed either diverges to infinity or converges to a critical point of $f$. In the latter case, the limit point cannot be a generalised saddle point. 

\end{theorem}

For Morse cost functions, the combination between New Q-Newton's method and Backtracking line search obtains the best theoretical guarantee for iterative optimization methods in the literature, see \cite{truong2021} and Subsection \ref{SectionRootMeromorphicFunction}  for details. 

\subsection{Large scale performance}
\label{SubsectionLargeScale}

In any event, large scale implementation in the current literature of modifications of Newton's method does not seem to function or competitive for datasets larger than MNIST, and even for MNIST it seems does not have any comprehensive comparison/evaluation on performance (in particular, on important indicators such as validation accuracy or running time) with Gradient descent methods (including Backtracking gradient descent methods) reported in the literature. 

Indeed, modern DNN are trained by Gradient descent methods, and popular among them are SGD, NAG, Momentum, Adam, Adamax and so on. There have been many experiments showing that with a good choice of learning rate, SGD can perform better than adaptive methods such as Adam. On the other hand, all just mentioned algorithms depend heavily on a good choice of learning rate: if one does not carefully choose a good learning rate, then the performance can be very poor. This leads to a lot of tricks about manual fine tuning of learning rates in the literature. 

Recently (since August 2018), two authors of the current paper have developed various new theoretical results and practical implementations of Backtracking GD, with very good performance, see \cite{truong-nguyen1, truong-nguyen2} for details, and see also the more recent work \cite{vaswani-etal} for similar implementations and experimental results. We also have combined Backtracking GD with other algorithms such as Momentum or NAG. A special feature of the newly developed algorithms (named MBT-GD, MBT-MMT and MBT-NAG) is that they are very stable with respect to the choice of initial learning rate $\delta _0$. Even with models which are not strong enough for a given problem, such as LeNet for CIFAR10, these new algorithms still work quite well and stable. To illustrate, we present in below some experimental results reported in \cite{truong-nguyen1, truong-nguyen2}, see Table \ref{tab:optimizers} and Figures \ref{fig:lr_attenuation_mini} and \ref{fig:BestValuation}. 

\begin{table}[htp]
\fontsize{10}{6}\selectfont
  \centering
  \begin{tabular}{|l|c|c|c|c|c|c|c|c|c|c|}
  \hline
~~~~~~~Learning rates	& $100$   & $10$    & $1$     & $10^{-1}$&$10^{-2}$ &$10^{-3}$& $10^{-4}$&$10^{-5}$	&$10^{-6}$\\
\hline
SGD & $10.00$ & $89.47$ & $91.14$ & \it{92.07} & $89.83$ & $84.70$ & $54.41$ & $28.35$ & $10.00$\\
MMT & $10.00$ & $10.00$ & $10.00$ & \it{92.28} & $91.43$ & $90.21$ & $85.00$ & $54.12$ & $28.12$\\
NAG & $10.00$ & $10.00$ & $10.00$ & \it{92.41} & $91.74$ & $89.86$ & $85.03$ & $54.37$ & $28.04$\\
\hline
Adagrad & $10.01$ & $81.48$ & $90.61$ & $88.68$ & \it{91.66} & $86.72$ & $54.66$ & $28.64$ & $10.00$\\
Adadelta & $91.07$ & $92.05$ & \it{92.36} & $91.83$ & $87.59$ & $73.05$ & $46.46$ & $22.39$ & $10.00$\\
RMSprop & $10.19$ & $10.00$ & $10.22$ & $89.95$ & $91.12$ & \it{91.81} & $91.47$ & $85.19$ & $65.87$\\
Adam & $10.00$ & $10.00$ & $10.00$ & $90.69$ & $90.62$ & \it{92.29} & $91.33$ & $85.14$ & $66.26$\\
Adamax & $10.01$ & $10.01$ & $91.27$ & $91.81$ & \it{92.26} & $91.99$ & $89.23$ & $79.65$ & $55.48$\\
\hline
MBT-GD  & \multicolumn{9}{c|}{\it{91.64}}\\
MBT-MMT & \multicolumn{9}{c|}{\it{93.70}}\\
MBT-NAG & \multicolumn{9}{c|}{\bf{93.85}}\\
 \hline
  \end{tabular}
   \caption{Best validation accuracy for CIFAR10 on Resnet18 after $200$ training epochs (batch size $200$) of different optimisers using different starting learning rates (MBT methods, being stable with starting learning rate, only use starting learning rate $10^{-2}$ as default). This table is taken from \cite{truong-nguyen1}.}  
  \label{tab:optimizers}
\end{table}

\begin{figure}
\centering
        \begin{subfigure}[b]{0.5\textwidth}
        
         \includegraphics[width=\linewidth ]{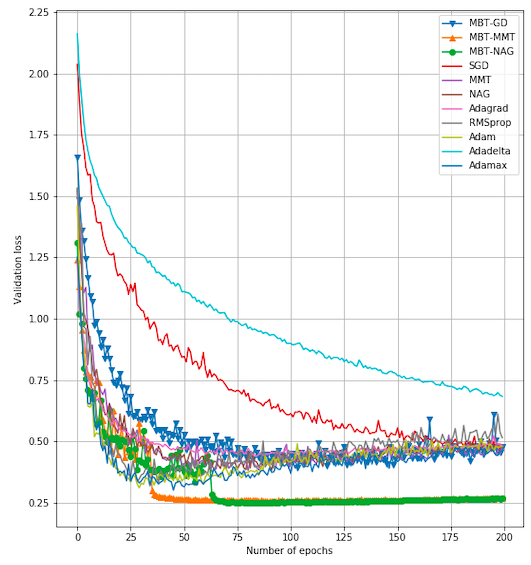}
              \caption{Validation loss for different algorithms, for CIFAR10 on Resnet18, mini-batch size 200.}  \label{fig:val_loss}
                      \end{subfigure}%
        \begin{subfigure}[b]{0.5\textwidth}
              
               \includegraphics[width=\linewidth ]{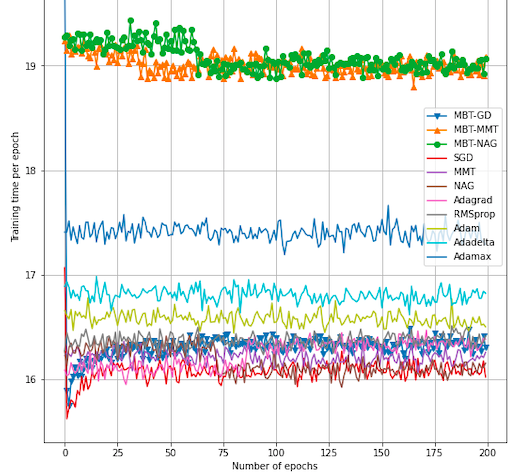}
              \caption{Training time (in seconds) per epoch for different algorithms, for CIFAR10 on Resnet18, mini-batch size 200. }  \label{fig:train_time}

                      \end{subfigure}%
        \caption{The actual training time from scratch for SGD, MMT, NAG, Adagrad, RMSProp, Adam, Adadelta and Adamax must be a high multiple of what reported here, in \ref{fig:train_time}, since these methods need manual fine-tune of hyperparameters to achieve good performance. This figure is taken from \cite{truong-nguyen2}, and has been produced in collaboration with Torus Actions SAS.} 
       \label{fig:lr_attenuation_mini}
\end{figure}

\begin{figure}
\centering

         \includegraphics[width=\linewidth ]{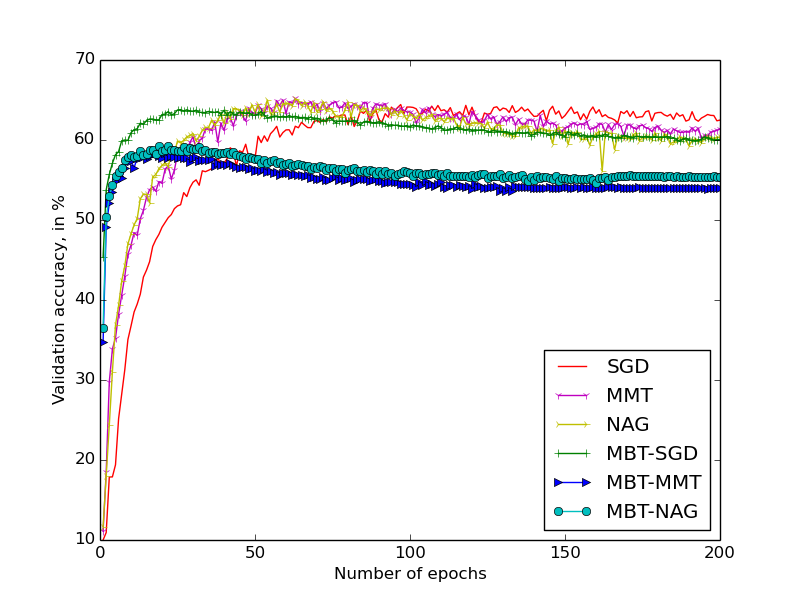}
            \caption{The evolution of validation accuracy in a training run, for both Non-Backtracking methods (SGD, MMT and NAG) and the corresponding Backtracking versions (MBT-SGD, MBT-MMT and MBT-NAG), for CIFAR10 on LeNet, for the same choice of mini-batch size 32 and normalisation as in \cite{bolte-etal}. For each method, we choose the best run among 5 random runs to report. The learning rate for Non-Backtracking methods is fixed to be $1e-2$ (which are found by a grid search to be very good for these methods). The initial learning rate for Backtracking methods is $1$. The momentum hyperparameter, for (MBT-)MMT and (MBT-)NAG methods is $\gamma =0.5$. This figure is taken from \cite{truong-nguyen2}.}
       \label{fig:BestValuation}
\end{figure}

\section{A new modification of Newton's methods: main results, proofs, and an application in meromorphic functions root finding}
We first give details of New Q-Newton's method and its main theoretical properties and their proofs. Then we review a new modification by the first author \cite{truong2021}, and discuss how to use it to quickly find roots of meromorphic functions in 1 complex variable.

\subsection{A new modification of Newton's methods and Main results}
\label{Subsection5}

We first recall a useful fact in Linear Algebra. Let $A$ be a symmetric $m\times m$ matrix with real entries. Then all eigenvalues of $A$  are real, and $A$ is diagonalisable. In fact,  there is an orthogonal matrix $Q$ so that $Q^TAQ$ is diagonal. In particular, if we let $\mathcal{E}^{\geq 0}(A)\subset \mathbb{R}^m$ (correspondingly $\mathcal{E}^{-}(A)\subset \mathbb{R}^m$) be the vector subspace generated by eigenvectors with non-negative eigenvalues of $A$ (correspondingly the vector subspace generated by eigenvectors with negative eigenvalues of $A$), then we have an orthogonal decomposition $\mathbb{R}^m=\mathcal{E}^{\geq 0}(A)\oplus \mathcal{E}^{-}(A)$, with respect to the usual inner product on $\mathbb{R}^m$. In particular, any $x\in \mathbb{R}^m$ can be written uniquely as $x=pr_{A,\geq}(x)+pr_{A,-}(x)$ where $pr_{A,\geq}(x)\in \mathcal{E}^{\geq 0}(A)$  and $pr_{A,-}(x)\in \mathcal{E}^{-}(A)$. 

In the situation of the above paragraph, if moreover $A$ is invertible, then all eigenvalues of $A$ are nonzero, and we denote in this case $\mathcal{E}^{+}(A)=\mathcal{E}^{\geq 0}(A)$ and $pr_{A,+}(x)=pr_{A,\geq 0}(x)$ for clarity. It is also worthwhile to note that $pr_{A,+}=pr_{A^{-1},+}$ and similarly $pr_{A,-}=pr_{A^{-1},-}$.

Now we are ready to  present our new modification of quasi-Newton's methods.

{\bf New Q-Newton's method.}  Let $\Delta =\{\delta _0,\delta _1,\delta _2,\ldots \}$ be a countable set of real numbers which has at least $m+1$ elements. Let $f:\mathbb{R}^m\rightarrow \mathbb{R}$ be a $C^2$ function. Let $\alpha >0$. For each $x\in \mathbb{R}^m$ such that $\nabla f(x)\not=0$, let $\delta (x)=\delta _j$, where $j$ is the smallest  number so that $\nabla ^2f(x)+\delta _j||\nabla f(x)||^{1+\alpha}Id$ is invertible. (If $\nabla f(x)=0$, then we choose $\delta (x)=\delta _0$.) Let $x_0\in \mathbb{R}^m$ be an initial point. We define a sequence of $x_n\in \mathbb{R}^m$ and invertible and symmetric $m\times m$ matrices $A_n$ as follows: $A_n=\nabla ^2f(x_n)+\delta (x_n) ||\nabla f(x_n)||^{1+\alpha}Id$ and $x_{n+1}=x_n-w_n$, where $w_n=pr_{A_n,+}(v_n)-pr_{A_n,-}(v_n)$ and $v_n=A_n^{-1}\nabla f(x_n)$. 

{\bf Remarks.} For to choose the set $\Delta$, we can do as in Backtracking GD: Let $\zeta _0>0$ and $0<\beta <1$, and define $\Delta =\{\beta ^n\zeta _0:~n=0,1,2,\ldots \}$. 

Note that if $\delta _0 =0$, then at points $x_n$ where $\nabla ^2f(x_n)$ is invertible, $A_n=\nabla ^2f(x_n)$. To ensure $\delta _0=0$, we can modify the construction of $\Delta$ in the previous paragraph as follows: $\Delta =\{\beta ^n\zeta _0-\zeta _0:~n=0,1,2,\ldots \}$. 

The following simple lemma is stated to emphasise the finiteness feature of the function $\delta (x)$ in the definition of New Q-Newton's method. 
\begin{lemma}
1) For all $x\in \mathbb{R}^m$, we have  $\delta (x)\in \{\delta _0,\ldots ,\delta _m\}$. 

2) If $x_{\infty}\in \mathbb{R}^m$ is such that $\nabla f(x_{\infty})=0$ and $\nabla ^2f(x_{\infty})$ is invertible, then for $x$ close enough to $x_{\infty}$ we have that $\delta (x)=\delta _0$. 
\label{Lemma1}\end{lemma}
 \begin{proof}
 1) If $\nabla f(x)=0$, then by definition we have $\delta (x)=\delta _0 \in \{\delta _0,\ldots ,\delta _m\}$ as claimed. In the case $\nabla f(x)\not= 0 $, then since $\nabla ^2f(x)$ has only $m$ eigenvalues, for at least one $\delta$ among $ \{\delta _0,\ldots ,\delta _m\}$ we must have $\nabla ^2f(x)+\delta ||\nabla f(x)||^2Id$ is invertible. Therefore, we have again that $\delta (x)\in  \{\delta _0,\ldots ,\delta _m\}$.
 
 2) For $x$ close enough to $x_{\infty}$, we have that $||\nabla f(x)||$ is small. Hence, since $\nabla ^2f(x_{\infty})$ is invertible, it follows that $\nabla ^2f(x)+\delta _0||\nabla f(x)||^2Id$ is invertible. Hence, by definition, for these $x$ we have $\delta (x)=\delta _0$. 
 
 \end{proof}

Now we are ready to prove Theorem \ref{TheoremMain}. 

\begin{proof}[Proof of Theorem \ref{TheoremMain}]

1) Since $\lim _{n\rightarrow\infty}x_n=x_{\infty}$, we have $w_n=x_{n+1}-x_n\rightarrow 0$. Moreover, $\nabla ^2f(x_n)\rightarrow \nabla ^2f(x_{\infty})$. Then, by Lemma \ref{Lemma1} and definition of $A_n$, we have that $||A_n||$ is bounded.  Note that by construction $||w_n||=||v_n||$ for all $n$, and hence $\lim _{n\rightarrow\infty}v_n=0$. It follows that  
\begin{eqnarray*}
\nabla f(x_{\infty})=\lim _{n\rightarrow\infty}\nabla f(x_n)=\lim _{n\rightarrow\infty}A_nv_n=0.  
\end{eqnarray*}

2) For simplicity, we can assume that $x_{\infty}=0$. We assume that $x_{\infty}$ is a saddle point, and will arrive at a contradiction. By 1) we have $\nabla f(0)=0$, and by the assumption we have that $\nabla ^2f(0)$ is invertible. 

We define $A(x)=\nabla ^2f(x)+\delta (x)||\nabla f(x)||^{1+\alpha}Id$, and $A=\nabla ^2f(0)=A(0)$. We look at the following (may not be continuous) dynamical system on $\mathbb{R}^m$: 
\begin{eqnarray*}
F(x)=x-w(x),
\end{eqnarray*}
 where $w(x)=pr_{A(x),+}(v(x))-pr_{A(x),-}(v(x))$ and $v(x)=A(x)^{-1}\nabla f(x)$.  
 
 Then for an initial point $x_0$, the sequence constructed by New Q-Newton's method is exactly the orbit of $x_0$ under the dynamical system $x\mapsto F(x)$. It follows from Lemma \ref{Lemma1} that $A(x)$ is $C^1$ near $x_{\infty}$, say in an open neighbourhood $U$ of $x_{\infty}$, and at every point $x\in U$ we have that $A(x)$ must be one of the $m+1$ maps $\nabla ^2f(x)-\delta _j||\nabla f(x)||^2Id$ (for $j=0,1,\ldots ,m$), and therefore $F(x)$ must be one of the corresponding $m+1$ maps $F_j(x)$. Since $f$ is assumed to be $C^3$, it follows that all of the corresponding $m+1$ maps $F_j$ are locally Lipschitz continuous. 
 
 Now we analyse the map $F(x)$ near the point $x_{\infty}=0$. Since $\nabla ^2f(0)$ is invertible, by Lemma \ref{Lemma1} again, we have that near $0$, then $A(x)=\nabla ^2f(x)+\delta _0||\nabla f(x)||^{1+\alpha}Id$.  Moreover, the maps $x\mapsto pr_{A(x),+}(A(x)^{-1}\nabla f(x))$ and $x\mapsto pr_{A(x),-}(A(x)^{-1}\nabla f(x))$ are $C^1$.  [This assertion is probably well known to experts, in particular in the field of perturbations of linear operators. Here, for completion we present a proof, following \cite{kato}, by using an integral formula for projections on eigenspaces via the theory of resolvents. Let $\lambda _1,\ldots ,\lambda _s$ be distinct solutions of the characteristic polynomials of $A$. By assumption, all $\lambda _j$ are non-zero. Let $\gamma _j\subset \mathbb{C}$ be a small circle with positive orientation enclosing $\lambda _j$ and not other $\lambda _r$. Moreover, we can assume that $\gamma _j$ does not contain $0$ on it or insider it, for all $j=1,\ldots ,s$. Since $A(x)$ converges to $A(0)$, we can assume that for all $x$ close to $0$, all roots of the characteristic polynomial of $A(x)$ are contained well inside the union $\bigcup _{j=1}^s\gamma _j$. Then by the formula (5.22) on page 39, see also  Problem 5.9, chapter 1 in \cite{kato}, we have that
\begin{eqnarray*}
P_j(x)=-\frac{1}{2\pi i}\int _{\gamma _j} (A(x)-\zeta )^{-1}d\zeta 
\end{eqnarray*} 
 is the projection on the eigenspace of $A(x)$ corresponding to the eigenvalues of $A(x)$ contained inside $\gamma _j$. Since $A(x)$ is $C^1$, it follows that $P_j(x)$ is $C^1$ in the variable $x$ for all $j=1,\ldots ,s$. Then, by the choice of the circles $\gamma _j$, we have 
\begin{eqnarray*}
pr_{A(x),+}=\sum _{j:~\lambda _j>0}-\frac{1}{2\pi i}\int _{\gamma _j} (A(x)-\zeta )^{-1}d\zeta 
\end{eqnarray*} 
 is $C^1$ in the variable $x$. Similarly,  
 \begin{eqnarray*}
pr_{A(x),-}=\sum _{j:~\lambda _j<0}-\frac{1}{2\pi i}\int _{\gamma _j} (A(x)-\zeta )^{-1}d\zeta 
\end{eqnarray*} 
is also $C^1$ in the variable $x$. Since $A(x)$ is $C^1$ in $x$ and $f(x))$ is $C^2$, the proof of the claim is completed.]

 Hence, since $x\mapsto (\nabla ^2f(x)+\delta _0||\nabla f(x)||^{1+\alpha}Id)^{-1}\nabla f(x)$ is $C^1$, it follows that the map $x\mapsto F(x)$ is $C^1$. We now compute the Jacobian of $F(x)$ at the point $0$. Since $\nabla f(0)=0$, it follows that $\nabla f(x)=\nabla ^2f(0).x+o(||x||)$, here we use the small-o notation, and hence
 \begin{eqnarray*}
 (\nabla ^2f(x)+\delta _0||\nabla f(x)||^{1+\alpha}Id)^{-1}\nabla f(x)=x+o(||x||).
 \end{eqnarray*} 
 It follows that $w(x)=pr_{A,+}(x)-pr_{A,-}(x)+o(||x||)$, which in turn implies that $F(x)=2pr_{A,-}(x)+o(||x||)$. Hence $JF(0)=2pr_{A,-}$. 

 Therefore, we obtain the existence of local Stable-central manifolds for the associated dynamical systems near saddle points of $f$ (see Theorems III.6 and III.7 in \cite{shub}). We can then, using the fact that under the assumptions that the hyperparameters $\delta _0,\ldots ,\delta _m$ are randomly chosen, to obtain:

 {\bf Claim:} The dynamical system is - outside of a set of Lebesgue measure $0$ - locally invertible, and hence the preimage of a set of Lebesgue measure $0$ again has Lebesgue measure $0$. 
 
A similar claim has been established for another dynamical systems in \cite{truong} - for a version of Backtracking GD. The idea in \cite{truong} is to show that the associated dynamical system (depending on $\nabla f$), which is locally Lipschitz continuous, has locally bounded torsion. The case at hand, where the dynamical system depends on the Hessian and also orthogonal projections on the eigenspaces of the Hessian, is more involved to deal with. 
 
We note that the fact that $\delta _0,\ldots ,\delta _m$ should be random to achieve  the truth of Claim has been overlooked in the arXiv version of this paper, and has now been corrected in a new work by the first author \cite{truongnew}, where the known results - including those in this paper - are extended to the Riemannian manifold setting. We will sketch here main ideas of how Claim can be proven, and refer the readers to \cite{truongnew} for more details.   

Putting, as  above, $A(x,\delta )=\nabla ^2f(x)+\delta ||\nabla f(x)||^{1+\alpha}Id$. Let $\mathcal{C}=\{x\in \mathbb{R}^m:~\nabla f(x)=0\}$ be the set of critical points of $f$. One first use the fact that $\det (A(x,\delta ))$ is a polynomial, and is non-zero for $x\notin \mathcal{C}$, to show that there is a set $\Delta \subset \mathbb{R}$ of Lebesgue measure $0$ so that for a given $\delta \notin \Delta$, the set $x\notin \mathcal{C}$ for which $A(x,\delta )$ is not invertible has Lebesgue measure $0$. One then shows, using that $w(x,\delta )$ (that is, the $w(x)$ as above, but now we add the parameter $\delta$ in to make clear the dependence on $\delta$), is a rational function in  $\delta$, and is non-zero (by looking to what happens when $\delta \rightarrow \infty$). This allows one to show that there is a set $\Delta '\subset \mathbb{R}\backslash \Delta$ of Lebesgue measure $0$ so that for all $\delta \notin (\Delta \cup \Delta ') $ then $A(x,\delta )$ is invertible and the set where the {\bf gradient} of the associated dynamical system $F(x)=x-w(x,\delta )$ is, locally outside $\mathcal{C}$, invertible. This proves the Claim.

From the above proof, we have an explicit criterion for $\delta _0,\ldots ,\delta _m$ to be random: they should avoid the set $\Delta \cup \Delta '$.

3) We can assume that $x_{\infty}=0$, and define $A=\nabla ^2f(0)$. The assumption that $\nabla ^2f(0)$ is invertible and 1) - as well as Lemma \ref{Lemma1} - imply that we can assume, without loss of generality, that $A_n=\nabla ^2f(x_n)+\delta _0||\nabla f(x_n)||^{1+\alpha}Id$ for all $n$, and that $\nabla ^2f(x_n)$ is invertible for all $n$.  Since $\nabla f(0)=0$ and $f$ is $C^3$, we obtain by Taylor's expansion $\nabla f(x_n)=A.x_n+O(||x_n||^2)$.   Then, by Taylor's expansion again we find that
\begin{eqnarray*}
A_n^{-1}&=&\nabla ^2f(x_n)^{-1}.(Id+\delta _0||\nabla f(x_n)||^{1+\alpha}\nabla ^2f(x_n))^{-1}\\
&=&\nabla ^2f(x_n)^{-1}(Id-\delta _0||\nabla f(x_n)||^{1+\alpha}\nabla ^2f(x_n)+(\delta _0||\nabla f(x_n)||^{1+\alpha}\nabla ^2f(x_n))^2+\ldots )\\
&=&\nabla ^2f(x_n)^{-1}+O(||x_n||^{1+\alpha})=A^{-1}+O(||x_n||).  
\end{eqnarray*}

Multiplying $A_n^{-1}$ into both sides of the equation $\nabla f(x_n)=\nabla ^2f(0).x_n+O(||x_n||^2)$, using the above approximation for $A_n^{-1}$, we find that
\begin{eqnarray*}
v_n=A_n^{-1}\nabla f(x_n)=x_n+O(||x_n||^2). 
\end{eqnarray*}

Since we assume that $x_0\notin \mathcal{A}$, it follows that $A$ is positive definite. Hence we can assume, without loss of generality, that $A_n$ is positive definite for all $n$. Then from the construction, we have that $w_n=v_n$ for all $n$. Hence, in this case, we obtain
\begin{eqnarray*}
x_{n+1}=x_n-w_n=x_n-v_n=O(||x_n||^2),
\end{eqnarray*} 
thus the rate of convergence is quadratic.  

4) The proof of part 3 shows that in general we still have
\begin{eqnarray*}
v_n=x_n+O(||x_n||^2).
\end{eqnarray*}
Therefore, by construction we have $w_n=pr_{A_n,+}(v_n)-pr_{A_n,-}(v_n)=O(||x_n||)$. Hence  $x_{n+1}=x_n-w_n=O(||x_n||)$, and thus the rate of convergence is at least linear. 
  
5) This assertion follows immediately from the proof of part 3). 

\end{proof}

\subsection{Quickly finding roots of meromorphic functions in 1 complex variable}
\label{SectionRootMeromorphicFunction}
In this subsection we discuss how our method can be used to quickly find roots of meromorphic functions in 1 complex variable. Since the main focus of our paper is on optimization in general, we will only briefly mention most relevant facts, and refer interested readers to references. 

Solving equations is an important task for both theory and applications. Solving polynomial equations $g(z)=0$ has been important in the development of mathematics and science, and there are thousands of algorithms devoted to them, see \cite{pan}.  We specially mention here two variants of Newton's method which have global convergence guarantee and  relevant to our method. One is random damping Newton's method $x_{n+1}=x_n-\delta _n[\nabla ^2g(x_n)]^{-1}\nabla g(x_n)$ (mentioned already in the review section, here $\delta _n$'s are random complex numbers), for which global convergence guarantee is established using techniques from complex dynamics \cite{sumi}. Another idea is in \cite{kalantari}, inspired by \cite{smale}, computing at each point $z$ a specific amount $\Delta z$ for which $|g(z+\Delta z)|^2<|g(z)|^2$. There are two versions proposed in \cite{kalantari}. One which has a quadratic rate of convergence, but convergence is only guaranteed locally when the initial point $z_0$ is chosen so that $|f(z_0)|$ is smaller than a quantity computed on the set of critical points of $g$. Another one has global convergence, but there is no statement on rate of convergence. The method in \cite{kalantari} can be viewed as a descent method in optimization, however it seems inflexible and  have restricted applications to polynomials in 1 variable. Compared to these two algorithms, New Q-Newton's method is more flexible and dynamic than the one in  \cite{kalantari}, while New Q-Newton's method is more deterministic than random damping Newton's method (New Q-Newton's method needs only $m+1$ hyperparameters $\delta _0,\ldots ,\delta _m$ and these are chosen from beginning). Also, New Q-Newton's method applies to  2 real variables and a system of 2 real equations (see below), while the mentioned algorithms apply to 1 complex variable and 1 complex equation.  

Coming to finding roots of a holomorphic function $g(z)=0$, there are fewer options. One most used idea seems to be that in \cite{delves-lyness},  which amounts to finding effective ways to compute integrals of the form 
\begin{eqnarray*}
\frac{1}{2\pi i}\int _{C}z^N\frac{g'(z)}{g(z)}dz,
\end{eqnarray*}
where $C$ is the boundary of a bounded domain in $\mathbb{C}$. By Cauchy's integral formula, the resulting is $\sum _{i=1}^Nz_i^N$, where $z_i$ are the roots of $g$ in the domain bounded by $C$. One can also combine this with iterative methods, for example estimating the number of roots inside the domain by this integral calculation with $N=0$ and then apply iterative methods; or finding a polynomial with the same roots in the domain as the function $g$ by calculating the integrals for N going from $1$ to the number of roots, and then apply methods for finding roots of a polynomial. A well known older method is that of Lehmer's \cite{lehmer}, which uses a special procedure to determine whether there is at least 1 root of $g$ inside a given domain, and then divide the domain to smaller domains and apply the same special procedure, to locate roots of $g$ to a given error. The idea in \cite{smale} can also be applied to holomorphic functions, but becomes more complicated.  

Computer algebra softwares, like Mathematica and Matlab, have routines to do the above tasks. While we do not know the precise algorithms used by these softwares, it is reasonable to guess that they are based on iterative methods, e.g. Newton's method.  

Optimization can be applied to solve the above questions, and more general systems of equations. Here, we explicitly describe how to use our method to find roots of meromorphic functions in 1 complex variable. This case, as far as we know, is not extensively discussed in the literature. Besides being usually fast, iterative optimization methods have the advantages of being easy to understand conceptually, flexible and easy to implement.  

Let $g$ be a meromorphic function in 1 complex variable $z\in \mathbb{C}$. Then, outside a discrete set (poles of $g$), $g$ is a usual holomorphic function. To avoid the trivial case, we can assume that $g$ is non-constant. We write $z=x+iy$, where $x,y\in \mathbb{R}$. We define $u(x,y)=$ the real part of $g$, and $v(x,y)=$ the imaginary part of $g$. Then we consider a function $f(x,y)=u(x,y)^2+v(x,y)^2$. Then a zero $z=x+iy$ of $g$ is a global minimum of $f$, at which the function value is $0$. Therefore, optimization algorithm can be used to find roots of $g$, by applying to the function $f(x,y)$, provided the  algorithm assure convergence to critical points and avoidance of saddle points, and provided that critical points of $f$ which are not zeros of $g$ must be saddle points of $f$. In the remaining of this subsection, we will address these issues. 

First of all, while New Q-Newton's method does not have convergence guarantee to critical points, a modification of it, called New Q-Newton's method Backtracking, has this property \cite{truong2021}. Roughly speaking, in defining New Q-Newton's method Backtracking, one does two changes from that of New Q-Newton's method. The first change is that instead of requiring $\det (\nabla ^2f(x_k)+\delta _j||\nabla f(x_k)||^{1+\alpha }Id)\not= 0$,  one asks for a stronger condition that all eigenvalues of $\nabla ^2f(x_k)+\delta _j||\nabla f(x_k)||^{1+\alpha }I$  has absolute value $\geq \frac{1}{2}(\inf _{i\not= i'}|\delta _i-\delta _{i'}|)||\nabla f(x_k)||^{1+\alpha}$. The second change is to add a Backtracking line search component, using that the vector $-w_k$ constructed by New Q-Newton's method is a descent direction. In the case of a Morse function, one obtains the following result, which is so far the best theoretical guarantee for iterative optimization methods in the literature, as far as we know. Interested readers are referred to \cite{truong2021} for details. 

\begin{theorem} Let $f:\mathbb{R}^m\rightarrow \mathbb{R}$ be a $C^3$ function. Let $x_0$ be an initial point and $\{x_n\}$ the sequence constructed by New Q-Newton's method Backtracking. 

1) $f(x_{n+1})\leq f(x_n)$ for all $n$. Moreover, any cluster point of $\{x_n\}$ is a critical point of $f$.

2) Assume moreover that $f$ is Morse (that is, all its critical points are non-degenerate) and $x_0$ is randomly chosen. Then we have two alternatives: 

i) $\lim _{n\rightarrow\infty}||x_n||=\infty$, 

or 

ii) $\{x_n\}$ converges to a local minimum of $f$, and the rate of convergence is quadratic. 

Moreover, if $f$ has compact sublevels, then only case ii) happens. 

\label{TheoremMorse}\end{theorem}  

We now discuss the application of this result to the function $f(x,y)$ constructed from a meromorphic function $g(z)$, as mentioned before. If $g$ is holomorphic, then $f$ is well-defined everywhere, and $f$ has compact sublevels iff it is a polynomial. In case $g$ is not holomorphic, then it has poles and hence $f(x,y)$ is not well-defined on the whole $\mathbb{R}^2$. However, near a pole of $g$, then the value of $f$ is very large, and hence if one starts from an initial point $(x_0,y_0)$ which is not a pole of $g$, then by virtue of  the descent property of New Q-Newton's method Backtracking, the sequence $\{x_n\}$ will never land on a pole of $g$ and hence is well-defined. Indeed, since in this case the function $f(x,y)$ is real analytic, combining the ideas from \cite{absil-mahony-andrews} and \cite{truong2021}, we obtain the following strengthen of  Theorem \ref{TheoremMorse}. 

\begin{theorem} Let $f(x,y)$ be the function constructed from a non-constant meromorphic function $g(z)$ as before. Assume that the constant $\alpha >0$ in the definition of New Q-Newton's method does not belong to the set $\{(n-3)/(n-1):~n=2,3,4,\ldots  \}$. (For example, we can choose $\alpha =1$.)  Let $(x_n,y_n)$ be a sequence constructed by New Q-Newton's method Backtracking from an arbitrary initial point which is not a pole of $f$. Then either $\lim _{n\rightarrow\infty}(x_n^2+y_n^2)=\infty$, or the sequence $\{(x_n,y_n)\}$ converges to a point $(x^*,y^*)$ which is a critical point of $f$. 
\label{TheoremMeromorphic2}\end{theorem}

The proof of Theorem \ref{TheoremMeromorphic2} will be given at the end of this subsection, after some preparations. 
 
We now discuss properties of critical points of $f(x,y)=u(x,y)^2+v(x,y)^2$, outside poles of the meromorphic function $g(z)=u(z)+iv(z)$, where $z=x+iy$. Recall that by Riemann-Cauchy's equation, we have
\begin{eqnarray*}
\frac{\partial u}{\partial x}&=&\frac{\partial v}{\partial y},\\
\frac{\partial u}{\partial y}&=&-\frac{\partial v}{\partial x}. 
\end{eqnarray*} 

\begin{lemma} Assumptions be as above. 

1) A point $(x^*,y^*)$ is a critical point of $f(x,y)$, iff $z^*=x^*+iy^*$ is a zero of $g(z)g'(z)$.

2) If $z^*=x^*+iy^*$ is a zero of $g$, then $(x^*,y^*)$ is an isolated global minimum of $f$. Moreover, if $z^*$ is not a root of $g'$, then $(x^*,y^*)$ is a non-degenerate critical point of $f$. 

3) If $z^*=x^*+iy^*$ is a zero of $g'$, but not a zero of $gg"$, then $(x^*,y^*)$ is a saddle point of $f$. 

\label{LemmaMeromorphic}\end{lemma}
\begin{proof} We will write $u_x$ for $\partial u/\partial x$, $u_{xy}$ for $\partial ^2u/\partial x\partial y$ and so on. 

1) By calculation, we have $\nabla f=(2uu_x+2vv_x,2uu_y+2vv_y)$. By Cauchy-Riemann's equation, a critical point $(x^*,y^*)$ of $f$ satisfies a system of equations
\begin{eqnarray*}
uu_x-vu_y&=&0,\\
uu_y+vu_x&=&0,
\end{eqnarray*}
Consider the above as a system of linear equations in variables $u_x,u_y$, we see that if $(x^*,y^*)$ is not a root of $g$, then it must be a root of $u_x,u_y$. In the latter case, by Cauchy-Riemann's equation, $(x^*,y^*)$ is also a root of $v_x,v_y$, and hence $z^*=x^*+iy^*$ is a root of $g'(z)$. 

2) Since $f\geq 0$, and $f(x^*,y^*)=0$ iff $z^*=x^*+iy^*$ is a root of $g$, such an $(x^*,y^*)$ is a global minimum of $f$. Moreover, since the zero set of $g$ is discrete, $(x^*,y^*)$ is an isolated global minimum. 

For the remaining claim, we need to show that if $z^*$ is not a root of $g'$, then $\nabla ^2f(x^*,y^*)$ is invertible. By calculation,  the Hessian of f at a general point is 2 times of: 

 \[ \left( \begin{array}{cc}
u_x^2+v_x^2+uu_{xx}+vv_{xx}&u_xu_y+v_xv_y+uu_{xy}+vv_{xy}\\
u_{x}u_y+v_xv_y+uu_{xy}+vv_{xy}&u_y^2+v_y^2+uu_{yy}+vv_{yy}\\
\end{array}\right) \]
At $(x^*,y^*)$ we have $u=v=0$, and hence by Cauchy-Riemann's equation the above matrix becomes: 
 \[ \left( \begin{array}{cc}
u_x^2+u_y^2&0\\
0&u_x^2+u_y^2\\
\end{array}\right) \]
which is positive definite if $z^*$ is not a root of $g'$, as wanted. 

3) Since here $(x^*,y^*)$ is a solution of $u_x=u_y=v_x=v_y=0$, the Hessian of $f$ at $(x^*,y^*)$ is 2 times of: 
 \[ \left( \begin{array}{cc}
uu_{xx}+vv_{xx}&uu_{xy}+vv_{xy}\\
uu_{xy}+vv_{xy}&uu_{yy}+vv_{yy}\\
\end{array}\right) \]
 Note that by Cauchy-Riemann's equation we have $u_{xx}+u_{yy}=0$ and $v_{xx}+v_{yy}=0$. Therefore, if we put $a=uu_{xx}+vv_{xx}$ and $b=uu_{xy}+vv_{xy}$, then the above matrix becomes:  
 \[ \left( \begin{array}{cc}
a&b\\
b&-a\
\end{array}\right) \]
Since the determinant is $-a^2-b^2$, we conclude that $(x^*,y^*)$ is a saddle point of $f$, except the case where $a=b=0$. In the latter case, by Cauchy-Riemann's equation we have $u_{xy}=v_{xx}$ and $v_{xy}=-u_{yy}$, and hence $(x^*,y^*)$ must be a solution to 
\begin{eqnarray*}
uu_{xx}+vv_{xx}&=&0,\\
vu_{xx}-uv_{xx}&=&0. 
\end{eqnarray*}
By Cauchy-Riemann's equation again, we find that this cannot be the case, except that $z^*$ is a root of $gg"=0$.

\end{proof} 
 
 For a generic meromorphic function $g$, we have that $g'$ and $gg"$ have no common roots. Hence, by this lemma and Theorem \ref{TheoremMorse}, we obtain
 \begin{theorem} Let $g$ be a generic meromorphic function in 1 complex variable, and let $f(x,y)$ be the function in 2 real variables constructed from $g$ as above.  Let $(x_n,y_n)$ be the sequence constructed by applying New Q-Newton's method Backtracking to $f$ from a random initial point $(x_0,y_0)$. Then either
 
 i) $\lim _{n\rightarrow\infty}(x_n^2+y_n^2)=\infty$, 
 
 or
 
 ii) $(x_n,y_n)$ converges to a point $(x_{\infty},y_{\infty})$ so that $z_{\infty}=x_{\infty}+iy_{\infty}$ is a root of $g$, and the rate of convergence is quadratic. 
 
 Moreover, if $g$ is a polynomial, then $f$ has compact sublevels, and hence only case ii) happens.

 \label{TheoremGenericMeromorphic}\end{theorem}

 If $h$ is a non-constant meromorphic function, then $g=h/h'$ has only simple zeros (which are either zeros or poles of $h$). Hence, they will be non-degenerate global minima of $f$. If $h$ is a polynomial, then $g=h/h'$ has compact sublevels.  
 
 Now we are ready to prove Theorem \ref{TheoremMeromorphic2}. 
 \begin{proof}[Proof of Theorem \ref{TheoremMeromorphic2}]
 
 Let $\Omega$ be the complement of the set of poles of $f$. Then as mentioned, $f$ is real analytic on $\Omega$. Let $z_n=(x_n,y_n)$ be a sequence constructed by New Q-Newton's method Backtracking in \cite{truong2021}. Then, as commented above, if the initial point is in $\Omega$, then the whole sequence stays in $\Omega$. 
 
We know by \cite{truong2021} that any cluster point of $\{z_n\}$ is a critical point of $f$. Hence, it remains to show that $\{z_n\}$ converges. To this end, by the arguments in \cite{absil-mahony-andrews}, it suffices to show that for every point $(x^*,y^*)\in \Omega$, if the point $z_k=(x_k,y_k)$ is in a small open neighbourhood of $(x^*,y^*)$ then there is a constant $C>0$ (depending on that neighbourhood) so that 
\begin{equation}\label{EquationMeromorphic2}
f(z_k)-f(z_{k+1})\geq C||z_{k+1}-z_k||\times ||\nabla f(z_k)||. 
\end{equation}

Let us recall that  if $w_k$ is the one constructed in New Q-Newton's method, then $z_{k+1}=z_k-\beta _kw_k$, where $\delta _k$ is chosen from the Backtracking line search so that Armijo's condition 
\begin{eqnarray*}
f(z_k)-f(z_{k+1})\geq \frac{1}{2}\beta _k<w_k,\nabla f(z_k)>. 
\end{eqnarray*}
     
For a 2x2 invertible matrix $A$, we define $sp(A)=\max \{|\lambda |:~$ $\lambda$ is an eigenvalue of $A\}$, and  $minsp(A)=\min \{|\lambda |:~$ $\lambda$ is an eigenvalue of $A\}$. Then by the arguments in \cite{truong2021}, we find that 
\begin{eqnarray*}
\delta _k<w_k,\nabla f(z_k)>&\geq& \beta _k||w_k||\times ||\nabla f(z_k)||\times minsp(A_k)/sp(A_k)\\
&=&||z_k-z_{k+1}||\times ||\nabla f(z_k)||\times minsp(A_k)/sp(A_k),
\end{eqnarray*}
where $A_k=\nabla ^2f(z_k)+\delta ||\nabla f(z_k)||^{1+\alpha}Id$ is constructed by New Q-Newton's method. Here, recall that $\delta $ belongs to a finite set $\{\delta _0,\ldots ,\delta _m\}$.  Hence, to show that (\ref{EquationMeromorphic2}) is satisfied, it suffices to show that every point $(x^*,y^*)\in \Omega$ has an open neighbourhood $U$ so that if $z_k\in U$ then $minsp(A_k)/sp(A_k)\geq C$ for some constant $C>0$ depending only on $U$. 

If $(x^*,y^*)$ is not a critical point of $f$, then by the construction of New Q-Newton's method Backtracking, $minsp(A_k)\geq ||\nabla f(z_k)||^{1+\alpha }$ is bounded away from 0 in a small neighbourhood $U$ of $(x^*,y^*)$, while $sp(A_k)$ is bounded from above in the same neighbourhood. Hence $minsp(A_k)/sp(A_k)$ is bounded away from 0 in $U$ as wanted. 

Hence, we need to check the wanted property only at the critical points of $f$. We saw in Lemma \ref{LemmaMeromorphic} that $(x^*,y^*)$ is a critical point of $f$ iff $z^*=x^*+iy^*$ is a root of $gg'$. Hence, we will consider two seperate cases. To simplify the arguments, we can assume that $z^*=0$ is the concerned root of $gg'$.

{\bf Case 1:} $z^*=0$ is a zero of $g$. 

We expand in a small neighbourhood of $0$: $g(z)=\tau z^N+h.o.t$, where $N \not= 0$ and $p\geq 1$ is the multiplicity of $0$. We first claim that when $z$ is close to $z^*$, then the two eigenvalues of $\nabla ^2f(z)$ are $\lambda _1(z)\sim (2N^2-N)|\tau |^2r^{2N-2} $ and $\lambda _2(z)\sim N|\tau |^2r^{2N-2}$, where $r=||z||$. For simplicity, we can assume that $\tau =1$. 

Write $z=re^{i\theta }$. We have, by definition $u+iv=z^N$, $u_x+iv_x=\frac{d}{dx}(x+iy)^{n}$ and so on. Hence, 
\begin{eqnarray*}
u&=&r^n\cos (n\theta )+h.o.t.,\\
v&=&r^n\sin (n\theta )+h.o.t.,\\
u_x&=&nr^{n-1}\cos ((n-1)\theta ),\\
v_x&=&nr^{n-1}\sin ((n-1)\theta ),\\
u_y&=&-v_x=-nr^{n-1}\sin ((n-1)\theta ),\\
v_y&=&u_x=nr^{n-1}\cos ((n-1)\theta ),\\
u_{xx}&=&n(n-1)r^{n-2}\cos ((n-2)\theta ),\\
v_{xx}&=&n(n-1)r^{n-2}\sin ((n-2)\theta ),\\
u_{yy}&=&-u_{xx}=-n(n-1)r^{n-2}\cos ((n-2)\theta ),\\
v_{yy}&=&-v_{xx}=-n(n-1)r^{n-2}\sin ((n-2)\theta ),\\
u_{xy}&=&v_{yy}=-n(n-1)r^{n-2}\sin ((n-2)\theta ),\\
v_{xy}&=&u_{xx}=n(n-1)r^{n-2}\cos ((n-2)\theta ).
\end{eqnarray*}
   
We recall that the Hessian $\nabla ^2f(x,y)$ is:  
 \[ \left( \begin{array}{cc}
u_x^2+v_x^2+uu_{xx}+vv_{xx}&u_xu_y+v_xv_y+uu_{xy}+vv_{xy}\\
u_{x}u_y+v_xv_y+uu_{xy}+vv_{xy}&u_y^2+v_y^2+uu_{yy}+vv_{yy}\\
\end{array}\right) \],
which by Cauchy-Riemann's equation becomes: 
\[ \left( \begin{array}{cc}
u_x^2+v_x^2+uu_{xx}+vv_{xx}&uu_{xy}+vv_{xy}\\
uu_{xy}+vv_{xy}&u_y^2+v_y^2+uu_{yy}+vv_{yy}\\
\end{array}\right) \],

The two concerned eigenvalues are the two roots of the characteristic polynomial of $A=\nabla ^2f(x,y)$, which is $t^2-tr (A)t+\det(A)$.  By Cauchy-Riemann's equation again, we have
\begin{eqnarray*}
tr(A)&=&u_x^2+v_x^2+u_y^2+v_y^2=2N^2r^{2N-2}+h.o.t.,\\
\det (A)&=&(u_x^2+v_x^2)(u_y^2+v_y^2)-(uu_{xx}+vv_{xx})^2-(uu_{xy}+vv_{xy})^2\\
&=&(u_x^2+v_x^2)(u_y^2+v_y^2)-(u^2+v^2)(u_{xx}^2+v_{xx}^2)\\
&=&N^4r^{4n-4}-N^2(N-1)^2r^{4N-4}=N^2(2N+1)r^{4N-4}+h.o.t.
\end{eqnarray*}
From this, it is easy to arrive at the claimed asymptotic values for the two eigenvalues of $\nabla ^2f(x,y)$:  $\lambda _1(z)\sim (2N^2-N)|\tau |^2r^{2N-2} $ and $\lambda _2(z)\sim N|\tau |^2r^{2N-2}$, where $r=||z||$. 

Now we complete the proof that (\ref{EquationMeromorphic2}) is satisfied in this case where $z^*=0$ is a root of $g(z)$. We need to estimate $minsp(A_k)/sp(A_k)$ when $z_k=(x_k,y_k)$ is close to $z^*$. We note that $A_k=\nabla ^2f(z_k)+\delta ||\nabla f(z_k)||^{1+\alpha }$. Hence, the two eigenvalues of $A_k$ are $\lambda _1(z_k)+\delta ||\nabla f(z_k)||^{1+\alpha }$ and $\lambda _2(z_k)+\delta ||\nabla f(z_k)||^{1+\alpha }$. Note that
\begin{eqnarray*}
||\nabla f(z_k)||^{1+\alpha}&=&[(uu_x+vv_x)^2+(uu_y+vv_y)^2]^{(1+\alpha )/2}\\
&=&N^{1+\alpha}r^{(2N-1)(1+\alpha)}+h.o.t.,
\end{eqnarray*}
 which is of smaller size compared to $\lambda _1(z_k)$ and $\lambda _2(z_k)$. Therefore, we have $minsp(A_k)/sp(A_k)\sim 1/(2N-1)$ for $z_k$ near $z^*$, which is bounded away from 0 as wanted. 
 
{\bf Case 2:}  $z^*=0$ is a root of $g'(z)$. 

If $z^*$ is also a root of $g(z)$, then we are reduced to Case 1. Hence, we can assume that $z^*$ is not a root of $g(z)$. Therefore, we can expand, in a small open neighbourhood of $z^*=0$: $g(z)=\gamma +\tau z^N+h.o.t.$, where $\gamma, \tau \not=0$. 

If $N=1$, then $z^*$ is not a root of $gg"$. Then by Lemma \ref{LemmaMeromorphic}, we obtain that $z^*$ is a saddle point of $f$. Hence, for $z_k$ near $z^*$ we obtain 
\begin{eqnarray*}
minsp(A_k)/sp(A_k)\sim minsp(\nabla ^2f(z^*))/sp(\nabla ^2f(z^*)),
\end{eqnarray*}
 which is bounded away from 0, as wanted. Thus, we can assume that $N\geq 2$. 
 
 Calculating as above we found: 
 \begin{eqnarray*}
 tr(\nabla ^2f(z))&=&2|\tau |^2N^2r^{2N-2},\\
 \det (\nabla ^2f(z))&=&|\tau |^4N^4r^{4N-4}-|\gamma |^2|\tau |^2N^2(N-1)^2r^{2N-4}. 
 \end{eqnarray*}
 Since $N\geq 2$, we have $|\det (\nabla ^2f(z))| >> | tr(\nabla ^2f(z))|^2$ near $z^*$. This means that the two eigenvalues $\lambda _1(z)$ and $\lambda _2(z)$ of $\nabla ^2f(z)$ are of the same size $\sim \sqrt{|\det (\nabla ^2f(z))|}/2$, which is about $|\gamma \tau | N(N-1) r^{n-2}/2$. 
 
 Now, the term $||\nabla f(z)||^{1+\alpha }$, which is about the size of $|\gamma |^{1+\alpha }|\tau |^{1+\alpha}N^{1+\alpha}r^{(N-1)(1+\alpha )}$, is of different size compared to $\lambda _1(z)$ and $\lambda _2(z)$, thanks to the condition that $\alpha$ does not belong to the set $\{(n-3)/(n-1):~n=2,3,\ldots \}$. Therefore, we obtain that $minsp(A_k)/sp(A_k)\sim 1$ near $z^*$ in this case.
 
 This completes the proof of the theorem.

 \end{proof}

\section{Implementation details, Some experimental results and Conclusions}

\subsection{Implementation details}
\label{SectionImplementation}
In this Subsection, we present some practical points concerning implementation details, for the language Python. Source code is in the GitHub link \cite{phuongGitHub}. 

Indeed, Python has already enough commands to implement New Q-Newton's method. There is a package, named numdifftools, which allows one to compute approximately the gradient and Hessian of a function. This package is also very convenient when working with a family $f(x,t)$ of functions, where $t$ is a parameter. Another package, named scipy.linalg, allows one to find (approximately) eigenvalues and the corresponding eigenvectors of a square matrix. More precisely, given a square matrix $A$, the command $eig(A)$ will give pairs $(\lambda ,v _{\lambda})$ where $\lambda $ is an approximate eigenvalue of $A$ and $v$ a corresponding eigenvector. 

One point to notice is that even if $A$ is a symmetric matrix with real coefficients, the eigenvalues computed by the command $eig$ could be complex numbers, and not real numbers, due to the fact that these are approximately computed. This can be easily resolved by taking the real part of $\lambda$, which is given in Python codes by $\lambda .real$. We can do similarly for the eigenvectors. A very convenient feature of the command $eig$ is that it already computes (approximate) orthonormal bases for the eigenspaces.    

Now we present the coding detail of the main part of New Q-Newton's method: Given a symmetric invertible matrix $A$ with real coefficients (in our case $A=\nabla ^2f(x_n)+\delta _j||\nabla f(x_n)||^{1+\alpha}$), and a vector $v$, compute $w$ which is the reflection of $A^{-1}.v$ along the direct sum of eigenspace of negative eigenvectors of $A$. First, we use the command $eig$ to get pairs $\{(\lambda _j, v_j)\}_{j=1,\ldots ,m}$. Use the command real to get real parts. If we write $v=\sum _{j=1}^m a_jv_j$, then $a_j=<v_j,v>$ (the inner product), which is computed by the Python command $np.dot(v_j,v)$. Then $v_{inv}=A^{-1}v=\sum _{j=1}^m(a_j/\lambda _j)v_j$. Finally, 
\begin{eqnarray*}
w=v_{inv}-2\sum _{j:~\lambda _j<0}(a_j/\lambda _j)v_j.
\end{eqnarray*}

\begin{remark}

1) We do not need to compute exactly the gradient and the Hessian of the cost function $f$, only approximately. Indeed, the proof of Theorem \ref{TheoremMain} shows that if one wants to stop when $||\nabla f(x_n)||$ and $||x_n-x_{\infty}||$ is smaller than a threshold $\epsilon$, then it suffices to compute the gradient and the Hessian up to an accuracy of order $\epsilon$. 

Similarly, we do not need to compute the eigenvalues and eigenvectors of the Hessian exactly, but only up to an accuracy of order $\epsilon$, where $\epsilon$ is the threshold to stop. 

In many experiments,  we only calculate the Hessian inexactly using the numdifftools package in Python, and still obtain good performance. 

2) While theoretical guarantees are proven only when the hyperparameters $\delta _0,\ldots ,\delta _m$ are randomly chosen and fixed from the  beginning, in experiments we have also tried to choose - at each iterate $n$ - choose randomly a $\delta$. We find that this variant, which will be named {\bf Random New Q-Newton's method}, has {a similar or better performance} as the original version. 

3) Note that similar commands are also available on PyTorch and TensorFlow, two popular libraries for implementing Deep Neural Networks.

\end{remark}

\subsection{Some experimental results} Here we present a couple of illustrating experimental results. Additional experiments, which are quite extensive, will be presented in the appendix to the paper. We use the python package numdifftools \cite{num} to compute gradients and Hessian, since symbolic computation is not quite efficient. Most of the experiments are run on a small personal laptop, except the cases $N_n=500$ and $N_n=1000$ in Table \ref{tab:Griewank2} where we have to run on a stronger computer (Processor: Intel Core i9-9900X CPU @ 3.50GHzx20, Memory: 125.5 GiB). The unit for running time is seconds.

Here, we will compare the performance of New Q-Newton's method against the usual Newton's method, BFGS \cite{Newton2} and Section 2.2 in \cite{bertsekas}, Adaptive Cubic Regularization \cite{nesterov-polyak, cartis-etal}, as well as Random damping Newton's method \cite{sumi}  and Inertial Newton's method \cite{bolte-etal}.
 
In the experiments, we will use the Generalised New Q-Newton's method in Section \ref{Subsection7}, since it uses smaller quantities in general. We remark that if we use the basic version of New Q-Newton's method in Table \ref{table:alg} then we obtain similar results.  We choose $\alpha =1$ in the definition. Moreover, we will choose $\Delta =\{0,\pm 1\}$, even though for theoretical proofs we need $\Delta$ to have at least $m+1$ elements, where $m=$ the number of variables. The justification is that when running New Q-Newton's method it almost never happens the case that both $\nabla ^2f(x)$ and $\nabla ^2f(x)\pm ||\nabla f(x)||^2Id$ are not invertible. The experiments are coded in Python and run on a usual personal computer. For BFGS: we use the function scipy.optimize.fmin$\_$bfgs available in Python, and put  $gtol=1e-10$ and $maxiter=1e+6$. For Adaptive cubic regularization for Newton's method, we use the AdaptiveCubicReg module in the implementation in \cite{ARCGitHub}. We use the default hyperparameters as recommended there, and use "exact" for the hessian$\_$update$\_$method.  For hyperparameters in Inertial Newton's method, we choose $\alpha =0.5$ and $\beta =0.1$ as recommended by the authors of \cite{bolte-etal}. Source codes for the current paper are available at the GitHub link \cite{phuongGitHub}.

We will also compare the performance to Unbounded Two-way Backtracking GD \cite{truong-nguyen1}. The hyperparameters for Backtracking GD are fixed through all experiments as follows: $\delta _0=1$, $\alpha =0.5$ and $\beta =0.7$. Recall that this means we have the following in Armijo's condition: $f(x-\beta ^m\delta _0x)-f(x)\leq -\alpha \beta ^m\delta _0||\nabla f(x)||^2$, where $m\in \mathbb{Z}_{\geq 0}$ depends on $x$. Here we recall the essence of Unbounded and Two-way variants of Backtracking GD, see \cite{truong-nguyen1} for more detail. In the Two-way version, one starts the search for learning rate $\delta _n$ - at the step n- not at $\delta _0$ but at $\delta _{n-1}$, and allows the possibility of increasing $\delta \mapsto \delta /\beta $, and not just decreasing $\delta \mapsto \delta \beta$ as in the standard version of Backtracking GD. In the Unbounded variant, one allows the upper bound for $\delta _n$ not as $\delta _0$ but as $\max\{\delta _0,\delta _0||\nabla f(x_n)||^{-\kappa}\}$ for some constant $0<\kappa <1$. In all the experiments here, we fix $\kappa =1/2$. The Two-way version helps to reduce the need to do function evaluations in checking Armijo's condition, while the Unbounded version helps to make large step sizes near degenerate critical points and hence also helps with quicker convergence. 

{\bf Legends:} We use the following abbreviations:  "ACR" for Adaptive cubic regularisation, "BFGS" for itself, "Rand" for Random damping Newton method, "Newton" for Newton's method, "Iner" for Inertial Newton's method, "NewQ" for New Q-Newton's method, "R-NewQ" for Random New Q-Newton's method, and "Back" for Unbounded Two-way Backtracking GD.  

{\bf Features reported:} We will report on the number of iterations needed, the function value and the norm of the gradient at the last point, as well as the time needed to run. 

\subsubsection{A toy model for protein folding}
 This problem is taken from \cite{shh}. Here is a brief description of the problem. The model has only two amino acids, called A and B, among 20 that occurs naturally. A molecule with n amino acids will be called an n-mer. The amino acids will be linked together and determined by the angles of bend $\theta _2,\ldots ,\theta _{n-1}\in [0,2\pi ]$. We specify the amino acids by boolean variables $\xi _1,\ldots ,\xi _n\in \{1,-1\}$, depending on whether the corresponding one is A or B. The intramolecular potential energy is given by: 
\begin{eqnarray*}
\Phi =\sum _{i=2}^{n-1}V_1(\theta _i)+\sum _{i=1}^{n-2} \sum_{j=i+2}^n V_2(r_{i,j},\xi _i,\xi _j).
\end{eqnarray*}
Here $V_1$ is the backbone bend potential and $V_2$ is the nonbonded interaction, given by:  

\begin{eqnarray*}
V_1(\theta _i)&=&\frac{1}{4}(1-\cos (\theta _i)),\\
r_{i,j}^2&=&[\sum _{k=i+1}^{j-1}\cos (\sum _{l=i+1}^k\theta _l)]^2+[\sum _{k=i+1}^{j-1}\sin (\sum _{l=i+1}^k\theta _l)]^2,\\
C(\xi _i,\xi _j)&=&\frac{1}{8}(1+\xi _i+\xi _j+5\xi _i\xi _j),\\
V_2(r_{i,j},\xi _i,\xi _j)&=&4(r_{i,j}^{-12}-C(\xi _i,\xi _j)r_{i,j}^{-6}). 
\end{eqnarray*}
Note that the value of $C(\xi _i,\xi _j)$ belongs to the finite set $\{1,0.5,-0.5\}$. 

In the first nontrivial dimension $n=3$, we have $\Phi =V_1(\theta _2)+V_2(r_{1,3},\xi _1,\xi _3)$ and 
 $r_{1,3}=1$. Hence 
\begin{eqnarray*}
\Phi =\frac{1}{4}(1-\cos (\theta _2))+4(1-C(\xi _1,\xi _3)).
\end{eqnarray*}
Therefore, the global minimum (ground state) of $\Phi$ is obtained when $\cos (\theta _2)=1$, at which the value of $\Phi$ is $4(1-C(\xi _1,\xi _3))$. In the special case where $\xi _1=1=\xi _3$ (corresponding to  AXA), the global minimum of $\Phi$ is $0$. This is different from the assertion in Table 1 in \cite{shh}, where the ground state of $\Phi$ has value $-0.65821$ at $\theta _2=0.61866$. Our computations for other small dimensions cases $n=4,5$ also obtain values different from that reported in Table 1 in \cite{shh}. In \cite{shh} results are reported for dimension $\leq 5$, while those for dimensions 6 and 7 are available upon request. 

Table \ref{tab:ProteinFolding} presents the optimal values for the potential-energy function $\Phi$ for  molecules n-mer, where $n\leq 5$, founded by running different optimization methods from many random initial points. The cases listed here are the same as those in Table 1 in \cite{shh}. For comparison, we also compute the function value at the points listed in Table 1 in \cite{shh}. 

\begin{table}[htp]
\fontsize{11}{11}\selectfont
  \centering
  \begin{tabular}{|l|c|c|c|c|c|}
  \hline
 Molecule&$min ~\Phi$&$\theta _2/\pi$	& $\theta _3/\pi$  & $\theta _4/\pi$&Comparison with the point $\theta ^*=(\theta _2^*,\theta _3^*,\theta _4^*)$ in \cite{shh}  \\
\hline
AAA &0&0& &&$(0.6186,.,.)\pi $, $\Phi (\theta ^*)=0.3410$  \\
\hline
AAB &6&0& &&$(0,.,.)\pi $, $\Phi (\theta ^*)=6$  \\
\hline
ABA &0&0& &&$(0.6186,.,.)\pi $, $\Phi (\theta ^*)=0.3410$  \\
\hline
ABB &6&0& &&$(0,.,.)\pi $, $\Phi (\theta ^*)=6$  \\
\hline
BAB &2&0& &&$(0,.,.)\pi $, $\Phi (\theta ^*)=2$  \\
\hline
BBB &2&0& &&$(0,.,.)\pi $, $\Phi (\theta ^*)=2$  \\
\hline
AAAA &-0.0615&0&0 &&$(0.6183, 0.3392,.)\pi$, $\Phi (\theta ^*)=0.3226$  \\
\hline
AAAB &6.0322&0& 0&& $(0.6175, -0.0513,.)\pi$, $\Phi (\theta ^*)=6.3763$ \\
\hline
AABA &5.3417&0&0.6186 &&$(0.3327, 0.6218,.)\pi$, $\Phi (\theta ^*)=5.4681$   \\
\hline
AABB &12.0322&0&0 && $(0, 0,.)\pi$, $\Phi (\theta ^*)=12.0322$ \\
\hline
ABAB &2.0322&0&0 && $(0.6176, -0.06667,.)\pi$, $\Phi (\theta ^*)=2.3790$ \\
\hline
ABBA &11.3417&0&-0.6186 && $(0.4769, 0.4769,.)\pi$, $\Phi (\theta ^*)=12.0995$ \\
\hline
ABBB &8.0322&0&0&& $(0, 0,.)\pi$, $\Phi (\theta ^*)=8.0322$ \\
\hline
BAAB &11.9697&0&0&& $(0, 0,.)\pi$, $\Phi (\theta ^*)=11.9697$ \\
\hline
BABB &7.9697&0&0&& $(0, 0,.)\pi$, $\Phi (\theta ^*)=7.9697$  \\
\hline
BBBB &3.9697&0&0&& $(0.5582, 0.3518,.)\pi$, $\Phi (\theta ^*)=4.3577$  \\
\hline
AAAAA &-1.6763&0&0.6183&0.3392& $(0.3359, 0.6202, 0.0454)\pi$, $\Phi (\theta ^*)=-0.7042$  \\
\hline
AAAAB &5.4147&0&0.6176& -0.0513& $(0.6189, 0.3374, -0.0689)\pi$, $\Phi (\theta ^*)=6.3677$\\
\hline
AAABA &4.5490&0&0.3326&0.6218& $(0.2972, 0.3330, 0.6217)\pi$, $\Phi (\theta ^*)=4.6503$ \\
\hline
AAABB &12.0672&0&0&0&$(0.6175, -0.0537, -0.0016)\pi$, $\Phi (\theta ^*)=12.4117$ \\
\hline
AABAA &10.3236&0&0.6183&0.3392&$(0.3294, 0.6235, 0.0455)\pi$, $\Phi (\theta ^*)=11.2914$ \\
\hline
AABAB &7.4147&0&0.6176&-0.0513& $(0.3326, 0.6213, -0.5457)\pi$, $\Phi (\theta ^*)=8.3433$\\
\hline
AABBA &16.5490&0&0.3326&0.6218& $(0.1672, 0.4822, 0.4732)\pi$, $\Phi (\theta ^*)=17.4098$\\
\hline
AABBB &14.0672&0&0&0& $(0, 0, 0)\pi$, $\Phi (\theta ^*)=14.067$\\
\hline
ABAAB &11.3506&0&-0.6176&1.2066& $(0.6222, 0.3311, -0.0630)\pi$, $\Phi (\theta ^*)=12.3050$\\
\hline
ABABA &2.0589&0&0&0&$(0.6190, 0.04739, 0.6190)\pi$, $\Phi (\theta ^*)=4.5373$ \\
\hline
ABABB &8.0047&0&0&0& $(0.6176, -0.0710, -0.0022)\pi$, $\Phi (\theta ^*)=8.3525$\\
\hline
ABBAB &13.3506&0&0.6176&-0.0667&$(0.4788, 0.4734, -0.1418)\pi$, $\Phi (\theta ^*)=14.1068$\\
\hline
ABBBA &13.9638&0&-0.4768&-0.4768&$(0.2457, 0.5555, 0.2457)\pi$, $\Phi (\theta ^*)=14.8761$\\
\hline
ABBBB &10.0047&0&0&0&$(0.0548,  -0.3423, -0.5617)\pi$, $\Phi (\theta ^*)=10.9039$\\
\hline
BAAAB &12.0617&0&0&0&$(0.0392,  -0.6167, 0.0392)\pi$, $\Phi (\theta ^*)=14.1842$\\
\hline
BAABB &17.9992&0&0&0&$(0,  0, 0)\pi$, $\Phi (\theta ^*)=17.9992$\\
\hline
BABAB &4.0617&0&0&0&$(0.0532,  -0.6168, 0.0532)\pi$, $\Phi (\theta ^*)=6.1938$\\
\hline
BABBB &9.9992&0&0&0&$(0.5692, 0.3357, 0.2665)\pi$, $\Phi (\theta ^*)=10.4814$\\
\hline
BBABB &13.8602&0&-0.5582&-0.3518&$(0.3177, 0.5764, 0.0973)\pi$, $\Phi (\theta ^*)=14.1087$\\
\hline
BBBBB &5.8602&0&-0.5582&-0.3518&$(0.3434, 0.5650, 0.0931)\pi$, $\Phi (\theta ^*)=6.1185$\\
\hline

  \end{tabular}
   \caption{Optimal values for the potential-energy function $\Phi$ for $n$-mers, where $n=3, 4, 5$.}%
  \label{tab:ProteinFolding}
\end{table}

Here we will perform experiments for two cases: ABBBA (dimension 5) and ABBBABABAB (dimension 10). The other cases (of dimensions 5 and 10) yield similar results. We will generate random initial points and report on the performance of the different algorithms.  We observe that the performance of Inertial Newton's method and Adaptive Cubic Regularization are less stable or more slow than the other methods. 

1) {\bf For ABBBA:}  In this case the performance of New Q-Newton's method and of Random New Q-Newton's method are very similar, so we report only that of New Q-Newton's method. We found that the optimal value seems to be about $13.963$. 

We will test for several (random) choices of initial points:  $$(\theta _2,\theta _3,\theta _4)=(-0.0534927, 1.61912758, 2.9567358),$$ with function value $2555432869.1351156$; 

$$(\theta _2,\theta _3,\theta _4)=(1.80953527, -1.74233202,  2.45974152),$$
with function value 538.020;

and 

$$(\theta _2,\theta _3,\theta _4)=(1.07689387, 2.97081771, 0.800213082), $$ 
with function value 6596446021.145492.

Table \ref{tab:ProteinFolding5} lists the performance of different methods (with a maximum number of 5000 iterates, but can stop earlier if $||\nabla f(z_n)||<1e-10$ or $||z_{n+1}-z_n||<1e-20$ or there is an unknown error): 

\begin{table}[htp]
\fontsize{11}{11}\selectfont
  \centering
  \begin{tabular}{|l|c|c|c|c|c|c|c|}
  \hline
 &ACR&BFGS	& Newton   & NewQ  & Rand& Iner& Back\\
\hline
&\multicolumn{7}{c|}{Initial point (-0.0534927, 1.61912758, 2.9567358)}\\
\hline
Iterations &7&57& 17 & 31 &31& 14&269 \\
\hline
$f$ &5e+6&14.058&3e+5 & 13.963 &3e+5 &  14.255&13.963 \\
\hline
$||\nabla f||$ &1e+8&1e-8& 6e-6 & 5e-12 &6e-6 & 0&7e-7 \\
\hline
Time &0.058&0.843& 0.337 & 0.617 &0.594 & 0.078&6.144 \\
\hline
&\multicolumn{7}{c|}{Initial point (1.80953527, -1.74233202,  2.45974152)}\\
\hline
Iterations &5&26& 27 & 15 &51& 13&24 \\
\hline
$f$ &14.117&13.963&13.963 & 13.963 &14.463 &  5e+4&14.058 \\
\hline
$||\nabla f||$ &47.388&6e-11& 4e-12 & 8e-12 &4e-10 & 0&1e-8 \\
\hline
Time &0.114&0.1773& 0.541 & 0.317 &1.033 & 0.084&0.628 \\
\hline
&\multicolumn{7}{c|}{Initial point (1.07689387, 2.97081771, 0.800213082)}\\
\hline
Iterations &19&57& 32 & 48 &32& 15&38 \\
\hline
$f$ &283.822&13.963&13.963 & 13.963 &13.963 &  39.726&14.058 \\
\hline
$||\nabla f||$ &3950.996&1e-10& 1e-11 & 5e-10 &4e-10 & 0&8e-10 \\
\hline
Time &2.760&0.398& 0.626 & 0.642 &0.928 & 0.085&0.938 \\
\hline

  \end{tabular}
   \caption{Performance of different optimization methods for the toy protein folding problem for the 5-mer ABBBA at some  random initial points. The function values at the initial points are respectively 2555432869.1351156; 538.020; and 6596446021.145492.}%
  \label{tab:ProteinFolding5}
\end{table}

2) {\bf For ABBBABABAB:} In this case, usually Newton's method and Random damping Newton's method encouter the error "Singular matrix". Hence, we have to take a more special care of them and reduce the number of iterations for them to 50. In this case, Random New Q-Newton's method can obtain better performances than New Q-Newton's methods, so we report both of them. In this case, it seems that the optimal value is about $19.387061837218972$, which is obtained near the point 
\begin{eqnarray*}
&&(\theta _2,\theta _3,\theta _4, \theta _5,\theta _6,\theta _7, \theta _8)\\
&=&(-4.7735907, -0.47766515,   -1.02890588,  -1.77319053,\\
&&-0.02340005, 0.08208585, -1.39102817,  0.27906532).  
\end{eqnarray*}
{\bf Remark.} We have tested with many random initial points, and found that none of the algorithms here (Adaptive Cubic Regularization, BFGS, Newton's method, New Q-Newton's method, Random Newton's method, Random New Q-Newton's method, Inertial Newton's method, and Backtracking GD) can find the above global minimum. The value has been found by running New Q-Newton's method Backtracking \cite{truong2021} with for example Point 1 below, with running time about 16.2 seconds.

We will test with 4 random initial points (see Table \ref{tab:ProteinFolding10}): 

Point 1 

\begin{eqnarray*}
&&(\theta _2,\theta _3,\theta _4, \theta _5,\theta _6,\theta _7, \theta _8)\\
&=&(-3.00156524, -1.5427558,   1.9394472,  -2.74672374,\\
&&-1.82664375, 1.96928115, -1.26350718,  2.82317321).  
\end{eqnarray*}
The function value at the initial point is $4185029.6878152043$. 

Point 2: 

\begin{eqnarray*}
&&(\theta _2,\theta _3,\theta _4, \theta _5,\theta _6,\theta _7, \theta _8)\\
&=&(1.50386159, -1.36306552,  2.93979824,  1.01082799,\\
&&-1.56261475,  1.61429959,
 -0.02311273, -1.8108999).  
\end{eqnarray*}
The function value at the initial point is $895386751.0677216$. 

Point 3: 

\begin{eqnarray*}
&&(\theta _2,\theta _3,\theta _4, \theta _5,\theta _6,\theta _7, \theta _8)\\
&=&(2.89936055,  2.5913901,  -1.40975004, -2.76032304,\\
&&-3.05060738,  1.09171554,
  1.33525563, -1.85212602).  
\end{eqnarray*}
The function value at the initial point is $12479713199090.754$.

Point 4: 

\begin{eqnarray*}
&&(\theta _2,\theta _3,\theta _4, \theta _5,\theta _6,\theta _7, \theta _8)\\
&=&(-1.3335047,   2.76782837, -1.89518385,  2.52345111,\\
&&-0.33519698, -1.98794015,
  0.02088706, -1.09200044).  
\end{eqnarray*}
The function value at the initial point is $579425.218039767$.

\begin{table}[htp]
\fontsize{11}{11}\selectfont
  \centering
  \begin{tabular}{|l|c|c|c|c|c|c|c|c|}
  \hline
&ACR&BFGS	& Newton   & NewQ  & Rand&R-NewQ& Iner& Back\\
\hline
&\multicolumn{8}{c|}{Initial point: Point 1}\\
\hline
Iterations &1e+4&197&50  & 35 &50&35& 13&104 \\
\hline
$f$ &7e+7&19.707&Err & 1.2e+4 &Err & 1.2e+4& 2e+7&20.225 \\
\hline
$||\nabla f||$ &1e+10& 6e-10 & Err &8e-8 & Err&8e-8&0&1e-7 \\
\hline
Time &395.49&14.27&Err & 16.20 &Err& 16.24&0.500&34.982 \\
\hline
&\multicolumn{8}{c|}{Initial point: Point 2}\\
\hline
Iterations &66&79& 50 & 70 &47&70& 13&34 \\
\hline
$f$ &5e+11&19.596&Err & 20.151 &20.207 & 20.151& 5e+6&20.147 \\
\hline
$||\nabla f||$ &5e+13& 5e-8 & Err&1e-7&4e-8&1e-7  &0&8e-9 \\
\hline
Time &14.17&4.118& Err& 32.76 & 21.47&32.35& 0.479&11.768 \\
\hline
&\multicolumn{8}{c|}{Initial point: Point 3}\\
\hline
Iterations &0&176& 50 & 500 &50&500 & 13&500 \\
\hline
$f$ &1e+13&19.727&Err & 20.225 &Err& 20.147& 3e+9&3e+3 \\
\hline
$||\nabla f||$&1e+15& 7e-9 &Err  &2e-5 &Err& 2e-8 &0&92.72\\
\hline
Time &0&9.91& Err& 380.1 &Err & 235.6&0.484&201.9 \\
\hline
&\multicolumn{8}{c|}{Initial point: Point 4}\\
\hline
Iterations &1&83& 50 & 95&50&95& 14&65 \\
\hline
$f$ &2e+20&19.596&Err & 3e+3 &Err & 3e+3&  7e+4&20.225 \\
\hline
$||\nabla f||$ &1e+7& 3e-8 & Err &2e-8 &Err&2e-8 & 0&2e-7\\
\hline
Time &2.301&4.365&Err & 43.55 & Err&43.64& 0.583&21.996 \\
\hline

  \end{tabular}
   \caption{Performance of different optimization methods for the toy protein folding problem for the 10-mer ABBBABABAB at several  random initial points. The function values at the initial points are respectively: 4185029.6878152043; 895386751.0677216; 12479713199090.754 and 579425.218039767. For Newton's method and Random damping Newton's method: we oftenly encounter singular matrix error.}%
  \label{tab:ProteinFolding10}
\end{table}

\subsubsection{Griewank problem - The deterministic case} This is a well known test function in global optimization. It has the form: 
\begin{eqnarray*}
f(x_1,\ldots ,x_m)=1+\frac{1}{4000}\sum _{i=1}^mx_i^2-\prod _{i=1}^m\cos (x_i/\sqrt{i}). 
\end{eqnarray*}
It has a unique global minimum at the point $(0,\ldots ,0)$, where the function value is 0. The special property of it is that, in terms of the dimension $m$,  it has exponentially many local minima. However, \cite{locatelli} explained that indeed when the dimension increases, it can become more and more easier to find the global minimum. We present here some experiments with $m=15$. 

Table \ref{tab:Griewank1} presents the performance at 2 initial points:

Point 1: $(10,\ldots ,10)$ (which was the choice of \cite{kk} in the stochastic setting, see the next subsection). The function value at the initial point is $1.364$. 

Point 2 (randomly chosen): 

\begin{eqnarray*}
&&(-0.24657266, -5.45285145, -0.92531932, -5.68778641,  1.64861456,\\  &&5.65718487,
 -6.17919738,  2.95625737, -6.47274618, -0.47513139,\\
 &&-8.60344445,  0.74612203,
  3.70371132, -6.39595989,  7.5908029
).\end{eqnarray*} The function value at the initial point is $1.092$. 

\begin{table}[htp]
\fontsize{11}{11}\selectfont

  \centering
  \begin{tabular}{|l|c|c|c|c|c|c|c|c|}
  \hline
 &ACR&BFGS	& Newton   & NewQ  & Rand& R-NewQ& Iner& Back\\
\hline
&\multicolumn{8}{c|}{Initial point: Point 1}\\
\hline

Iterations &1e+4&62&5  & 7 &46&7& 13& 54\\
\hline
$f$ &1.234&0.054&0& 0 &1.002 &0 & 6e+29 &0.380 \\
\hline
$||\nabla f||$ &0.015& 6e-9 &0  & 7e-15&8e-12&7e-15 &0 &2e-8\\
\hline
Time &311&2.733& 1.492& 2.111& 14.022&2.070 & 0.254&17.488 \\
\hline

&\multicolumn{8}{c|}{Initial point: Point 2}\\
\hline

Iterations &1e+4&50&6 & 7 &108&7& 14& 230\\
\hline
$f$ &0.966&0&0& 0 &1.004 &0 & 5e+30 &0 \\
\hline
$||\nabla f||$ &0.152& 2e-9 &0  & 0&3e-11&0 &0 &5e-9\\
\hline
Time &228&1.766& 2.766& 2.028& 34.601&2.772 & 0.270&68.730 \\
\hline

  \end{tabular}
   \caption{Performance of different optimization methods for the Griewank test function in dimension $15$ at 2 initial points. Point 1 is  $(10,\ldots ,10)$, and Point 2 is randomly chosen. The function values at the initial points are respectively 1.364 and 1.092.} %
  \label{tab:Griewank1}
\end{table}

\subsubsection{Griewank problem - The stochastic case} Here, we consider the stochastic version of Griewank problem. It is to illustrate how New Q-Newton's method performs in the stochastic setting, which is more relevant to the set up in realistic DNN.  

We briefly recall some generalities of stochastic optimization. One  considers a function $f(x,\xi )$, which besides a variable $x$, also depends on a random parameter $\xi $. One wants to optimize the expectation of $f(x,\xi )$: Find $\min _{x}F(x)$, where $F(x)=E(f(x,\xi ))$. 

Assume now that one has an optimization method $A$ to be used to the above problem. Since computing the expectation is unrealistic in general, what one can do is as follows: 

At step n: choose randomly $N_n$ parameters $\xi _{n,1},\ldots ,\xi _{n,N}$, where $N_n$ can depend on $n$ and is usually chosen to be large enough. Then one approximates $F(x)$ by 
\begin{eqnarray*}
F_n(x)=\frac{1}{N_n}\sum _{i=1}^{N_n}f(x,\xi _{n,i}). 
\end{eqnarray*}

This is close to the mini-batch practice in Deep Learning, with a main difference is that in Deep Learning one trains a DNN on a large but finite set of data, and at each step n (i.e. epoch n) decompose the training set randomly into mini-batches to be used. The common practice is to use mini-batches of the same size $N_n=N$ fixed from beginning. There are also experiments with varying/increasing the mini-batch sizes, however this can be time consuming while not archiving competing results. There are also necessary modifications (such as rescaling of learning rates) in the mini-batch practice to obtain good performance, however in the experiments below we do not impose this to keep things simple.   

In this subsection we perform experiments on the stochastic version of the Griewank test function considered in the previous subsection. This problem was considered in \cite{kk}, where the dimension of $x=(x_1,\ldots ,x_m)$ is $m=10$ and of $\xi$ is $1$ (with the normal distribution $N(1,\sigma ^2)$), and $f(x,\xi)$ has the form: 
\begin{eqnarray*}
f(x,\xi )=1+\frac{1}{4000}||\xi x||^2-\prod _{i=1}^m\cos (x_i \xi /\sqrt{i}). 
\end{eqnarray*}
At each step, \cite{kk} chooses $N_n$ varying in an interval $[N_{min},N_{max}]$ according to a complicated rule. Here, to keep things simple, we do not vary $N_n$ but fix it as a constant from beginning. Also, we do not perform experiments on BFGS and Adaptive Cubic Regularization in the stochastic setting, because the codes of these algorithms are either not available to us or too complicated and depend on too many hyperparameters (and the performance is very sensitive to these hyperparameters) to be succesfully changed for the stochastic setting. We note however that the BFGS was tested in \cite{kk}, and interested readers can consult Table 8 in that paper for more detail. 

The settings in \cite{kk} are as follows: the dimension is 10, the $\sigma$ is chosen between 2 values $\sqrt{0.1}$ (with $N_{max}=500$) and $\sqrt{1}$ (with $N_{max}=1000$). We will also use these parameters in the below, for ease of comparison. We note that in \cite{kk}, time performance is not reported but instead the number of function evaluations (about 1.8 million for $\sigma=\sqrt{0.1}$, and about 6.3 million for $\sigma=\sqrt{1}$). Also, only the norm of gradient was reported in \cite{kk}, in which case it is ranged from $0.005$ to $0.01$.

\begin{table}[htp]
\fontsize{10}{10}\selectfont
  \centering
  \begin{tabular}{|l|c|c|c|c|c|c|}
  \hline
 	& Newton   & NewQ  & Rand& R-NewQ& Iner& Back\\
\hline
&\multicolumn{6}{c|}{$N_n=10$, $\sigma =\sqrt{0.1}$}\\
\hline

Iterations &1000  & 33 &1000&53& 14& 1000\\
\hline
$f$ &3.8e+18& 0 &5.4e+19 &0 & 2.2e+31 &1.079 \\
\hline
$||\nabla f||$  &6.2e+7  & 0&2.4e+8&0 &0 &0.008\\
\hline
Time & 587.853& 17.702& 586.482&26.820 & 0.667&612.408 \\

\hline
&\multicolumn{6}{c|}{$N_n=10$, $\sigma =\sqrt{1}$}\\
\hline

Iterations &1000  & 1000 &1000&1000& 13& 486\\
\hline
$f$ &5.3e+19& 1.7e+21 &9.3e+18 &5.15e+18 & 6.6e+29 &5.3e-14 \\
\hline
$||\nabla f||$  &2.9e+8  & 2.1e+9&1.3e+8&8.3e+7 &0 &1.1e-7\\
\hline
Time & 546.710& 503.274&507.936 &514.627 & 0.600&258.496 \\

\hline
&\multicolumn{6}{c|}{$N_n=100$, $\sigma =\sqrt{0.1}$}\\
\hline

Iterations &1000  & 1000 &1000&110& 13& 1000\\
\hline
$f$ &2.5e+17& 2.3e+18 &1e+18 &0 & 4.6e+29 &0.967 \\
\hline
$||\nabla f||$  &1.7e+7  & 5.2e+7&1e+8&0 &0 &0.002\\
\hline
Time & 3.7e+3& 1e+4& 3.6e+3&401 & 4.333&5.5e+3 \\
\hline

&\multicolumn{6}{c|}{$N_n=100$, $\sigma =\sqrt{1}$}\\
\hline

Iterations &1000 & 1000 &1000&1000& 13& 395\\
\hline
$f$ &6.3e+18& 7.8e+16 &2.1e+19 &8.5e+16 & 9.3e+29 &9.5e-13 \\
\hline
$||\nabla f||$  &1.1e+8  & 1.1e+8&1.9e+8&1.2e+7 &0 &6.3e-7\\
\hline
Time & 3.6e+3& 6.8e+3& 3.6e+3&5.9e+3 & 5.019&3.7e+3 \\
\hline

&\multicolumn{6}{c|}{$N_n=500$, $\sigma =\sqrt{0.1}$}\\
\hline

Iterations &1000 & 14 &1000&1000& 13& 1000\\
\hline
$f$ &8.6e+17& 0 &4.5e+18 &1.0e+19 & 4.4e+29 &0.964 \\
\hline
$||\nabla f||$  &3.1e+7  & 3.7e-16&6.9e+7&1.0e+18 &0 &0.014\\
\hline
Time & 1.0e+4& 142.838& 1.0e+4&1.0e+4 & 12.221&1.1e+4 \\
\hline

&\multicolumn{6}{c|}{$N_n=500$, $\sigma =\sqrt{1}$}\\
\hline

Iterations &1000 & 1000 &1000&17& 14& 361\\
\hline
$f$ &2.6e+18& 9.8e+18 &7.6e+18 &0 & 3e+21 &6.6e-13 \\
\hline
$||\nabla f||$  &7.3e+7  & 1.3e+8&3.8e+7&4e-16 &0 &9.9e-7\\
\hline
Time & 9.9e+3& 9.6e+3& 9.6e+3&160.256 & 12.324&3.7e+3 \\
\hline

&\multicolumn{6}{c|}{$N_n=1000$, $\sigma =\sqrt{0.1}$}\\
\hline

Iterations &1000 & 19 &1000&1000& 13& 1000\\
\hline
$f$ &2.1e+17& 0 &7.9e+16 &2.0e+17 & 4.6e+29 &0.945 \\
\hline
$||\nabla f||$  &1.5e+7  & 1.6e-16&9.4e+7&1.4e+7 &0 &0.003\\
\hline
Time & 20e+3& 365& 20e+3&19e+3 & 23.303&21e+3 \\
\hline

&\multicolumn{6}{c|}{$N_n=1000$, $\sigma =\sqrt{1}$}\\
\hline

Iterations &1000 & 1000 &1000&1000& 14& 347\\
\hline
$f$ &1.9e+20& 1.7e+18 &1.2e+18 &2.4e+17 & 3e+31 &2.5e-12 \\
\hline
$||\nabla f||$  &6.2e+8  & 5.9e+7&4.9e+7&2.2e+7 &0 &1.0e-6\\
\hline
Time & 20e+3& 19e+3& 20e+3&19e+3 & 25&7.3e+3 \\
\hline

  \end{tabular}
   \caption{ Performance of different optimization methods for the Griewank test function in the stochastic setting. The dimension is  $10$ and the initial point is  $(10,\ldots ,10)$. The function value of the deterministic Griewank test function $F(x)=E(f(x,\xi))$ at the initial point is $1.264$. Mini-batch size $N_n$ is fixed in every steps of each experiment.} %
  \label{tab:Griewank2}
\end{table}

\subsubsection{Finding roots of univariate meromorphic functions}\label{SectionExperimentMeromorphic} As discussed in Section \ref{SectionRootMeromorphicFunction}, given a non-constant univariate function $g(z)$, we will construct a function $f(x,y)=u(x,y)^2+v(x,y)^2$, where $z=x+iy$, $u=$ the real part of $g$ and $v=$ the imaginary part of $g$. Global minima of $f$ are exactly roots of $g$, at which the function value of $f$ is precisely $0$. We will apply different optimization algorithms to $f$.  See  Table \ref{tab:RootMeromorphic}. 

We will consider a tricky polynomial \cite{delves-lyness}, for which Lehmer's method encountered errors:  
\begin{eqnarray*}
g_1(z)&=&1250162561z^{16}+385455882z^{15}+845947696z^{14}+240775148z^{13}\\
&&+247926664z^{12}+64249356z^{11}+41018752z^{10}+9490840z^9\\
&&+4178260z^{18}+837860z^7+267232z^6+44184z^5\\
&&+10416z^4+1288z^3+242z^2+16z+2.
\end{eqnarray*}
The (randomly chosen) initial point is $(x,y)$ $=$ $(6.58202917,-7.93929341)$, at which point the function value of $f$ is $4e+50$. 

We will consider a simple function, for which the point $(0,0)$ is a saddle point of the function $f$:
\begin{eqnarray*}
g_2(z)=z^2+1. 
\end{eqnarray*}
We look at 2 (random initial) points. Point 1: $(x,y)$ $=$ $(4.0963223, -8.0935966)$, at which point the value of $f$ is $6482$. Point 2: (closer to the point $(0,0)$): $(x,y)=(0.317,-0.15)$, at which point the function value  of $f$ is $1.171$. 

We will consider a meromorphic function, which is the derivative of the function in formula (7.4) in \cite{delves-lyness}:
\begin{eqnarray*}
g_3(z)=\frac{d}{dz}[\frac{1-1.005e^{-z}+0.525e^{-2z}-0.475e^{-3z}-0.045e^{-4z}}{2.27e^{-z}-2.19e^{-2z}+1.86e^{-3z}-0.38e^{-4z}}] 
\end{eqnarray*}
The root of smallest absolute value of $g_3$ is near to $0.3430042+1.0339458i$. It has a pole near $-0.227+1.115i$ of absolute value just slightly larger than that of this root, and hence when one applies the method in \cite{delves-lyness} one has to be careful. We choose (randomly) an initial point which is close to the pole of $g_3$: $(x,y)=(-0.227,1.115)$, at which point the value of $f$ is 0.0415.   

We will consider a polynomial function with multiple roots: 
\begin{eqnarray*}
g_4(z)=z (z-1)^2(z-2)^3(z-5)^5. 
\end{eqnarray*}
We consider a (random) initial point $(x,y)=(4.48270522, 3.79095724)$, at which point the function value is  $1e+14$. 

We will consider the $101$-th summand of the series defining Riemann zeta function: 
\begin{eqnarray*}
g_5(z)=\sum _{n=1}^{101}n^{-z}. 
\end{eqnarray*}
Here, recall that $n^{-z}=e^{-ln(n)z}$. We choose an (randomly chosen) initial point $$(x,y)=(- 8.5209648, 1.28480016),$$ at which the function value is $1e+36$.

We will consider the $1001$-th summand of the series defining Riemann zeta function: 
\begin{eqnarray*}
g_6(z)=\sum _{n=1}^{1001}n^{-z}. 
\end{eqnarray*}
Here, recall that $n^{-z}=e^{-ln(n)z}$. We choose an (randomly chosen) initial point $$(x,y)=(9.76536427, -4.15647151),$$ at which the function value is $0.9977$.

\begin{table}[htp]
\fontsize{11}{11}\selectfont

  \centering
  \begin{tabular}{|l|c|c|c|c|c|c|c|c|}
  \hline
&ACR&BFGS	& Newton   & NewQ  & Rand&R-NewQ& Iner& Back\\
\hline
&\multicolumn{8}{c|}{Function $g_1$}\\
\hline
Iterations &Err&Err& 149 &149 &149&149&Err & Err\\
\hline
$f$ &Err&Err&6e-14 &6e-14  &6e-14 &6e-14 &Err & Err\\
\hline
$||\nabla f||$ &Err& Err &9e-11  &9e-11 &9e-11 &9e-11&Err&Err \\
\hline
Time &Err&Err& 2.076& 1.935 &1.922&1.959 &Err& Err\\
\hline
&\multicolumn{8}{c|}{Function $g_2$, Point 1}\\
\hline
Iterations &0&25&11 & 11&33&11& 4&14 \\
\hline
$f$ &6482&1e-23& 1e-39& 1e-40& 8e-22&1e-40&3e+78 &1e-22 \\
\hline
$||\nabla f||$ &2900& 1e-11 & 0 &0 &qe-10 &0&0&4e-11 \\
\hline
Time &0.002&0.107& 0.112& 0.112 &0.331&0.113&0.015&0.195 \\
\hline
&\multicolumn{8}{c|}{Function $g_2$, Point 2}\\
\hline
Iterations &4&10& 5& 9&19&9&6 &11 \\
\hline
$f$ &1e-10&4e-24&1 &3e-43 & 1&3e-43&2e+160 & 1e-24\\
\hline
$||\nabla f||$ &4e-5&  8e-12& 0 &0 &9e-10 &0&0&4e-12 \\
\hline
Time &0.014&0.062& 0.051&0.094  &0.188&0.092&0.020&0.148 \\
\hline
&\multicolumn{8}{c|}{Function $g_3$}\\
\hline
Iterations &Err&1& 13&18 &Err&18& Err&1 \\
\hline
$f$ &Err&0.040& 0.387&5e-28 &Err &5e-28&Err &0.040 \\
\hline
$||\nabla f||$ &Err&0.205  & 6e-10 & 3e-14&Err &3e-14&Err& 0.779\\
\hline
Time &Err&16.77& 15.43& 22.06 &Err&21.48&Err&3.810 \\
\hline
&\multicolumn{8}{c|}{Function $g_4$}\\
\hline
Iterations &46&132&56 & 56&54&56&Err & 405\\
\hline
$f$ &2e-9&8e-15&2e-14 &2e-14 &2e-14 &2e-14&Err &1e-13 \\
\hline
$||\nabla f||$ &7e-7& 8e-11 & 2e-11 & 2e-11&2e-11 &2e-11&Err&9e-11 \\
\hline
Time &0.159&0.558&0.572 & 0.578 &0.547&0.578&Err&5.358 \\
\hline
&\multicolumn{8}{c|}{Function $g_5$}\\
\hline
Iterations &95&2&111 &89 &107&89& Err& 71\\
\hline
$f$ &6e-11&Err&1 &1e-28 & 1&1e-28&Err &5e-23 \\
\hline
$||\nabla f||$ &3e-5& Err & 3e-11 & 1e-13& 4e-12&1e-13&Err&6e-11 \\
\hline
Time &4.242&8.656&16.40 &  13.39&15.87&13.39&Err&14.45 \\
\hline
&\multicolumn{8}{c|}{Function $g_6$}\\
\hline
Iterations &Err&2&18&46 &16&46& Err& 103\\
\hline
$f$ &Err&Err&0.9999 &1e-30 & 0.9999&1e-30&Err &6e-21 \\
\hline
$||\nabla f||$ &Err& Err & 2e-11 & 3e-14& 4e-11&3e-14&Err&7e-10 \\
\hline
Time &Err&79.04&23.55 &  59.94&20.85&60.05&Err&180.3 \\
\hline

  \end{tabular}
   \caption{Performance of different optimization methods for finding roots of meromorphic functions at random initial points. See Section \ref{SectionExperimentMeromorphic} for more detail. "Err" means some errors encountered.}%
  \label{tab:RootMeromorphic}
\end{table}

\subsection{Conclusions and Future work} 
\label{Subsection7}

In this paper, we proposed a new modification of Newton's method,  named New Q-Newton's method, and showed that it can avoid saddle points. Hence, in contrast to all existing versions of Newton's method in the literature, our New Q-Newton's method can be used for the purpose of finding local minima. We obtain the result by adapting the arguments in \cite{truong}, compare Subsection \ref{Subsection4}. We demonstrated the good performance of this method on various benchmark examples, against the algorithms Newton's method, BFGS, Random damping Newton's method and Inertial Newton's method. We also find that the random version of New Q-Newton's method (when the parameters $\delta _0,\ldots ,\delta _m$ are not fixed from beginning, but are randomly chosen at each iteration) can be easier to use while having similar or better performance as New Q-Newton's method. 

{\bf Open questions:} It is an open question of whether the condition that $f$ is $C^3$ is needed in Theorem \ref{TheoremMain}, or $C^2$ is enough. It is also an open question of whether part 2) of Theorem \ref{TheoremMain} also holds, even in the more general setting where $\nabla ^2f(x_{\infty})$ is not invertible. Experiments in the previous Subsection seem to indicate that this is the case. 

On the one hand, New Q-Newton's method has the same rate of convergence as the usual Newton's method, and hence is better than all GD (including Backtracking GD). On the other hand, unlike Backtracking GD \cite{truong-nguyen1, truong-nguyen2}, we still do not have a result guaranteeing convergence for New Q-Newton. Additionally, readers can easily check that New Q-Newton's method, when applied to functions, such as $f(x)=|x|$, which are not $C^2$ and whose Hessian is identically $0$, can diverge - even though the function has compact sublevels. This has been resolved in recent work by first author \cite{truong2021}, where Backtracking line search is incorporated into New Q-Newton's method. We obtain in particular the best theoretical guarantee for Morse cost functions, among all iterative optimization algorithms in the current literature, see Theorem \ref{TheoremMorse}.  

Analysing the proof of Theorem \ref{TheoremMain}, we see that only the facts that  the map $x\mapsto ||\nabla f(x)||^{1+\alpha}$ is $C^1$ near critical points of $f$ and in general locally Lipschitz continuous are needed. Therefore, Theorem \ref{TheoremMain}, and hence also Corollary \ref{CorollaryMain}, is valid for the following generalisation of New Q-Newton's method: 

{\bf Generalised New Q-Newton's method: } Let $\Delta =\{\delta _0,\delta _1,\delta _2,\ldots \}$ be a countable set of real numbers which has at least $m+1$ elements. Let $f:\mathbb{R}^m\rightarrow \mathbb{R}$ be a $C^2$ function. Let $h:[0,\infty ) \rightarrow \mathbb{R}$ be a function such that: i) $h(t)=0$ iff $t=0$,  ii) $h$ is $C^1$ near $t=0$, and iii) $h$ is locally Lipschitz continuous. For each $x\in \mathbb{R}^m$ such that $\nabla f(x)\not=0$, let $\delta (x)=\delta _j$, where $j$ is the smallest  number so that $\nabla ^2f(x)+\delta _jh(||\nabla f(x)||)Id$ is invertible. (If $\nabla f(x)=0$, then we choose $\delta (x)=\delta _0$.) Let $x_0\in \mathbb{R}^m$ be an initial point. We define a sequence of $x_n\in \mathbb{R}^m$ and invertible and symmetric $m\times m$ matrices $A_n$ as follows: $A_n=\nabla ^2f(x_n)+\delta (x_n) h(||\nabla f(x_n)||)Id$ and $x_{n+1}=x_n-w_n$, where $w_n=pr_{A_n,+}(v_n)-pr_{A_n,-}(v_n)$ and $v_n=A_n^{-1}\nabla f(x_n)$. 

One could choose $h(t)$ to be bounded, such as $h(t)=\min \{1,t^{1+\alpha}\}$, so that the perturbation $\delta (x)h(||\nabla f(x)||) Id$ is not too big when $||\nabla f(x)||$ is too big. We tested the experiments in the previous subsection with such bounded functions, and obtained similar results.  

The orthogonal diagonalization of real symmetric matrices needed in New Q-Newton's method is expensive when the dimension $m$ is large. The research in this topic is very extensive. Among some common such methods we find (the readers can find more information in the corresponding Wikipedia pages): the QR algorithm \cite{francis1, francis2} whose cost is $O(m^3)$, the Jacobi eigenvalue algorithm \cite{jacobi} whose cost is also $O(m^3)$, and the Divide-and-conquer eigenvalue algorithm \cite{cuppen} whose cost is again $O(m^3)$ - where $m$ is the dimension (however, the precise constant multiples involved are different). Hence, more work is needed to implement this method into huge scale optimisation problems such as in DNN. We are exploring this in an ongoing work. We note that an implementation, for the folklore heuristic version for some simple DNN or for the simple dataset MNIST has been given in \cite{dauphin-pascanu-gulcehre-cho-ganguli-bengjo}, which could be useful for the task of implementing in deeper DNN and for more difficult datasets and tasks. A more large scale implementation for the paper \cite{dauphin-pascanu-gulcehre-cho-ganguli-bengjo} is recently available on GitHub \cite{SFO}, which has some important differences to the algorithm proposed in our paper. There is also a problem of how to extend New Q-Newton's method to the infinite dimensional setting, so to obtain an analog of results in \cite{truong4} for Banach spaces. To this end, we note that tools needed (Morse's lemma and integral formula of projection on eigenspaces of linear operators) in the proof of Theorem \ref{TheoremMain} are available on Banach spaces \cite{palais, kato}, however there are still many differences between the finite and infinite dimensional spaces which hinder extending the proof of Theorem \ref{TheoremMain} to the infinite dimensional setting. 

Finally, we comment about the usefulness of implementations of Newton's method and its modifications in Deep Neural Networks (DNN). There are at least 2 issues. The first issue concerns saddle points. Since cost functions in DNN involves a lot of variables (for state of the art networks, we could have hundreds of million) and since generically the ratio between saddle points and minima of these cost functions grows exponentially \cite{bray-dean, dauphin-pascanu-gulcehre-cho-ganguli-bengjo}, we expect that a random initial point $x_0$ will most of the time close to a saddle point. Since Newton's method has the tendency of converging to the critical point nearest to the initial point, we expect that most of the time Newton's method will converge to saddle points of the cost functions appearing in DNN. Therefore, Newton's method per se is of limited usefulness, if the goal is to find local minima of the cost functions. On the other hand, it can be, because of its fast convergence when close to a minimum, for example, combined with Backtracking GD (whose convergence to local minima is guaranteed theoretically in generic situations). The same comment applies to modifications of Newton's method which have the same tendency of converging to the critical point nearest to the initial point. Of course, this comment does not apply to modifications of Newton's methods, such as our New Q-Newton's method, which are theoretically proven to avoid saddle points and to converge fast in generic situations. The second issue one faces when implementing Newton's method and modifications into DNN: Experiments in the previous Subsection show that Newton's method and its modifications could have a problem of convergence when the cost function is not $C^2$. We note that Backtracking GD has, on the other hand, better convergence properties. The combination between Backtracking line search and New Q-Newton's method, as proposed in \cite{truong2021}, helps to resolve the convergence issue as well.       

As mentioned in Section \ref{SubsectionLargeScale}, in any event, currently we are not aware of any implementation of Newton's method or variants in Deep Neural Networks that can compete with Gradient Descent and variants (including Backtracking Gradient Descent) - in particular on important indicators such as validation accuracy or running time. This is besides the fact that variants of Backtracking Gradient Descent have the best theoretical guarantee in the current literature, see Theorem \ref{Theorem1}. Therefore, having a new variant of Newton's method such as New Q-Newton's method (and New Q-Newton's method Backtracking), with a simple framework and implementation, working well on small scale (see the experimental results reported in the appendix)  while having good theoretical guarantees and applicable in general settings, can be beneficial and hence worth further study.

\newpage
\section{Appendix: Some experimental results on benchmark test functions}
\label{Subsection6}

In this appendix we will compare the performance of New Q-Newton's method against the usual Newton's method, BFGS \cite{Newton2} and Section 2.2 in \cite{bertsekas}, Adaptive Cubic Regularization \cite{nesterov-polyak, cartis-etal}, as well as Random damping Newton's method \cite{sumi}  and Inertial Newton's method \cite{bolte-etal}. Since in experiments in Table \ref{tab:Tasks}, the performance of Random damping Newton's method is always better or the same as the performance of the usual Newton's method, we report only the performance of Random damping Newton's method. 

In the experiments, we will use the Generalised New Q-Newton's method in Section \ref{Subsection7}, since it uses smaller quantities in general. We remark that if we use the basic version of New Q-Newton's method in Table \ref{table:alg} then we obtain similar results.  We choose $\alpha =1$ in the definition. Moreover, we will choose $\Delta =\{0,\pm 1\}$, even though for theoretical proofs we need $\Delta$ to have at least $m+1$ elements, where $m=$ the number of variables. The justification is that when running New Q-Newton's method it almost never happens the case that both $\nabla ^2f(x)$ and $\nabla ^2f(x)\pm ||\nabla f(x)||^2Id$ are not invertible. The experiments are coded in Python and run on a usual personal computer. For BFGS: we use the function scipy.optimize.fmin$\_$bfgs available in Python, and put  $gtol=1e-10$ and $maxiter=1e+6$. For Adaptive cubic regularization for Newton's method, we use the AdaptiveCubicReg module in the implementation in \cite{ARCGitHub}. We use the default hyperparameters as recommended there, and use "exact" for the hessian$\_$update$\_$method.  For hyperparameters in Inertial Newton's method, we choose $\alpha =0.5$ and $\beta =0.1$ as recommended by the authors of \cite{bolte-etal}. Source codes for the current paper are available at the GitHub link \cite{phuongGitHub}.

We will also compare the performance to Unbounded Two-way Backtracking GD \cite{truong-nguyen1}. The hyperparameters for Backtracking GD are fixed through all experiments as follows: $\delta _0=1$, $\alpha =0.5$ and $\beta =0.7$. Recall that this means we have the following in Armijo's condition: $f(x-\beta ^m\delta _0x)-f(x)\leq -\alpha \beta ^m\delta _0||\nabla f(x)||^2$, where $m\in \mathbb{Z}_{\geq 0}$ depends on $x$. Here we recall the essence of Unbounded and Two-way variants of Backtracking GD, see \cite{truong-nguyen1} for more detail. In the Two-way version, one starts the search for learning rate $\delta _n$ - at the step n- not at $\delta _0$ but at $\delta _{n-1}$, and allows the possibility of increasing $\delta \mapsto \delta /\beta $, and not just decreasing $\delta \mapsto \delta \beta$ as in the standard version of Backtracking GD. In the Unbounded variant, one allows the upper bound for $\delta _n$ not as $\delta _0$ but as $\max\{\delta _0,\delta _0||\nabla f(x_n)||^{-\kappa}\}$ for some constant $0<\kappa <1$. In all the experiments here, we fix $\kappa =1/2$. The Two-way version helps to reduce the need to do function evaluations in checking Armijo's condition, while the Unbounded version helps to make large step sizes near degenerate critical points and hence also helps with quicker convergence. 

The test functions include many different behaviours, among them are various benchmarks functions from the Wikipedia page for Newton's method \cite{Newton1} and from the Wikipedia page on test functions for optimization \cite{test}.  They include many benchmark functions from test sets such as CUTEer/st \cite{cute} and \cite{jamil-yang}.  

From Examples 1 to 15, we compute gradients and Hessians symbolically. However, from Examples 16 onward, we use the python package numdifftools \cite{num} to compute gradients and Hessian, since symbolic computation is not quite efficient. All the experiments are run on a usual personal computer. Experimental results are summarised in Tables \ref{tab:Tasks},  \ref{tab:Tasks2} and \ref{tab:Tasks3}.   

The unit for running time is seconds. In the experiments, running time will be reported for an algorithm only if it does not diverge to infinity or encounter errors.

{\bf Data for Table \ref{tab:Tasks2aa}}: Here the cost function is the Rosenbrock function $$f_D(x_1,\ldots ,x_D)=\sum _{i=1}^{D-1}f_7(x_i,x_{i+1}),$$ see  \cite{cute, test}, where $f_7(x,y)=(x-1)^2+100(y-x^2)^2$. This function has a global minimum at $x_1=\ldots =x_D=1$, with function value $0$. Here the dimension is $D=30$, and the initial point is randomly chosen with entries in the interval $[-20,20]$. 

Here, the function value of the initial point is $73511310.022068908795$. The initial point (which is randomly chosen in $[-20,20]^{30}$) is: 

[0.26010457, -10.91803423, 2.98112261, -15.95313456,  -2.78250859, -0.77467653,  -2.02113182,   9.10887908, -10.45035903,  11.94967756, -1.24926898,  -2.13950642,   

7.20804014,   1.0291962,    0.06391697, 2.71562242, -11.41484204,  10.59539405,  

12.95776531,  11.13258434,
   8.16230421, -17.21206152,  -4.0493811,  -19.69634293,  14.25263482, 3.19319406,  11.45059677,  18.89542157,  19.44495031,  -3.66913821].

{\bf Data for Table \ref{tab:Tasks3aa}}: Here the cost function is the Styblinski-Tang function $$f_{26}(x_1,\ldots ,x_D)=\sum _{i=1}^D(x_i^4-16x_i^2+5x_i)/2,$$ see \cite{jamil-yang}. The global minimum is at $(x_1,\ldots ,x_D)$ $=$ $(-2.903534$, $\ldots $, $-2.903534)$. The optimal function value is in the interval $(-39.16617D,-39.16616D)$. Here  the dimension is $D=100$. The initial point is randomly chosen with entries in the interval $[-1,1]$. 

 In the case reported here, the function value of the initial point is -247.248. The initial point (which is randomly chosen in $[-1,1]^{100}$) is: 
 
 [-0.15359941, -0.59005902,  0.45366905, -0.94873933,  0.52152264, -0.02738085, 
 
 0.17599868,  0.36736119,  0.30861332,  0.90622707,  0.10472251, -0.74494753, 
 
 0.67337336, -0.21703503, -0.17819413, -0.14024491, -0.93297061,  0.63585997, -0.34774991, -0.02915787, -0.17318147, -0.04669807,  0.03478713, -0.21959983, 
 
 0.54296245,  0.71978214, -0.50010954, -0.69673303,  0.583932,   -0.38138978,

 -0.85625076,  0.20134663, -0.71309977, -0.61278167,  0.86638939,  0.45731164,  -0.32956812,  0.64553452, -0.89968231,  0.79641384,  0.44785232,  0.38489415, 
 
 -0.51330669,  0.81273771, -0.54611157, -0.87101225, -0.72997209, -0.16185048, 
 
 0.38042508, -0.63330049,  0.71930612, -0.33714448, -0.24835364, -0.78859559, -0.07531072,  0.19087508, -0.95964552, -0.72759281,  0.13079216,  0.6982817,
  0.54827214,  0.70860856, -0.51314115, -0.54742142,  0.73180924, -0.28666226, 0.89588517,  
  
  0.35797497, -0.21406766, -0.05558283,  0.89932563, -0.16479757, -0.29753867,  
  
  0.5090385,   0.95156811,  0.8701501,   0.62499125, -0.22215331, 0.8355082,  
  
  -0.83695582, -0.96214862, -0.22495384, -0.30823426,  0.55635375, 0.38262606, -0.60688932, -0.04303575,  0.59260985,  0.5887739,  -0.00570958,
 -0.502354,    0.50740011, -0.08916369,  0.62672251,  0.13993309, -0.92816931, 0.50047918,  0.856543, 0.99560466, -0.44254687]

\begin{table}[htp]
\fontsize{8}{9}\selectfont
  \centering
  \begin{tabular}{|l|c|c|c|c|c|c|c|}
  \hline
 $\#$/Method&ACR&BFGS	& Newton   & NewQ  & Rand& Iner& Back\\
\hline
$1$ &5e+7&4.848e+7& 1.1e+7  & 1.1e+7 &2.1e+7 & 2e+24&1.9e+7 \\
\hline
$2$ &7e+6&4.393e+7& 8.8e+6 & 8.8e+6 &2.9e+8 &  5e+73&3.5e+6 \\
\hline
$3$ &1.9e+6&4.305e+7& 7.9e+6  & 7.9e+6 &1.8e+8 & 6e+79&4.6e+5 \\
\hline
$4$ &6.7e+5&4.284e+7& 1.5e+6 & 1.5e+6 &9.1e+8 & 3e+85&6.6e+4 \\
\hline
$5$ &2.4e+5&4.276e+7& 3.1e+5 & 3.1e+5 &4.8e+7 & 5e+90&1e+4 \\
\hline
$6$ &9.5e+4&4.274e+7& 6.8e+4  & 6.8e+4 &9.3e+6 & 3e+95&1838.355 \\
\hline
$7$ &3.5e+4&4.271e+7& 1.3e+4  & 1.4e+4 &1.2e+6 & $>$1e+100&872.696\\
\hline
$8$ &1.5e+4&4.268e+7& 9500.837 & 3.3e+4 &2.2e+5 & $>$1e+100&598.926 \\
\hline
$9$ &7100.203&4.264e+7& 2057.675 & 3.5e+6 &1.7e+5 & $>$1e+100&416.258 \\
\hline
$10$ &3653.787&4.257e+7& 2.8e+6  & 7.0e+5 &1.3e+5 & $>$1e+100&325.297 \\
\hline
$11$ &2040.195&4.248e+7& 5.7e+5  & 1.3e+5 &2.3e+5 & $>$1e+100&199.156 \\
\hline
$12$ &1163.326&4.242e+7& 1.1e+5  & 2.7e+4 &4e+4 & $>$1e+100&177.524 \\
\hline
$13$ &664.231&4.234e+7& 3.7e+5  & 5229.068 &3.4e+4 & $>$1e+100&150.866 \\
\hline
$14$ &392.672&4.219e+7& 7.4e+4 & 1069.167 &2.7e+4 & $>$1e+100&134.882 \\
\hline
$15$ &248.317&4.191e+7& 1.4e+4  & 282.508 &9.5e+7 & $>$1e+100&83.909 \\
\hline
$16$ &169.778&4.139e+7& 2907.813  & 304.788 &5.7e+7 & $>$1e+100&61.573 \\
\hline
$17$ &103.254&4.067e+7& 595.479  & 1245.013 &2.2e+7 & $>$1e+100&40.437 \\
\hline
$18$ &82.442&4.025e+7& 170.796 & 292.143 &1.8e+7 & $>$1e+100&30.304 \\
\hline
$19$ &50.973&4.005e+7& 99.278  & 111.045 &2.4e+6 & $>$1e+100&29.503 \\
\hline
$20$ &63.640&4.000e+7& 1.7e+5 & 1616.337 &1.4e+6 & $>$1e+100&29.455 \\
\hline
$21$ &31.978&3.996e+7& 2.9e+4  & 1379.143 &7.5e+5 & $>$1e+100&29.400 \\
\hline
$22$ &28.330&3.993e+7& 2.1e+4  & 9940.244 &4.2e+5 & $>$1e+100&29.350 \\
\hline
$23$ &27.805&3.989e+7& 957.175  & 1963.902 &2.3e+5 & $>$1e+100&29.232 \\
\hline
$24$ &26.979&3.988e+7& 199.549  & 320.529 &9.5e+4 & $>$1e+100&29.121 \\
\hline
$25$ &26.711&3.987e+7& 101.736  & 47.979 &7.5e+4 & $>$1e+100&28.983 \\
\hline
$26$ &25.624&3.985e+7& 36.899  & 6.388 &2.0e+4 & $>$1e+100&28.895 \\
\hline
$27$ &25.307&3.97e+7& 25.363  & 2.999 &1.2e+4 & $>$1e+100&28.819 \\
\hline
$28$ &24.262&3.5e+7& 25.046  & 2.201 &6228.802 & $>$1e+100&28.757 \\
\hline
$29$ &23.898&2.82e+7& 23.287  & 1.711 &2407.019 & $>$1e+100&28.687 \\
\hline
$30$ &22.901&2.801e+7& 23.970  & 0.943 &2159.139 & $>$1e+100&28.636 \\
\hline
$31$ &22.562&2.800e+7& 21.750& 1.480 &1573.550 & $>$1e+100&28.541 \\
\hline
$32$ &21.544&2.0e+7& 22.221  & 0.095 &938.376 & $>$1e+100&28.468\\
\hline
$33$ &21.168&1.0e+7& 20.238  & 0.065 &712.356 & $>$1e+100&28.391 \\
\hline
$34$ &20.186&3.7e+6& 20.744 & 3.1e-4 &598.098 & $>$1e+100&28.317 \\
\hline
$35$ &19.828&1.5e+6& 18.722  & 3.9e-7 &601.366 & $>$1e+100&28.272 \\
\hline
$36$ &18.827&6.6e+5& 19.355  & 2.3e-13 &392.864 & $>$1e+100&28.245 \\
\hline
$37$ &18.502&4.1e+5& 17.191  & 2.5e-25 &182.599 & $>$1e+100&28.152 \\
\hline
$38$ &17.467&2.4e+5& 17.582  & 5.2e-29 &336.663 & $>$1e+100&28.084 \\
\hline
$39$ &17.086&1.8e+5& 15.690  & 1.2e-29 &330.673 & $>$1e+100&28.036 \\
\hline
$40$ &16.108&1.2e+5& 16.384 &1.2e-29 &253.452 & $>$1e+100&28.009 \\
\hline
$41$ &15.765&9.8e+4& 14.150 & 1.2e-29 &171.692 & $>$1e+100&27.976 \\
\hline
$42$ &14.751&7.4e+4& 14.549 & 1.2e-29 &127.121 & $>$1e+100&27.955 \\
\hline
$43$ &14.417&5.6e+4& 12.651 &1.2e-29 &119.368 & $>$1e+100&27.925 \\
\hline
$44$ &13.390&4.6e+4& 13.230 & 1.2e-29 &96.072 & $>$1e+100&27.903 \\
\hline
$45$ &13.021&4.2e+4& 11.118 &1.2e-29 &85.073 & $>$1e+100&27.880 \\
\hline
$46$ &12.030&3.6e+4& 11.752 & 1.2e-29 &83.087 & $>$1e+100&27.862 \\
\hline
$47$ &11.711&2.6e+4& 9.603 &1.2e-29 &77.609 & $>$1e+100&27.832\\
\hline
$48$ &10.671&1.3e+4& 9.830 &1.2e-29 &134.342 & $>$1e+100&27.810 \\
\hline
$49$ &10.309&1.1e+4& 8.100  & 1.2e-29 &105.408 & $>$1e+100&27.789 \\
\hline
$50$ &9.309&8990.601& 9.408 & 1.2e-29 &644.618 & $>$1e+100&27.770 \\
\hline
 Time&7.570&4.600& 128.454 & 113.720 &114.362 & 3.561&120.102 \\
\hline
  \end{tabular}
   \caption{\footnotesize{Typical evolution of function values for several different algorithms, in the first 50 iterations. Cost function is the Rosenbrock function in dimension $D=30$. {\bf Legends}: "$\sharp$" for  iteration number, "Time" for running time in seconds,  "ACR" is Adaptive cubic regularization, "Newton" is Newton's method, "NewQ" is New Q-Newton's method, "Rand" is Random damping Newton's method, "Iner" Inertial Newton's method, "Back" is Unbounded Two-way Backtracking gradient descent. New Q-Newton's method is the only algorithm that clearly converges to the global minimum within 50 iterations.}}
  \label{tab:Tasks2aa}
\end{table}

\begin{table}[htp]
\fontsize{8}{9}\selectfont
  \centering
  \begin{tabular}{|l|c|c|c|c|c|c|c|}
  \hline
 $\#$/Method&ACR&BFGS	& Newton   & NewQ  & Rand& Iner& Back\\
\hline
$1$ &8.5e+8&-1862.231&5.533 & -1055.065 &-5.664 & 6.5e+5&-1244.750 \\
\hline
$2$ &8.5e+8&-2041.620& 19.522 & 8.4e+5 &-10.192 &  7.3e+13&-2320.487 \\
\hline
$3$ &8.5e+8&-2125.694& 19.561  & 4.7e+5&-9.488 &4.6e+36&-2808.259 \\
\hline
$4$ &8.5e+8&-2255.984& 19.561 & 8.6e+5&14.829 & $>$1e+100&-3074.425\\
\hline
$5$ &8.5e+8&-2426.524& 19.561 & 1.6e+5&17.946 &$>$1e+100&-3142.183 \\
\hline
$6$ &8.5e+8&-2559.800& 19.561  &2.7e+7 &18.847 & $>$1e+100&-3211.936 \\
\hline
$7$ &8.5e+8&-2697.341& 19.561 & 1.0e+8 &18.949& $>$1e+100&-3267.532\\
\hline
$8$ &8.5e+8&-2804.609& 19.561 & 2.1e+8 &19.475 & $>$1e+100&-3304.603 \\
\hline
$9$ &8.5e+8&-2896.972& 19.561 & 4.1e+7 &19.489 & $>$1e+100&-3308.610 \\
\hline
$10$ &8.4e+8&-3015.704& 19.561 & 8.2e+6 &19.560 & $>$1e+100&-3308.736 \\
\hline
$11$ &8.2e+8&-3187.799&19.561  & 1.6e+6&19.560 & $>$1e+100&-3308.737 \\
\hline
$12$ &8.2e+8&-3232.460&19.561  &3.1e+5 &19.561 & $>$1e+100&-3308.737 \\
\hline
$13$ &8.2e+8&-3239.224&19.561  & 5.8e+4 &19.561 & $>$1e+100&-3308.737 \\
\hline
$14$ &8.2e+8&-3251.150& 19.561 & 8584.995 &19.561 & $>$1e+100&-3308.737 \\
\hline
$15$ &8.2e+8&-3271.454& 19.561 & -1105.810 &19.561 & $>$1e+100&-3308.737 \\
\hline
$16$ &8.2e+8&-3275.160& 19.561 &-2932.534 &19.561 & $>$1e+100&-3308.737 \\
\hline
$17$ &8.2e+8&-3281.961& 19.561 &-3255.707 &19.561 & $>$1e+100&-3308.737 \\
\hline
$18$ &8.2e+8&-3291.847& 19.561 & -3304.050 &19.561 & $>$1e+100&-3308.737 \\
\hline
$19$ &8.2e+8&-3293.607& 19.561  & -3308.608 &19.561 & $>$1e+100&-3308.737 \\
\hline
$20$ &8.2e+8&-3296.911& 19.561 & -3308.737 &19.561& $>$1e+100&-3308.737 \\
\hline
$21$ &8.2e+8&-3299.434& 19.561  & -3308.737 &19.561& $>$1e+100&-3308.737 \\
\hline
$22$ &8.2e+8&-3303.019&19.561 &-3308.737 &19.561 & $>$1e+100&-3308.737 \\
\hline
$23$ &8.2e+8&-3307.570& 19.561 &-3308.737 &19.561 & $>$1e+100&-3308.737 \\
\hline
$24$ &8.2e+8&-3307.706& 19.561  & -3308.737 &19.561 & $>$1e+100&-3308.737 \\
\hline
$25$ &8.2e+8&-3307.959& 19.561  & -3308.737 &19.561 & $>$1e+100&-3308.737 \\
\hline
$26$ &8.0e+8&-3308.090& 19.561  &-3308.737 &19.561 & $>$1e+100&-3308.737 \\
\hline
$27$ &8.0e+8&-3308.317& 19.561  & -3308.737 &19.561 & $>$1e+100&-3308.737 \\
\hline
$28$ &8.0e+8&-3308.519& 19.561  & -3308.737&19.561 & $>$1e+100&-3308.737 \\
\hline
$29$ &8.0e+8&-3308.591& 19.561  &-3308.737 &19.561 & $>$1e+100&-3308.737 \\
\hline
$30$ &7.9e+8&-3308.699& 19.561  &-3308.737 &19.561 & $>$1e+100&-3308.737 \\
\hline
$31$ &7.9e+8&-3308.704& 19.561& -3308.737 &19.561 & $>$1e+100&-3308.737 \\
\hline
$32$ &7.9e+8&-3308.713& 19.561 & -3308.737 &19.561 & $>$1e+100&-3308.737\\
\hline
$33$ &7.9e+8&-3308.716& 19.561  &-3308.737 &19.561 & $>$1e+100&-3308.737 \\
\hline
$34$ &7.9e+8&-3308.723& 19.561 & -3308.737&19.561 & $>$1e+100&-3308.737 \\
\hline
$35$ &7.9e+8&-3308.727& 19.561  & -3308.737 &19.561 & $>$1e+100&-3308.737 \\
\hline
$36$ &7.6e+8&-3308.732& 19.561  & -3308.737 &19.561 & $>$1e+100&-3308.737 \\
\hline
$37$ &7.6e+8&-3308.736& 19.561  &-3308.737 &19.561 & $>$1e+100&-3308.737 \\
\hline
$38$ &7.6e+8&-3308.736& 19.561  & -3308.737 &19.561 & $>$1e+100&-3308.737 \\
\hline
$39$ &7.5e+8&-3308.737& 19.561  & -3308.737 &19.561 & $>$1e+100&-3308.737 \\
\hline
$40$ &7.5e+8&-3308.737& 19.561&-3308.737 &19.561 & $>$1e+100&-3308.737 \\
\hline
$41$ &6.9e+8&-3308.737& 19.561& -3308.737 &19.561 & $>$1e+100&-3308.737 \\
\hline
$42$ &6.9e+8&-3308.737& 19.561 & -3308.737 &19.561 & $>$1e+100&-3308.737 \\
\hline
$43$ &6.9e+8&-3308.737& 19.561 &-3308.737 &19.561 & $>$1e+100&-3308.737 \\
\hline
$44$ &6.9e+8&-3308.737& 19.561 &-3308.737 &19.561 & $>$1e+100&-3308.737 \\
\hline
$45$ &6.9e+8&-3308.737& 19.561 &-3308.737 &19.561 & $>$1e+100&-3308.737 \\
\hline
$46$ &6.8e+8&-3308.737& 19.561 & -3308.737 &19.561& $>$1e+100&-3308.737 \\
\hline
$47$ &6.8e+8&-3308.737& 19.561 &-3308.737 &19.561 & $>$1e+100&-3308.737\\
\hline
$48$ &6.8e+8&-3308.737& 19.561 &-3308.737 &19.561 & $>$1e+100&-3308.737 \\
\hline
$49$ &6.8e+8&-3308.737& 19.561 &-3308.737 &19.561 & $>$1e+100&-3308.737 \\
\hline
$50$ &6.8e+8&-3308.737& 19.561 & -3308.737 &19.561 & $>$1e+100&-3308.737 \\
\hline
Time &29.071&7.290& 501.080& 496.514 &474.127 & 4.7571&468.119 \\
\hline
  \end{tabular}
   \caption{\footnotesize{Typical evolution of function values for several different algorithms, in the first 50 iterations, for the Styblinski-Tang function in dimension $D=100$.  {\bf Legends}: "$\sharp$" for  iteration number, "Time" for running time in seconds,  "ACR" is Adaptive cubic regularization, "Newton" is Newton's method, "NewQ" is New Q-Newton's method, "Rand" is Random damping Newton's method, "Iner" Inertial Newton's method, "Back" is Unbounded Two-way Backtracking gradient descent. In this case, BFGS, New Q-Newton's method and Unbounded Two-way Backtracking GD are the  algorithms that have the best performance within 50 iterations.}} 
  \label{tab:Tasks3aa}
\end{table}

\begin{table}[htp]
\fontsize{8}{9}\selectfont
  \centering
  \begin{tabular}{|l|c|c|c|c|c|c|}
  \hline
~~~~~~~	& ACR& BFGS   & New Q  & Rand& Iner& Back\\
\hline
$f_1$ & E  &34/$x_{1,BFGS}$/C,G,*  & $x_{1,NewQ}$/D &$x_{1,Rand}$/D & $x_{1,Iner}$/D&30/$x_{1,Back}$/G,* \\
\hline
$f_2$ & E  & 1/$x_{2,BFGS}$/E  & 100/$x_{2,NewQ}$/G,*&$x_{2,Rand}$/D &$x_{2,Iner}$/D &100/$x_{2,Back}$/ G,*\\
\hline
$f_3$ & 8/$x_{3,ACR}$/E  & 2/$x_{3,BFGS}$/G  &22/$x_{3,NewQ}/G$  &$x_{3,Rand}$/D &$x_{3,Iner}$/D & 935/$x_{3,Back}$/C\\
\hline
$f_4$ & 3/$x_{4,ACR}$/L  & 4/$x_{4,BFGS}$/L  & 6/$x_{4,NewQ}$/L &6/$x_{4,Rand}$/L &1945/$x_{4,Iner}$/L & 1e+4/$x_{4,Back}$/L\\
\hline
$f_6$, 1 & 4/$x_{4,ACR}$/E,G,*  & 7/$x_{6,1,BFGS}$/G,*  & 10/$x_{6,1,NewQ}$/G,* &23/$x_{6,1,Rand}/U$ &E& 1e+4/$x_{6,1,Back}$/G,*\\
\hline
$f_6$, 2 &5/$x_{5,ACR}$/E,G ,* & 7/$x_{6,2,BFGS}$/G,*  & 9/$x_{6,2,NewQ}$/G,* &31/$x_{6,2,Rand}$/L &E& 1e+4/$x_{6,2,Back}$/G,*\\
\hline
$f_6$, 3 & 5/$x_{5,ACR}$/E,G,* & 2/$x_{6,3,BFGS}$/G,*  & 10/$x_{6,3,NewQ}$/G,* &33/$x_{6,3,Rand}$/G,* &E& 1e+4/$x_{6,3,Back}$/G,*\\
\hline
$f_7$ & 6/$x_{7,ACR}$/E,G,*  & 15/$x_{7,BFGS}$/G,*  & 6/$x_{7,NewQ}$/G,*  &121/$x_{7,Rand}$/G,* & E& 9556/$x_{7,Back}$/G,*\\
\hline
$f_8$ &17/$x_{8,ACR}$/E,G,*   &46/$x_{8,BFGS}$/G,*   &22/$x_{8,NewQ}$/G,*   &59/$x_{8,Rand}$/G,* & E&1e+4/$x_{8,Back}$/G,* \\
\hline

$f_9$ &0/$x_{9,ACR}$/E&5/$x_{9,BFGS}$/E  &$x_{9,NewQ}$  &E &$x_{9,Iner}$/D&1e+4/$x_{9,Back}$/G, *  \\
\hline

$f_{10}$ &4/$x_{10,ACR}$/G,*&6/$x_{10,BFGS}$/G,*  & 9/$x_{10,NewQ}$/G,* & 34/$x_{10,Rand}$/G,*&E& 1e+4/$x_{10,Back}$/G,* \\
\hline

$f_{11}$ &E& 5/$x_{11,BFGS}$ & 28/$x_{11,NewQ}$ &161/$x_{11,Rand}$ &3327/$x_{11,Iner}$& 6/$x_{11,Back}$ \\
\hline

$f_{12}$ &5/$x_{12,ACR}$& $x_{12,BFGS}$/D,* &$x_{12,NewQ}$/D,* &25/$x_{12,Rand}$/S &$x_{12,Iner}$/D,*&$x_{12,Back}$/D,*  \\
\hline

$f_{13}$ &3/$x_{13,ACR}$/G,*&4/$x_{13,BFGS}$/G,*  &1/$x_{13,NewQ}$/G,* &20/$x_{13,Rand}$/G ,*&4929/$x_{13,Iner}$/G,*&29/$x_{13,Back}$/G,*  \\
\hline

$f_{14}$ &6/$x_{14,ACR}$/G,*& 2/$x_{14,BFGS}$/G,* &4/$x_{14,NewQ}$/G,* &E &$x_{14,Iner}$/D& 1e+4/$x_{14,Back}$/G,* \\
\hline

$f_{15}$ &2/$x_{15,ACR}$/E& 2/$x_{15,BFGS}$/E &$x_{15,NewQ}$/D,* &E &$x_{15,Iner}$/D,*& $x_{15,Back}$/D,* \\
\hline

$f_{16}$, 1 &7/$x_{16,1,ACR}$& 13/$x_{16,1,BFGS}$/E &14/$x_{16,1,NewQ},*$&25/$x_{16,1,Rand}$ &$x_{16,1,Iner}$/D& 1e+4/$x_{16,1,Back}$\\
\hline

$f_{16}$, 2 &1e+4/$x_{16,2,ACR}$& 17/$x_{16,2,BFGS}$/G,* &23/$x_{16,2,NewQ}$ &34/$x_{16,2,Rand},*$ &$x_{16,1,Iner}$/D& 62/$x_{16,2,Back}$/G.* \\
\hline

$f_{17}$, 1 &5/$x_{17,1,ACR}$/*& 13/$x_{17,1,BFGS}$/* &8/$x_{17,1,NewQ}$ &35/$x_{17,1,Rand}$ &$x_{17,1,Iner}$/D& 1e+4/$x_{17,1,Back}$/* \\
\hline

$f_{17}$, 2 &4/$x_{17,2,ACR}$/E&11/$x_{17,2,BFGS}$/G,*  &11/$x_{17,2,NewQ}$/G,* &31/$x_{17,2,Rand}$ &$x_{17,2,Iner}$/D&1e+4/$x_{17,2,Back}$/G,*  \\
\hline

$f_{18}$ &20/$x_{18,ACR}$&48/$x_{18,BFGS}$ &21/$x_{18,NewQ}$&41/$x_{18,Rand}$ &$x_{18,Iner}$/D&4000/$x_{18,Back}$/G,* \\
\hline

$f_{19}$ &10/$x_{19,ACR}$/G,*&19/$x_{19,BFGS}$/G,*&172/$x_{19,NewQ}$&53/$x_{19,Rand}$ &E&214/$x_{19,Back}$/G,* \\
\hline

$f_{20}$, 1 &28/$x_{20,1,ACR}/E$&1/$x_{20,1,BFGS}$/E&1e+4/$x_{20,1,NewQ}$&E&$x_{20,1,Iner}$/D&1e+4/$x_{20,1,Back}$,* \\
\hline

$f_{20}$, 2 &1e+4/$x_{20,2,ACR}$/G&3/$x_{20,2,BFGS}$/E,G&1e+4/$x_{20,2,NewQ}$&E&$x_{20,2,Iner}$/D&1e+4/$x_{20,2,Back}$/G,* \\
\hline

$f_{21}$, 1 &8/$x_{21,1,ACR}/E$&11/$x_{21,1,BFGS}$&9/$x_{21,1,NewQ}$&24/$x_{21,1,Rand}$&$x_{21,1,Iner}$/D&479/$x_{21,1,Back}$/G ,*\\
\hline

$f_{21}$, 2 &4/$x_{21,2,ACR}$/G,*&11/$x_{21,2,BFGS}$/G,*&6/$x_{21,2,NewQ}$/G,*&28/$x_{21,2,Rand}$/G,*&$x_{21,1,Iner}$/D&485/$x_{21,2,Back}$/G,* \\
\hline

$f_{22}$, 1 &16/$x_{22,1,ACR}$/E&9/$x_{22,1,BFGS}$/E &10/$x_{22,1,NewQ}$/*&20/$x_{22,1,Rand}$&$x_{22,1,Iner}$/D&1e+4/$x_{22,1,Back}$\\
\hline

$f_{22}$, 2 &12/$x_{22,2,ACR}$&8/$x_{22,2,BFGS}$/* &5/$x_{22,2,NewQ}$/*&34/$x_{22,2,Rand}$/*&$x_{22,2,Iner}$/D&1e+4/$x_{22,2,Back}$/*\\
\hline

$f_{23}$&6/$x_{23,ACR}$/E,G,*&9/$x_{23,BFGS}$/G,* &6/$x_{23,NewQ}$/G,*&31/$x_{23,Rand}$&$x_{23,Iner}$/D&1e+4/$x_{23,Back}$/G,*\\
\hline

$f_{24},1$&1124/$x_{24,1,ACR}$/E,*&6/$x_{24,1,BFGS}$/E &$x_{24,1,NewQ}$/D&$x_{24,1,Rand}$/D&$x_{24,1,Iner}$/D&1e+4/$x_{24,1,Back}$\\
\hline

$f_{24},2$&254/$x_{24,2,ACR}$/E&35/$x_{24,2,BFGS}/G,*$ &10/$x_{24,2,NewQ}$/G,*&$35/x_{24,2,Rand}/G,*$&$x_{24,2,Iner}$/D&1e+4/$x_{24,2,Back}$/G,*\\
\hline

$f_{25},1$&63/$x_{25,1,ACR}/E,*$&0/$x_{25,1,BFGS}$/E &$x_{25,1,NewQ}$/D&$x_{25,1,Rand}$/D&$x_{25,1,Iner}$/D&1e+4/$x_{25,1,Back}$/*\\
\hline

$f_{25},2$&3425/$x_{25,2,ACR}$/E,*&27/$x_{25,2,BFGS}$/* &19/$x_{25,2,NewQ}$/*&49/$x_{25,2,Rand}$&$x_{25,2,Iner}$/D&1e+4/$x_{25,2,Back}$/*\\
\hline

$f_{26},1$&8/$x_{26,1,ACR}$/E,*&13/$x_{26,1,BFGS}$/E,* &13/$x_{26,1,NewQ}$/*&27/$x_{26,1,Rand}$&$x_{26,1,Iner}$/D&1e+4/$x_{26,1,Back}$/*\\
\hline

$f_{26},2$&8/$x_{26,2,ACR}$/E,G,*&9/$x_{26,2,BFGS}$/G,* &6/$x_{26,2,NewQ}$/G,*&35/$x_{26,2,Rand}$/G,*&$x_{26,2,Iner}$/D&1e+4/$x_{26,2,Back}$/G,*\\
\hline

\hline
  \end{tabular}
   \caption{Results of experiments on different Newton's method variant algorithms, with Unbounded Two-way Backtracking GD included for a comparison. The maximum number of iterates is  $1e+4$ (but for some examples we need to reduce this number to avoid errors such as division by zero), but the algorithm can stop before that because either the size of the gradient is smaller than a threshold ($1e-10$), there is error, or (for BFGS and ACR) some unknown reasons. The format is n/x/Remarks, where n is the number of iterates needed to achieve the point $x$. Legends: "E" for errors, "D" for divergence, "C" for convergence, "Back" for Unbounded Two-way Backtracking GD, "ACR" for Adaptive Cubic Regularization, "Iner" for Inertial Newton's method, "New Q" for New Q-Newton's method, "Rand" for Random damping Newton's method, "S" for (near ) a saddle point or local maximum, "L" for (near) a local minimum, "G" for near a global minimum, "U" for unstable convergence behaviour, "*": best performance.}
  \label{tab:Tasks}
\end{table}

\begin{table}[htp]
\fontsize{9}{9}\selectfont
  \centering
  \begin{tabular}{|l|c|c|c|c|c|c|c|}
  \hline
Iteration $\#$/Method&ACR&BFGS	& Newton   & NewQ  & Rand& Iner& Back\\
\hline
$1$ &5e+7&4.848e+7& 1.1e+7  & 1.1e+7 &2.1e+7 & 2e+24&1.9e+7 \\
\hline
$2$ &7e+6&4.393e+7& 8.8e+6 & 8.8e+6 &2.9e+8 &  5e+73&3.5e+6 \\
\hline
$3$ &1.9e+6&4.305e+7& 7.9e+6  & 7.9e+6 &1.8e+8 & 6e+79&4.6e+5 \\
\hline
$4$ &6.7e+5&4.284e+7& 1.5e+6 & 1.5e+6 &9.1e+8 & 3e+85&6.6e+4 \\
\hline
$5$ &2.4e+5&4.276e+7& 3.1e+5 & 3.1e+5 &4.8e+7 & 5e+90&1e+4 \\
\hline
$6$ &9.5e+4&4.274e+7& 6.8e+4  & 6.8e+4 &9.3e+6 & 3e+95&1838.355 \\
\hline
$7$ &3.5e+4&4.271e+7& 1.3e+4  & 1.4e+4 &1.2e+6 & $>$1e+100&872.696\\
\hline
$8$ &1.5e+4&4.268e+7& 9500.837 & 3.3e+4 &2.2e+5 & $>$1e+100&598.926 \\
\hline
$9$ &7100.203&4.264e+7& 2057.675 & 3.5e+6 &1.7e+5 & $>$1e+100&416.258 \\
\hline
$10$ &3653.787&4.257e+7& 2.8e+6  & 7.0e+5 &1.3e+5 & $>$1e+100&325.297 \\
\hline
$11$ &2040.195&4.248e+7& 5.7e+5  & 1.3e+5 &2.3e+5 & $>$1e+100&199.156 \\
\hline
$12$ &1163.326&4.242e+7& 1.1e+5  & 2.7e+4 &4e+4 & $>$1e+100&177.524 \\
\hline
$13$ &664.231&4.234e+7& 3.7e+5  & 5229.068 &3.4e+4 & $>$1e+100&150.866 \\
\hline
$14$ &392.672&4.219e+7& 7.4e+4 & 1069.167 &2.7e+4 & $>$1e+100&134.882 \\
\hline
$15$ &248.317&4.191e+7& 1.4e+4  & 282.508 &9.5e+7 & $>$1e+100&83.909 \\
\hline
$16$ &169.778&4.139e+7& 2907.813  & 304.788 &5.7e+7 & $>$1e+100&61.573 \\
\hline
$17$ &103.254&4.067e+7& 595.479  & 1245.013 &2.2e+7 & $>$1e+100&40.437 \\
\hline
$18$ &82.442&4.025e+7& 170.796 & 292.143 &1.8e+7 & $>$1e+100&30.304 \\
\hline
$19$ &50.973&4.005e+7& 99.278  & 111.045 &2.4e+6 & $>$1e+100&29.503 \\
\hline
$20$ &63.640&4.000e+7& 1.7e+5 & 1616.337 &1.4e+6 & $>$1e+100&29.455 \\
\hline
$21$ &31.978&3.996e+7& 2.9e+4  & 1379.143 &7.5e+5 & $>$1e+100&29.400 \\
\hline
$22$ &28.330&3.993e+7& 2.1e+4  & 9940.244 &4.2e+5 & $>$1e+100&29.350 \\
\hline
$23$ &27.805&3.989e+7& 957.175  & 1963.902 &2.3e+5 & $>$1e+100&29.232 \\
\hline
$24$ &26.979&3.988e+7& 199.549  & 320.529 &9.5e+4 & $>$1e+100&29.121 \\
\hline
$25$ &26.711&3.987e+7& 101.736  & 47.979 &7.5e+4 & $>$1e+100&28.983 \\
\hline
$26$ &25.624&3.985e+7& 36.899  & 6.388 &2.0e+4 & $>$1e+100&28.895 \\
\hline
$27$ &25.307&3.97e+7& 25.363  & 2.999 &1.2e+4 & $>$1e+100&28.819 \\
\hline
$28$ &24.262&3.5e+7& 25.046  & 2.201 &6228.802 & $>$1e+100&28.757 \\
\hline
$29$ &23.898&2.82e+7& 23.287  & 1.711 &2407.019 & $>$1e+100&28.687 \\
\hline
$30$ &22.901&2.801e+7& 23.970  & 0.943 &2159.139 & $>$1e+100&28.636 \\
\hline
$31$ &22.562&2.800e+7& 21.750& 1.480 &1573.550 & $>$1e+100&28.541 \\
\hline
$32$ &21.544&2.0e+7& 22.221  & 0.095 &938.376 & $>$1e+100&28.468\\
\hline
$33$ &21.168&1.0e+7& 20.238  & 0.065 &712.356 & $>$1e+100&28.391 \\
\hline
$34$ &20.186&3.7e+6& 20.744 & 3.1e-4 &598.098 & $>$1e+100&28.317 \\
\hline
$35$ &19.828&1.5e+6& 18.722  & 3.9e-7 &601.366 & $>$1e+100&28.272 \\
\hline
$36$ &18.827&6.6e+5& 19.355  & 2.3e-13 &392.864 & $>$1e+100&28.245 \\
\hline
$37$ &18.502&4.1e+5& 17.191  & 2.5e-25 &182.599 & $>$1e+100&28.152 \\
\hline
$38$ &17.467&2.4e+5& 17.582  & 5.2e-29 &336.663 & $>$1e+100&28.084 \\
\hline
$39$ &17.086&1.8e+5& 15.690  & 1.2e-29 &330.673 & $>$1e+100&28.036 \\
\hline
$40$ &16.108&1.2e+5& 16.384 &1.2e-29 &253.452 & $>$1e+100&28.009 \\
\hline
$41$ &15.765&9.8e+4& 14.150 & 1.2e-29 &171.692 & $>$1e+100&27.976 \\
\hline
$42$ &14.751&7.4e+4& 14.549 & 1.2e-29 &127.121 & $>$1e+100&27.955 \\
\hline
$43$ &14.417&5.6e+4& 12.651 &1.2e-29 &119.368 & $>$1e+100&27.925 \\
\hline
$44$ &13.390&4.6e+4& 13.230 & 1.2e-29 &96.072 & $>$1e+100&27.903 \\
\hline
$45$ &13.021&4.2e+4& 11.118 &1.2e-29 &85.073 & $>$1e+100&27.880 \\
\hline
$46$ &12.030&3.6e+4& 11.752 & 1.2e-29 &83.087 & $>$1e+100&27.862 \\
\hline
$47$ &11.711&2.6e+4& 9.603 &1.2e-29 &77.609 & $>$1e+100&27.832\\
\hline
$48$ &10.671&1.3e+4& 9.830 &1.2e-29 &134.342 & $>$1e+100&27.810 \\
\hline
$49$ &10.309&1.1e+4& 8.100  & 1.2e-29 &105.408 & $>$1e+100&27.789 \\
\hline
$50$ &9.309&8990.601& 9.408 & 1.2e-29 &644.618 & $>$1e+100&27.770 \\
\hline
Running time (seconds)&7.570&4.600& 128.454 & 113.720 &114.362 & 3.561&120.102 \\
\hline
  \end{tabular}
   \caption{Typical evolution of function values for several different algorithms, in the first 50 iterations. Cost function is the Rosenbrock function in dimension $D=30$.}
  \label{tab:Tasks2}
\end{table}

\begin{table}[htp]
\fontsize{9}{9}\selectfont
  \centering
  \begin{tabular}{|l|c|c|c|c|c|c|c|}
  \hline
Iteration $\#$/Method&ACR&BFGS	& Newton   & NewQ  & Rand& Iner& Back\\
\hline
$1$ &8.5e+8&-1862.231&5.533 & -1055.065 &-5.664 & 6.5e+5&-1244.750 \\
\hline
$2$ &8.5e+8&-2041.620& 19.522 & 8.4e+5 &-10.192 &  7.3e+13&-2320.487 \\
\hline
$3$ &8.5e+8&-2125.694& 19.561  & 4.7e+5&-9.488 &4.6e+36&-2808.259 \\
\hline
$4$ &8.5e+8&-2255.984& 19.561 & 8.6e+5&14.829 & $>$1e+100&-3074.425\\
\hline
$5$ &8.5e+8&-2426.524& 19.561 & 1.6e+5&17.946 &$>$1e+100&-3142.183 \\
\hline
$6$ &8.5e+8&-2559.800& 19.561  &2.7e+7 &18.847 & $>$1e+100&-3211.936 \\
\hline
$7$ &8.5e+8&-2697.341& 19.561 & 1.0e+8 &18.949& $>$1e+100&-3267.532\\
\hline
$8$ &8.5e+8&-2804.609& 19.561 & 2.1e+8 &19.475 & $>$1e+100&-3304.603 \\
\hline
$9$ &8.5e+8&-2896.972& 19.561 & 4.1e+7 &19.489 & $>$1e+100&-3308.610 \\
\hline
$10$ &8.4e+8&-3015.704& 19.561 & 8.2e+6 &19.560 & $>$1e+100&-3308.736 \\
\hline
$11$ &8.2e+8&-3187.799&19.561  & 1.6e+6&19.560 & $>$1e+100&-3308.737 \\
\hline
$12$ &8.2e+8&-3232.460&19.561  &3.1e+5 &19.561 & $>$1e+100&-3308.737 \\
\hline
$13$ &8.2e+8&-3239.224&19.561  & 5.8e+4 &19.561 & $>$1e+100&-3308.737 \\
\hline
$14$ &8.2e+8&-3251.150& 19.561 & 8584.995 &19.561 & $>$1e+100&-3308.737 \\
\hline
$15$ &8.2e+8&-3271.454& 19.561 & -1105.810 &19.561 & $>$1e+100&-3308.737 \\
\hline
$16$ &8.2e+8&-3275.160& 19.561 &-2932.534 &19.561 & $>$1e+100&-3308.737 \\
\hline
$17$ &8.2e+8&-3281.961& 19.561 &-3255.707 &19.561 & $>$1e+100&-3308.737 \\
\hline
$18$ &8.2e+8&-3291.847& 19.561 & -3304.050 &19.561 & $>$1e+100&-3308.737 \\
\hline
$19$ &8.2e+8&-3293.607& 19.561  & -3308.608 &19.561 & $>$1e+100&-3308.737 \\
\hline
$20$ &8.2e+8&-3296.911& 19.561 & -3308.737 &19.561& $>$1e+100&-3308.737 \\
\hline
$21$ &8.2e+8&-3299.434& 19.561  & -3308.737 &19.561& $>$1e+100&-3308.737 \\
\hline
$22$ &8.2e+8&-3303.019&19.561 &-3308.737 &19.561 & $>$1e+100&-3308.737 \\
\hline
$23$ &8.2e+8&-3307.570& 19.561 &-3308.737 &19.561 & $>$1e+100&-3308.737 \\
\hline
$24$ &8.2e+8&-3307.706& 19.561  & -3308.737 &19.561 & $>$1e+100&-3308.737 \\
\hline
$25$ &8.2e+8&-3307.959& 19.561  & -3308.737 &19.561 & $>$1e+100&-3308.737 \\
\hline
$26$ &8.0e+8&-3308.090& 19.561  &-3308.737 &19.561 & $>$1e+100&-3308.737 \\
\hline
$27$ &8.0e+8&-3308.317& 19.561  & -3308.737 &19.561 & $>$1e+100&-3308.737 \\
\hline
$28$ &8.0e+8&-3308.519& 19.561  & -3308.737&19.561 & $>$1e+100&-3308.737 \\
\hline
$29$ &8.0e+8&-3308.591& 19.561  &-3308.737 &19.561 & $>$1e+100&-3308.737 \\
\hline
$30$ &7.9e+8&-3308.699& 19.561  &-3308.737 &19.561 & $>$1e+100&-3308.737 \\
\hline
$31$ &7.9e+8&-3308.704& 19.561& -3308.737 &19.561 & $>$1e+100&-3308.737 \\
\hline
$32$ &7.9e+8&-3308.713& 19.561 & -3308.737 &19.561 & $>$1e+100&-3308.737\\
\hline
$33$ &7.9e+8&-3308.716& 19.561  &-3308.737 &19.561 & $>$1e+100&-3308.737 \\
\hline
$34$ &7.9e+8&-3308.723& 19.561 & -3308.737&19.561 & $>$1e+100&-3308.737 \\
\hline
$35$ &7.9e+8&-3308.727& 19.561  & -3308.737 &19.561 & $>$1e+100&-3308.737 \\
\hline
$36$ &7.6e+8&-3308.732& 19.561  & -3308.737 &19.561 & $>$1e+100&-3308.737 \\
\hline
$37$ &7.6e+8&-3308.736& 19.561  &-3308.737 &19.561 & $>$1e+100&-3308.737 \\
\hline
$38$ &7.6e+8&-3308.736& 19.561  & -3308.737 &19.561 & $>$1e+100&-3308.737 \\
\hline
$39$ &7.5e+8&-3308.737& 19.561  & -3308.737 &19.561 & $>$1e+100&-3308.737 \\
\hline
$40$ &7.5e+8&-3308.737& 19.561&-3308.737 &19.561 & $>$1e+100&-3308.737 \\
\hline
$41$ &6.9e+8&-3308.737& 19.561& -3308.737 &19.561 & $>$1e+100&-3308.737 \\
\hline
$42$ &6.9e+8&-3308.737& 19.561 & -3308.737 &19.561 & $>$1e+100&-3308.737 \\
\hline
$43$ &6.9e+8&-3308.737& 19.561 &-3308.737 &19.561 & $>$1e+100&-3308.737 \\
\hline
$44$ &6.9e+8&-3308.737& 19.561 &-3308.737 &19.561 & $>$1e+100&-3308.737 \\
\hline
$45$ &6.9e+8&-3308.737& 19.561 &-3308.737 &19.561 & $>$1e+100&-3308.737 \\
\hline
$46$ &6.8e+8&-3308.737& 19.561 & -3308.737 &19.561& $>$1e+100&-3308.737 \\
\hline
$47$ &6.8e+8&-3308.737& 19.561 &-3308.737 &19.561 & $>$1e+100&-3308.737\\
\hline
$48$ &6.8e+8&-3308.737& 19.561 &-3308.737 &19.561 & $>$1e+100&-3308.737 \\
\hline
$49$ &6.8e+8&-3308.737& 19.561 &-3308.737 &19.561 & $>$1e+100&-3308.737 \\
\hline
$50$ &6.8e+8&-3308.737& 19.561 & -3308.737 &19.561 & $>$1e+100&-3308.737 \\
\hline
Running time (seconds)&29.071&7.290& 501.080& 496.514 &474.127 & 4.7571&468.119 \\
\hline
  \end{tabular}
   \caption{Typical evolution of function values for several different algorithms, in the first 50 iterations, for the Styblinski-Tang function in dimension $D=100$.} 
  \label{tab:Tasks3}
\end{table}

{\bf Data for Table \ref{tab:Tasks2}}: Here the cost function is the Rosenbrock function $f_D(x_1,\ldots ,x_D)=\sum _{i=1}^{D-1}f_7(x_i,x_{i+1})$, see  \cite{cute, test}, where $f_7(x,y)=(x-1)^2+100(y-x^2)^2$. This function has a global minimum at $x_1=\ldots =x_D=1$, with function value $0$. Here the dimension is $D=30$, and the initial point is randomly chosen with entries in the interval $[-20,20]$. 

In the case reported here, the function value of the initial point is 73511310.022068908795. The initial point (which is randomly chosen in $[-20,20]^{30}$) is: 

[0.26010457, -10.91803423, 2.98112261, -15.95313456,  -2.78250859, -0.77467653,  -2.02113182,   9.10887908, -10.45035903,  11.94967756, -1.24926898,  -2.13950642,   7.20804014,   1.0291962,    0.06391697, 2.71562242, -11.41484204,  10.59539405,  12.95776531,  11.13258434,
   8.16230421, -17.21206152,  -4.0493811,  -19.69634293,  14.25263482, 3.19319406,  11.45059677,  18.89542157,  19.44495031,  -3.66913821].

{\bf Data for Table \ref{tab:Tasks3}}: Here the cost function is the Styblinski-Tang function $f_{26}(x_1,\ldots ,x_D)=$ $\sum _{i=1}^D(x_i^4-16x_i^2+5x_i)/2$, see \cite{jamil-yang}. The global minimum is at $(x_1,\ldots ,x_D)$ $=$ $(-2.903534$, $\ldots $, $-2.903534)$. The optimal function value is in the interval $(-39.16617D,-39.16616D)$. Here  the dimension is $D=100$. The initial point is randomly chosen with entries in the interval $[-1,1]$. 

 In the case reported here, the function value of the initial point is -247.248. The initial point (which is randomly chosen in $[-1,1]^{100}$) is: 
 
 [-0.15359941, -0.59005902,  0.45366905, -0.94873933,  0.52152264, -0.02738085, 0.17599868,  0.36736119,  0.30861332,  0.90622707,  0.10472251, -0.74494753, 0.67337336, -0.21703503, -0.17819413, -0.14024491, -0.93297061,  0.63585997, -0.34774991, -0.02915787, -0.17318147, -0.04669807,  0.03478713, -0.21959983, 0.54296245,  0.71978214, -0.50010954, -0.69673303,  0.583932,   -0.38138978,
 -0.85625076,  0.20134663, -0.71309977, -0.61278167,  0.86638939,  0.45731164,  -0.32956812,  0.64553452, -0.89968231,  0.79641384,  0.44785232,  0.38489415, -0.51330669,  0.81273771, -0.54611157, -0.87101225, -0.72997209, -0.16185048, 0.38042508, -0.63330049,  0.71930612, -0.33714448, -0.24835364, -0.78859559, -0.07531072,  0.19087508, -0.95964552, -0.72759281,  0.13079216,  0.6982817,
  0.54827214,  0.70860856, -0.51314115, -0.54742142,  0.73180924, -0.28666226, 0.89588517,  0.35797497, -0.21406766, -0.05558283,  0.89932563, -0.16479757, -0.29753867,  0.5090385,   0.95156811,  0.8701501,   0.62499125, -0.22215331, 0.8355082,  -0.83695582, -0.96214862, -0.22495384, -0.30823426,  0.55635375, 0.38262606, -0.60688932, -0.04303575,  0.59260985,  0.5887739,  -0.00570958,
 -0.502354,    0.50740011, -0.08916369,  0.62672251,  0.13993309, -0.92816931, 0.50047918,  0.856543, 0.99560466, -0.44254687]

{\bf Example 1:} We test for the function $f_1(x)=|x|^{1+1/3}$. This function has compact sublevels and has one global minimum at $0$, and no other critical points. Initial point $x_0=1$ (other points have similar behaviour). Points to be used in Table \ref{tab:Tasks}:
 \begin{eqnarray*}
 x_{1,ACR}&=&Error,\\
 x_{1,BFGS}&=&-3e-31,~Running ~time=0.0059\\
 x_{1,New Q}&=&\infty,\\
 x_{1,Rand}&=&\infty,\\
 x_{1,Iner}&=&\infty,\\
 x_{1,Back}&=&-2e-33,~Running~time=0.0023.
 \end{eqnarray*}

{\bf Example 2:}  We test for the function $f_2(x)=|x|^{1/3}$.  This function has compact sublevels and had has one global minimum at $0$, and no other critical points. (The result in this case is quite surprising, since the function here is more singular than the function in Experiment 1.) Initial point $x_0=1$ (other points have similar behaviour). Points to be used in Table \ref{tab:Tasks}:
 \begin{eqnarray*}
 x_{2,ACR }&=&Error,\\
 x_{2,BFGS}&=&1,~Running~time=0.006\\
 x_{2,New Q}&=&8e-31,~Running~time=0.291\\
 x_{2,Rand}&=&\infty,~\\
 x_{2,Iner}&=&\infty,\\
 x_{2,Back}&=&8e-85,~Running ~time =0.0093.
 \end{eqnarray*}

{\bf Example 3:} We test for the function $f_3(x)=e^{-1/x^2}$. This function has a global minimum at $x=0$, but also $\lim _{|x|\rightarrow\infty}f'(x)=0$. It does not have compact sublevels. Initial point is $x=3$ (other points have similar behaviour). Points to be used in Table \ref{tab:Tasks}:
 \begin{eqnarray*}
 x_{3,ACR}&=&0.230,~Running~time=0.0215,\\
 x_{3,BFGS}&=&-0.1129,~Running~time=0.0009\\
 x_{3,New Q}&=&0.1826,~Running~time=0.0013\\
 x_{3,Rand}&=&\infty,\\
 x_{3,Iner}&=&\infty,\\
 x_{3,Back}&=&0.1864,~Running~time=0.075.
 \end{eqnarray*}

{\bf Example 4:} We test for the function $f_4(x)=x^3sin(1/x)$.  This function has compact sublevels, and has countably many local maxima and local minima, and these converge to the singular point $0$. We choose the initial point to be $x_0=0.75134554$ (randomly chosen). Points to be used in Table \ref{tab:Tasks}:
 \begin{eqnarray*}
 x_{4,ACR }&=&0.2452,~f_4(x_{4,ACR})=-0.0118,~Running~time=0.0188,\\
 x_{4,BFGS}&=&-0.2452,~f_4(x_{4,BFGS})=-0.0118,~Running~time=0.006,\\
 x_{4,New Q}&=&-0.006,~f_4(x_{4,NewQ})=-2e-7,~Running~time=0.0004,\\
x_{4,Rand}&=&-0.2452,~f_4(x_{4,Rand})=-0.0118,~Running~time=0.001,\\
 x_{4,Iner}&=&-0.2452,~f_4(x_{4,Inder})=-0.0118,~Running~time=0.0429\\
 x_{4,Back}&=&0.2452,~f_4(x_{4,Back})=-0.0118,~Running~time=0.600.
 \end{eqnarray*}

(Interestingly, if the initial point is $1.01$, then after $1$ step, BFGS arrives at $0$.)

{\bf Example 5:} We test for the function $f(x)=x^3cos(1/x)$.  This function does not have compact sublevels, and has countably many local maxima and local minima, and these converge to the singular point $0$. We obtain similar results as in Example 4. 

{\bf Example 6:} We test for the function $f_6(x)=e^{x^2}-2x^3$.  This function has compact sublevels. It has $1$ local minimum, one global minimum and one local maximum. Depending on the randomly chosen initial point $x_0$, there are 3 typical behaviours.

\underline{Case 1}: Initial point $x_0=0.6$. Points to be used in Table \ref{tab:Tasks}:
  \begin{eqnarray*}
 x_{6,1,ACR}&=&1.08737056,~Running~time=0.020,\\
 x_{6,1,BFGS }&=&1.08737056,~Running~time=0.0020\\
 x_{6,1,New Q}&=&1.0873705644002134,~Running~time=0.00047\\
  x_{6,1,Iner}&=&Error,\\
 x_{6,1,Back}&=&1.08737041,~Running~time=0.427. 
\end{eqnarray*}
while $x_{6,1,Rand}=1.0873705644101557$ or $0.3872694020085596$ (unstable, varying on different runnings), with Running time $=$ 0.00057.

\underline{Case 2}: Initial point $x_0=0.8$. Points to be used in Table \ref{tab:Tasks}:
 \begin{eqnarray*}
 x_{6,2,ACR }&=&1.08737056,~Running~time=0.022,\\
 x_{6,2,BFGS}&=&1.08737056,~Running~time=0.00196,\\
 x_{6,2,New Q}&=&1.0873705644002136,~Running~time=0.00081,\\
 x_{6,2,Rand}&=&-3e-11,~Running~time=0.00056,\\
 x_{6,2,Iner}&=&Error,\\
 x_{6,2,Back}&=&1.08737057,~Running~time=0.429. 
  \end{eqnarray*}

\underline{Case 3}: Initial point $x_0=0.9$. Points to be used in Table \ref{tab:Tasks}:
 \begin{eqnarray*}
 x_{6,3,ACR }&=&1.08737056,~Running~time=0.033,\\
 x_{6,3,BFGS}&=&1.08737056,~Running~time=0.0013,\\
 x_{6,3,New Q}&=&1.0873705644002134,~Running~time=0.00048,\\
 x_{6,3,Rand}&=&1.0873705643974583,~Running ~time=0.00055,\\
 x_{6,3,Iner}&=&Error,\\
 x_{6,3,Back}&=&1.08737061,~Running~time=0.429. 
 \end{eqnarray*}

{\bf Example 7:} We test for the function $f_7(x,y)=(x-1)^2+100(y-x^2)^2$ (Rosenbrock's function), \cite{cute}.  This function  has compact sublevels. It has $1$ global minimum $(1,1)$, and no other critical points. Initial point $(0.55134554, 0.75134554)$, which is randomly chosen. Points to be used in Table \ref{tab:Tasks}:
 \begin{eqnarray*}
 x_{7,ACR }&=&(1,1),~Running~time=0.032,\\
 x_{7,BFGS}&=&(1,1),~Running~time=0.0042,\\
 x_{7,New Q}&=&(1,1),~Running~time=0.0036,\\
 x_{7,Rand}&=&(1,1),Running~time=0.0043,\\
 x_{7,Iner}&=&Error,\\
 x_{7,Back}&=&(1, 1),~Running~time=1.12. 
  \end{eqnarray*}

 {\bf Example 8}: We test for the function $f_8(x_1,x_2,x_3,x_4)$ $=f_7(x_1,x_2)+f_7(x_2,x_3)+f_7(x_3,x_4)$ (where $f_7(x,y)$ is Rosenbrock's function in Example 7).  This function  has compact sublevels, \cite{cute}. It has $1$ global minimum $(1,1,1,1,)$, and one local minimum near $(-1,1,1,1)$, and no other critical points. Initial point $(-0.7020, 0.5342, -2.0101, 2.002)$, which is randomly chosen. Points to be used in Table \ref{tab:Tasks}:
 \begin{eqnarray*}
 x_{8,ACR }&=&(0.999,0.999,0.999,0.999),~Running~time=0.087,\\
 x_{8,BFGS}&=&(1,1,1,1),~Running~time=0.0099,\\
 x_{8,New Q}&=&(1,1,1,1),~Running~time=0.0146,\\
 x_{8,Rand}&=&(1,1,1,1),~Running~time=0.0073,\\
  x_{8,Iner}&=&Error,\\
 x_{8,Back}&=&(0.999, 0.999, 0.999, 0.999),~Running~time=1.804. 
\end{eqnarray*}

 {\bf Example 9}: We test for the function $f_9(x,y)=100(y-|x|)^2+|1-x|$ (introduced in \cite{bolte-etal}).  This function  has compact sublevels, but it is not even $C^1$. On the other hand, it is smooth on a dense open subset of $\mathbb{R}^2$.  It has one global minimum at $(1,1)$. Initial point $(-0.99998925, 2.00001188)$, which is randomly chosen. Points to be used in Table \ref{tab:Tasks}:
 \begin{eqnarray*}
 x_{9,ACR }&=&(-0.99998925, 2.00001188),\\
 x_{9,BFGS}&=&(0.31505191, 0.31253145),~Running~time=0.008,\\
 x_{7,Rand}&=&Error,\\
x_{9,Iner}&=&\infty,\\
 x_{9,Back}&=&(1, 0.9978),~Running~time=0.920.
 \end{eqnarray*}  

$ x_{9,New Q}=$ a "near" cycle $(0.49875934, 0.5012469)$ $\mapsto$  $(1.0012469,  0.99875934)$ $\mapsto$ $w_1=$ $(0.49875934, 0.5012469)$, with Running time $=$ 2.502. {\bf Remark:} There are some interesting phenomena to note. First, one point in the cycle $(1.0012469,  0.99875934)$ is  {\bf close} to the global minimum $(1,1)$. Second, if we choose a different random initial point, then we still arrive at one similar "near" cycle but does not converge. Also, it is interesting that if we use the basic version of New Q-Newton's method, in Table \ref{table:alg}, we get similar near cycles, even though the size of the cycle may change. At the moment, it is not clear to us whether this could only be a consequence of computational errors or an intrinsic property of this special function. (We speculate that the first reason could be more possible.) In this example,  note that {\bf only} Two-way Backtracking GD can clearly converge to the global maximum, even though a bit slow.

 {\bf Example 10}: We test for the function $f_{10}(t)=(t^4/4)-t^2+2t$ (mentioned in \cite{Newton1}).  This function  has compact sublevels. It has one global minimum, 1 local minimum and 1 local maximum. Initial point $0$ (this point is specially chosen to illustrate that Newton's method may enter an infinite cycle, in this case $0\mapsto 1\mapsto 0$,  without convergence).  Points to be used in Table \ref{tab:Tasks}:
 \begin{eqnarray*}
 x_{10,ACR }&=&-1.76929235,~Running~time=0.016,\\
 x_{10,BFGS}&=&-1.76929235,~Running~time=0.0018,\\
 x_{10,New Q}&=&-1.769292354,~Running~time=0.0017,\\
 x_{10,Rand}&=&-1.769292354,~Running~time=0.0010,\\
 x_{10,Iner}&=&Error,\\
 x_{10,Back}&=&-1.76929237,~Running~time=0.491.
 \end{eqnarray*}  

 {\bf Example 11}: We test for the function $f_{11}(t)=4/3 ci(2/t)+t(t^2-2)sin(2/t)/3+t^2/2+t^2cos(2/t)/3$ (mentioned in \cite{Newton1}). Initial point $1.00001188$ (randomly chosen). Points to be used in Table \ref{tab:Tasks}:
 \begin{eqnarray*}
 x_{11,ACR }&=&Error,\\
 x_{11,BFGS}&=&4e-14,~Running~time=0.0017,\\
 x_{11,New Q}&=&3e-11,~Running~time=0.0008,\\
 x_{11,Rand}&=&3e-11,~Running~time=0.0123,\\
 x_{11,Iner}&=&9e-11,~Running~time=0.053,\\
x_{11,Back}&=&-3e-25,~Running~time=0.001.
 \end{eqnarray*}

{\bf Example 12:}  We test for the function $f_{12}(x,y)=x^2+y^2+4xy$. This function has only one critical point $(0,0)$, which is non-degenerate and is a saddle point. A good method should diverge. The function does not have compact sublevels. Initial point $(1,2)$, other points have similar behaviour. Points to be used in Table \ref{tab:Tasks}:
 \begin{eqnarray*}
 x_{12,ACR }&=&(-18.881,14.839),~Running~time=0.018,\\
 x_{12,BFGS}&=&\infty,\\
 x_{12,New Q}&=&\infty,\\
 x_{12,Rand}&=&(3e-12,6e-12),~Running~time=0.0036,\\
 x_{12,Iner}&=&\infty,\\
 x_{12,Back}&=&\infty.
 \end{eqnarray*}

 {\bf Example 13:} We test for the function $f_{13}(x,y)=x^2+y^2+xy$. This function has compact sublevel. It has only one critical point $(0,0)$, which is non-degenerate global minimum. Initial point $(0.55134554, 0.75134554)$, which is randomly chosen. Points to be used in Table \ref{tab:Tasks}:
 \begin{eqnarray*}
 x_{13,ACR }&=&(-3e-6,-5e-6),~Running~time=0.025\\
 x_{13,BFGS}&=&(-1e-11,-1e-11),~Running~time=0.0016,\\
 x_{13,New Q}&=&(-2e-32,9e-32),~Running~time=0.0093,\\
 x_{13,Rand}&=&(6e-7,8e-7),~Running~time=0.003,\\
 x_{13,Iner}&=&(7e-11,-7e-11),~Running~time=0.0092,\\
 x_{13,Back}&=&(-8e-12,6e-12),~Running~time=0.0049.
 \end{eqnarray*}

 {\bf Example 14:}  We test for the function $f_{14}(x,y)=x^2+y^2+2xy$. This function has global minima on the line $x+y=0$, and no other critical points. Initial point $(0.55134554, 0.75134554)$, which is randomly chosen. Points to be used in Table \ref{tab:Tasks}:
 \begin{eqnarray*}
 x_{14,ACR }&=&(73.924,-73.924),~Running~time=0.059,\\
 x_{14,BFGS}&=&(-0.1,0.1),~Running~time=0.0021,\\
 x_{14,New Q}&=&(-0.1,0.1),~Running~time=0.0025,\\
 x_{14,Rand}&=&Error,\\
 x_{14,Iner}&=&\infty,\\
 x_{14,Back}&=&(-0.1,0.1),~Running~time=0.938.
 \end{eqnarray*}
 
 {\bf Example 15:} Here we test for a homogeneous function $f_{15}$ of degree $2$ in $3$ variables, whose Hessian matrix is:  
 
 \[ \left( \begin{array}{ccc}
-23&-61&40\\
-61&-39.5&155\\
40&155&-50\\
\end{array}\right) \]

 The Hessian matrix is not invertible, it has one positive and one negative eigenvalue. Hence, the critical points of this function are all generalised saddle points, but they are degenerate. A good method should diverge. Initial point $(0.00001188, 0.00002188, 0.00003188)$, which is randomly chosen. Points to be used in Table \ref{tab:Tasks}:
 \begin{eqnarray*}
x_{14,ACR }&=&(-75.032,-150.111,149.953),~Running~time=0.034,\\
 x_{15,BFGS}&=&\infty,\\
 x_{15,New Q}&=&\infty,\\
 x_{14,Rand}&=&Error,\\
 x_{15,Iner}&=&\infty,\\
 x_{15,Back}&=&\infty.
 \end{eqnarray*}
 
 {\bf Example 16:} We test for the Ackley function $f_{16}(x_1,\ldots ,x_D)=$ $-20*exp[-0.2*\sqrt{0.5\sum _{i=1}^Dx_i^2}]-exp$ $[0.5*\sum _{i=1}^D\cos (2\pi x_i)]$ $+e+20$, see \cite{jamil-yang, test}. The global minimum is at $(x_1,\ldots ,x_D)=(0,\ldots ,0)$. We choose $D=3$. Depending on the randomly chosen initial point $x_0$, there are 2 typical behaviours.

 \underline{Case 1:} The initial point is $(-2.94501548, -1.81794532, -2.44883475)$ (randomly chosen). Points to be used in Table \ref{tab:Tasks}:
 \begin{eqnarray*}
 x_{16,1,ACR }&=&(-2.963, -1.975, -7e-7),~f_{16}(x_{16,1,ACR})=6.777,~Running~time=0.035,\\
 x_{16,1,BFGS}&=&(-2.970, -1.980, -1.980),~f_{16}(x_{16,1,BFGS})=7.546,~Running~time=0.115,\\
 x_{16,1,New Q}&=&(-1.974, -1.974,-1.974),~f_{16}(x_{16,1,NewQ})=6.559,~Running~time=0.348,\\
 x_{16,1,Rand}&=&(-2.945, -1.817,-2.448),~f_{16}(x_{16,1,Rand})=8.753,~Running~time=0.571,\\
 x_{16,1,Iner}&=&\infty,\\
 x_{16,1,Back}&=&(-2.970, -1.980, -1.980),~f_{16}(x_{16,1,Back})=7.546,~Running~time=307.877.
 \end{eqnarray*}
 
 \underline{Case 2:} The initial point is $(0.01,0.02,-0.07)$ (closer to the global minimum). In this case we see that the behaviour is much better.  Points to be used in Table \ref{tab:Tasks}:
 \begin{eqnarray*}
 x_{16,2,ACR }&=&(-0.007,-0.015,0.041),~f_{16}(x_{16,2,ACR})=0.137,~Running~time=59.160,\\
 x_{16,2,BFGS}&=&(-8e-14,-1e-13,8e-14),~Running~time=0.282,\\
 x_{16,2,New Q}&=&(1e-15,-1e-16,1.946),~Running~time=0.598,\\
 x_{16,2,Rand}&=&(-1e-15,-4e-15,-0.617),~Running ~time=0.801,\\
 x_{16,2,Iner}&=&\infty,\\
 x_{16,2,Back}&=&(-1e-12,1e-12,-1e-12),~Running~time=1.546.
 \end{eqnarray*}
 
 {\bf Example 17: } We test for the Rastrigin function $f_{17}(x_1,\ldots ,x_D)=$ $A*D$ $+\sum _{i=1}^D(x_i^2-A\cos (2\pi x_i))$, see \cite{test}. The global minimum is at $(x_1,\ldots ,x_D)$ $=(0,\ldots ,0)$. We choose $D=4$, $A=10$. Depending on the randomly chosen initial point $x_0$, there are 2 typical behaviours.

 \underline{Case 1:} The initial point $(-4.66266579, -2.69585675, -3.08589085, -2.25482451)$ (randomly chosen). Points to be used in Table \ref{tab:Tasks}:
 \begin{eqnarray*}
 x_{17,1,ACR }&=&(-4.974, -2.984, -2.984, -1.989),~f_{17}(x_{17,1,ACR})=46.762,~Running~time=0.032,\\
 x_{17,1,BFGS}&=&(-4.974, -2.984, -2.984, -1.989),~f_{17}(x_{17,1,BFGS})=46.762,~Running~time=0.135,\\
 x_{17,1,New Q}&=&(-4.974,-2.984,-2.984,3.979),~f_{17}(x_{17,1,NewQ})=58.702,~Running~time=0.241,\\
 x_{17,1,Rand}&=&(-4.523, -1.990,-2.984 ,-13.926),~f_{17}(x_{17,1,Rand})=248.282,~Running~time=0.950,\\
 x_{17,1,Iner}&=&\infty,\\
 x_{17,1,Back}&=&(-4.974, -2.984,-2.984,-1.989),~f_{17}(x_{17,1,Back})=46.762,~Running~time=278.342.
 \end{eqnarray*}
 
 \underline{Case 2:} The initial point is $(0.01, 0.5,-0.07, -0.3)$ (closer to the global minimum). We see that the convergence is better. Points to be used in Table \ref{tab:Tasks}:
 \begin{eqnarray*}
 x_{17,2,ACR }&=&(-1e-8,1.989,4e-8,1e-7),~f_{17}(x_{17,2,ACR})=3.979,~Running~time=0.025,\\
 x_{17,2,BFGS}&=&(3e-12,-1e-11,-6e-12,2e-11),~Running~time=0.438,\\
 x_{17,2,New Q}&=&(-5e-18,2e-15,-1e-16,-1e-16),~Running~time=0.421,\\
 x_{17,2,Rand}&=&(1e-15,0.502 ,-4e-15,-0.502),~f_{17}(x_{17,2,Rand})=40.502,~Running~time=1.223,\\
 x_{17,2,Iner}&=&\infty,\\
 x_{17,2,Back}&=&(-1e-10,3e-10,-8e-10,-2e-10),~Running~time=295.757.
 \end{eqnarray*}

 {\bf Example 18:} Rosenbroch's function in higher dimension \cite{cute, test}: $$f_{18}(x_1,\ldots ,x_D) = \sum _{i=1}^{D-1}f_7(x_{i},x_{i+1}),$$ where $f_7(.,.)$ is the Rosenbrock's function in Example 7. It has a global minimum $(1,1,\ldots ,1)$.  We check for example in the case the dimension is $D=7$. The initial point is $(-2.95108579,$ $-0.76552935,$  $1.83618076,$ $-0.6336922,$   $1.33774087,$ $-0.93499206,$ $3.51430143)$, which is randomly chosen. Points to be used in Table \ref{tab:Tasks}:
 \begin{eqnarray*}
 x_{18,ACR }&=&(-0.992,0.995,0.996, 0.996, 0.993, 0.987, 0.976),~f_{18}(x_{18,ACR})=3.985,~Running~time=0.251,\\
 x_{18,BFGS}&=&(-0.992,0.995,0.996, 0.996, 0.993, 0.987, 0.976),~f_{18}(x_{18,BFGS})=3.985,~Running~time=0.727,\\
 x_{18,New Q}&=&(-0.992,0.995,0.996, 0.996, 0.993, 0.987, 0.976),~f_{18}(x_{18,New Q})=3.985,~Running~time=2.029,\\
 x_{18,Rand}&=&(-0.992,0.995,0.996, 0.996, 0.993, 0.987, 0.976),~f_{18}(x_{18,Rand})=3.985,~Running~time=4.005,\\
 x_{18,Iner}&=&\infty,\\
 x_{18,Back}&=&(1,1,1,1,1,1,1),~Running~time=437.373.
 \end{eqnarray*}  
 
{\bf Example 19:} Beale's function $f(x,y)=$ $(1.5-x+xy)^2$ $+(2.25-x-xy^2)^2$ $+(2.625-x-xy^3)^2$, see \cite{jamil-yang, test}. The global minimum is $(x,y)$ $=(3,0.5)$. The initial point is $(-0.52012358,$ $-1.28227229)$, which is randomly chosen. Points to be used in Table \ref{tab:Tasks}:
 \begin{eqnarray*}
 x_{19,ACR }&=&(2.999,0.4999),~Running~time=0.039,\\
 x_{19,BFGS}&=&(3,0.5),~Running~time=0.116,\\
 x_{19,New Q}&=&(1e-21,  -1e+7),~f_{19}(x_{19,NewQ})=7.3125,~Running~time=1.452,\\
 x_{19,Rand}&=&(-1e-13,  1),~f_{19}(x_{19,Rand})=14.203,~Running~time=0.499,\\
 x_{19,Iner}&=&Error,\\
 x_{19,Back}&=&(3,0.5),~Running~time=2.468.
 \end{eqnarray*}  

{\bf Example 20:} Bukin function $\#6$: $f_{20}(x,y)=100\sqrt{|y-0.01x^2|}+0.01|x+10|$, see \cite{jamil-yang, test}.  The global minimum is $(x,y)=(-10,1)$. Depending on the randomly chosen initial point $x_0$, there are 2 typical behaviours.

\underline{Case 1:} The initial point is $(4.38848192 , -3.47943683)$ (randomly chosen). Points to be used in Table \ref{tab:Tasks}:
 \begin{eqnarray*}
 x_{20,1,ACR }&=&(2.653,-1.940),~f_{20}(x_{20,1,ACR})=141.929,~Running~time=0.046,\\
 x_{20,1,BFGS}&=&(4.067, 0.166),~f_{20}( x_{20,1,BFGS})=3.029,~Running~time=0.180,\\
 x_{20,1,New Q}&=&(-0.149, -3.671),~f_{20}(x_{20,1,New Q})=191.723,~Running~time=93.671,\\
 x_{20,1,Rand}&=&Error,\\
 x_{20,1,Iner}&=&\infty,\\
 x_{20,1,Back}&=&(3.994, 0.160),~ f_{20}(x_{20,1,Back})=2.413,~Running~time=102.196.
 \end{eqnarray*}

\underline{Case 2:} The initial point is $(-9.7,0.7)$ (closer to the global minimum). Points to be used in Table \ref{tab:Tasks}:
 \begin{eqnarray*}
 x_{20,2,ACR }&=&(-9.600,1.001),~f_{20}(x_{20,2,ACR})=28.247,~Running~time=26.470,\\
 x_{20,2,BFGS}&=&(-9.653,  0.932),~f_{20}(x_{20,2,BFGS})=1.038,~Running~time=0.177,\\
x_{20,2,New Q}&=&(-0.514, -0.238),~f_{20}(x_{20,2,New Q})=49.176,~Running~time=88.328,\\
 x_{20,2,Rand}&=&Error,\\
 x_{20,2,Iner}&=&\infty,\\
 x_{20,2,Back}&=&(-9.679,  0.936),~f_{20}(x_{20,2,Back})=0.003,~Running~time=101.469.
 \end{eqnarray*}  

We observe that this function is not even $C^1$, and hence does not satisfy the assumptions to apply New Q-Newton's method. 

{\bf Example 21:} L\'evi function $\#13$: $f_{21}(x,y)=\sin ^2(3\pi x)+(x-1)^2*(1+\sin ^2(3\pi y))+(y-1)^2(1+\sin ^2(2\pi y))$, see \cite{test}. The global minimum is at $(1,1)$. Depending on the randomly chosen initial point $x_0$, there are 2 typical behaviours.

\underline{Case 1:} The initial point is $(-3.52914182,  1.36683019)$ (randomly chosen). Points to be used in Table \ref{tab:Tasks}:
 \begin{eqnarray*}
 x_{21,1,ACR }&=&(3.306, 0.002),~f_{21}(x_{21,1,ACR})=6.380,~Running~time=0.028,\\
 x_{21,1,BFGS}&=&(-3.273,  0.334),~f_{21}(x_{21,1,BFGS})=19.322,~Running~time=0.077,\\
 x_{21,1,New Q}&=&(-3.273,   1.333),~f_{21}(x_{21,1,New Q})=18.742,~Running~time=0.128,\\
 x_{21,1,Rand}&=&(-3.570,  1.333),~f_{21}(x_{21,1,Rand})=21.703,~Running~time=0.284,\\
 x_{21,1,Iner}&=&\infty,\\
 x_{21,1,Back}&=&(1,1),~Running~time=7.554.
 \end{eqnarray*}

\underline{Case 2:} The initial point is $(0.95,1.15)$ (closer to the global minimum). We see that the convergence is better. Points to be used in Table \ref{tab:Tasks}:
 \begin{eqnarray*}
 x_{21,2,ACR }&=&(1,1),~Running~time=0.021,\\
 x_{21,2,BFGS}&=&(1,1),~Running~time=0.063,\\
 x_{21,2,New Q}&=&(1, 1),~Running~time=0.069,\\
x_{21,2,Rand}&=&(1, 1),~Running~time=0.295,\\
 x_{21,2,Iner}&=&\infty,\\
 x_{21,2,Back}&=&(1, 1),~Running~time=7.709.
 \end{eqnarray*}

{\bf Example 22:} Eggholder function $f_{22}(x,y)=-(y+47)*\sin \sqrt{|(x/2)+(y+47)|}-x*\sin \sqrt{|x-(y+47)|}$, see \cite{jamil-yang, test}. The global minimum is $(512,404.2319)$, with function value $-959.6407$. Depending on the randomly chosen initial point $x_0$, there are 2 typical behaviours. 

\underline{Case 1:} The initial point is $(224.63208339, -188.85104265)$ (randomly chosen). Points to be used in Table \ref{tab:Tasks}:
 \begin{eqnarray*}
 x_{22,1,ACR }&=&(263.344,-200.698),~f_{22,1,ACR}=-417.014,~Running~time=0.071,\\
 x_{22,1,BFGS}&=&(267.375, -202.898),~f_{22}(x_{22,1,BFGS})=-420.139,~Running~time=0.052,\\
 x_{22,1,New Q}&=&(399.558,-367.691),~f_{22}(x_{22,1,NewQ})=-716.671,~Running~time=0.086,\\
 x_{22,1,Rand}&=&(356.294, -247.954),~f_{22}(x_{22,1,Rand })=155.394,~Running~time=0.169,\\
 x_{22,1,Iner}&=&\infty,\\
 x_{22,1,Back}&=&(267.375, -202.898),~f_{22}(x_{22,1,Back})=-420.139,~Running~time=109.009.
 \end{eqnarray*}

\underline{Case 2:} The initial point is $(500, 450)$ (closer to the global minimum). Points to be used in Table \ref{tab:Tasks}:
 \begin{eqnarray*}
 x_{22,2,ACR }&=&(498.166,448.486),~f_{22}(x_{22,2,ACR})=-910.643,~Running~time=0.044,\\
 x_{22,2,BFGS}&=&(482.353, 432.878),~f_{22}(x_{22,2,BFGS })=-956.918,~Running~time=0.148,\\
 x_{22,2,New Q}&=&(482.353, 432.878),~f_{22}(x_{22,2,New Q })=-956.918,~Running~time=0.058,\\
 x_{22,2,Rand}&=&(482.353, 432.878),~f_{22}(x_{22,2,Rand })=-956.918,~Running~time=0.336,\\
 x_{22,2,Iner}&=&\infty,\\
 x_{22,2,Back}&=&(482.353, 432.879),~f_{22}(x_{22,2,Back})=-956.918,~Running~time=123.976.
 \end{eqnarray*}

This function is also not even $C^1$. 

{\bf Example 23:} McCormick function $f_{23}(x,y)=\sin (x+y)+(x-y)^2-1.5*x+2.5*y+1$, see \cite{cute, test}. The global minimum is $(-0.54719, -1.54719)$, with function value $-1.9133$. The initial point is $(-2.28637302,  1.52532269)$, which is randomly chosen. Points to be used in Table \ref{tab:Tasks}:
 \begin{eqnarray*}
 x_{23,ACR}&=&(-0.54719454, -1.54719754),~Running~time=0.036,\\
 x_{23,BFGS}&=&(-0.54719755, -1.54719755),~Running~time=0.036,\\
 x_{23,New Q}&=&(-0.54719755, -1.54719755),~Running~time=0.066,\\
 x_{23,Rand}&=&(-1.594, -2.594),~f_{23}(x_{23,Rand})=-1.228,~Running~time=2.047,\\
 x_{23,Iner}&=&\infty,\\
 x_{23,Back}&=&(-0.54719754, -1.54719754),~Running~time=105.411.
 \end{eqnarray*}

 {\bf Example 24:} Schaffer function $\#2$: $f_{24}(x,y)=0.5+(\sin ^2(x^2-y^2)-0.5)/(1+0.001(x^2+y^2))^2$, see \cite{jamil-yang, test}. The global minimum is $(0,0)$, with the function value $0$. Depending on the randomly chosen initial point $x_0$, there are 2 typical behaviours. 
 
 \underline{Case 1:} The initial point is $(-57.32135254, -17.85920667)$ (randomly chosen). Points to be used in Table \ref{tab:Tasks}:
 \begin{eqnarray*}
 x_{24,1,ACR }&=&(0.798,0.798),~f_{24}(x_{24,1,ACR})=0.0012,~Running~time=9.288,\\
 x_{24,1,BFGS}&=&(-56.237, -18.137),~f_{24}(x_{24,1,BFGS})=0.475,~Running~time=0.170,\\
 x_{24,1,New Q}&=&\infty,\\
 x_{24,1,Rand}&=&\infty,\\
 x_{24,1,Iner}&=&\infty,\\
 x_{24,1,Back}&=&(-57.296, -17.812),~f_{24}(x_{24,1,Back})=0.476,~Running~time=108.445.
 \end{eqnarray*}  
 
 \underline{Case 2:} The initial point is $(0.5,-0.7)$ (closer to the global minimum). Points to be used in Table \ref{tab:Tasks}:
 \begin{eqnarray*}
 x_{24,2,ACR}&=&(-7.073,7.074),~f_{24}(x_{24,1,ACR})=0.086,~Running~time=30.134,\\
 x_{24,2,BFGS}&=&(9e-9,3e-9),~Running~time=0.142,\\
 x_{24,2,New Q}&=&(-1e-12,-4e-12),~Running~time=0.099,\\
 x_{24,2,Rand}&=&(1e-8,2.6e-10),~Running~time=0.287,\\
 x_{24,2,Iner}&=&\infty,\\
 x_{24,2,Back}&=&(5e-8,-5e-8),~Running~time=106.102.
 \end{eqnarray*}

 {\bf Example 25:} Schaffer function $\#4$: $f_{25}(x,y)=0.5 + [\cos ^2(\sin (|x^2-y^2|))-0.5]/[1+0.001(x^2+y^2)]^2$, see \cite{jamil-yang, test}. The global minima are $(0,\pm 1.25313)$, with function value $0.292579$. Depending on the randomly chosen initial point $x_0$, there are 2 typical behaviours. 
 
 \underline{Case 1:} The initial point is $(86.64664502, 23.63197178)$ (randomly chosen). Points to be used in Table \ref{tab:Tasks}:
 \begin{eqnarray*}
 x_{25,1,ACR }&=&(83.014,1.860),~f_{25}(x_{25,1,ACR})=0.496,~Running~time=0.232,\\
 x_{25,1,BFGS}&=&(86.646, 23.631),~f_{25}(x_{25,1,BFGS})=0.506,~Running~time=0.080,\\
 x_{25,1,New Q}&=&\infty,\\
 x_{25,1,Rand}&=&\infty,\\
 x_{25,1,Iner}&=&\infty,\\
 x_{25,1,Back}&=&(86.710,23.634),~f_{25}(x_{25,1,Back})=0.497,~Running~time=117.947.
 \end{eqnarray*}  
 
 It is interesting to note that the function values for the methods New Q-Newton's method, Random Newton's method and Inertial Newton's method are about $0.5$, {\bf better} than that of BFGS, even though they diverse.

 \underline{Case 2:} The initial point is $(0.5,1.25313+0.8)$ (closer to a global minimum). Points to be used in Table \ref{tab:Tasks}:
 \begin{eqnarray*}
 x_{25,2,ACR }&=&(-0.005, -2.170),~f_{25}(x_{25,2,N})=0.293,~Running~time=26.289,\\
 x_{25,2,BFGS}&=&(-9e-12,  2.170),~f_{25}(x_{25,2,BFGS})=0.293,~Running~time=0.141,\\
 x_{25,2,New Q}&=&(-6e-14,  2.170),~f_{25}(x_{25,2,NewQ})=0.293,~Running~time=0.166,\\
 x_{25,2,Rand}&=&(9e-12, 2.802),~f_{25}(x_{25,2,Rand})=0.295,~Running~time=0.417,\\
 x_{25,2,Iner}&=&\infty,\\
 x_{25,2,Back}&=&(3.8e-5, 2.170),~f_{25}(x_{25,2,Back})=0.293,~Running~time=111.854.
 \end{eqnarray*}

 {\bf Example 26:} Styblinski-Tang function $f_{26}(x_1,\ldots ,x_D)=$ $\sum _{i=1}^D(x_i^4-16x_i^2+5x_i)/2$, see \cite{jamil-yang}. The global minimum is at $(x_1,\ldots ,x_D)$ $=$ $(-2.903534$, $\ldots $, $-2.903534)$. The function value is in the interval $(-39.16617D,-39.16616D)$. We choose $D=2$. The minimum value of the function is then about $-78.33233140754284$. Depending on the randomly chosen initial point $x_0$, there are 2 typical behaviours. 
 
 \underline{Case 1:} The initial point is $(1.02183524, 0.13979978)$ (randomly chosen).  Points to be used in Table \ref{tab:Tasks}:
 \begin{eqnarray*}
 x_{26,1,ACR }&=&(2.746,-2.903),~f_{26}(x_{26,1,ACR})=-64.195,~Running~time=0.024,\\
 x_{26,1,BFGS}&=&(2.7466,-2.903),~f_{26}(x_{26,1,BFGS})=-64.195,~Running~time=0.092,\\
 x_{26,1,New Q}&=&(2.746,-2.903),~f_{26}(x_{26,1,NewQ})=-64.195,~Running~time=0.144,\\
 x_{26,1,Rand}&=&(0.156,0.156),~f_{26}(x_{26,1,Rand})=0.391,~Running~time=0.298,\\
 x_{26,1,Iner}&=&\infty,\\
 x_{26,1,Back}&=&(2.746,-2.903),~f_{26}(x_{26,1,Back})=-64.195, ~Running~time=133.348.
 \end{eqnarray*}

\underline{Case 2:} The initial point is $(-2.903534+0.3, -2.903534-0.8)$ (closer to the global minimum). Points to be used in Table \ref{tab:Tasks}:
 \begin{eqnarray*}
 x_{26,2,ACR}&=&(-2.90353478, -2.90353428),~Running~time=0.029,\\
 x_{26,2,BFGS}&=&(-2.90353403, -2.90353403),~Running~time=0.053,\\
 x_{26,2,New Q}&=&(-2.90353403, -2.90353403),~Running~time=0.085,\\
 x_{26,2,Rand}&=&(-2.90353403, -2.90353403),~Running~time=0.443,\\
 x_{26,1,Iner}&=&\infty,\\
 x_{26,1,Back}&=&(-2.903534, -2.90353403),~Running~time=134.293.
 \end{eqnarray*}

\end{document}